\documentclass{amsart}

\usepackage{amsthm, amssymb,amsmath}
\usepackage{mathscinet}
\usepackage[utf8]{inputenc}
\usepackage{hyperref}
\usepackage{graphicx}
\usepackage{subcaption}
\usepackage{dsfont}
\usepackage{rotating}

\newtheorem{theorem}{Theorem}[section]
\newtheorem{remark}[theorem]{Remark}
\newtheorem{proposition}[theorem]{Proposition}
\newtheorem{lemma}[theorem]{Lemma}
\newtheorem{corollary}[theorem]{Corollary}
\newtheorem*{definition*}{Definition}
\newtheorem*{theorem*}{Theorem}

\newcommand{\E}[1]{\mathbb{E}\left(#1\right)}

\renewcommand{\P}{\mathbb{P}}

\newcommand{\R}{\mathbb{R}}
\newcommand{\Z}{\mathbb{Z}}
\newcommand{\N}{\mathbb{N}}
\renewcommand{\H}{\mathbb{H}}
\renewcommand{\S}{\mathbb{S}}
\newcommand{\F}{\mathcal{F}}
\newcommand{\X}{\mathcal{X}}

\newcommand{\e}{\varepsilon}

\newcommand{\vf}{\varphi}

\newcommand{\G}{\Gamma}
\newcommand{\g}{\gamma}

\newcommand{\LL}{\mathcal L}
\newcommand{\PP}{\mathcal P}
\newcommand{\MM}{\mathcal M}

\newcommand{\la}{\lambda}
\newcommand{\La}{\Lambda}
\newcommand{\de}{\delta}

\newcommand{\s}{\sigma}
\newcommand{\Si}{\Sigma}
\newcommand{\x}{\times}
\newcommand{\om}{\omega}
\newcommand{\Om}{\Omega}

\newcommand{\GG}{\mathcal G}
\newcommand{\ov}{\overline}
\newcommand{\un}{\underline}

\DeclareMathOperator{\GL}{GL}

\newcommand{\dist}{\operatorname{dist}}

\newcommand{\planes}{\mathcal{P}}
\newcommand{\lines}{\mathcal{L}}

\renewcommand{\det}{\text{det}}

\begin{document}

\title{Exact dimension of Furstenberg measures}
\author{Fran\c cois Ledrappier and Pablo Lessa}
\address{Fran\c cois Ledrappier, Universit\'e de Paris et Sorbonne Universit\'e, CNRS, LPSM, Bo\^{i}te Courrier 158, 4, Place Jussieu, 75252 PARIS cedex
05, France,} \email{fledrapp@nd.edu}
\address{ Pablo Lessa,  IMERL, Facultad de Ingeniería, Julio Herrera y Reissig 565,  11300 Montevideo, Uruguay}  \email{plessa@fing.edu.uy}

\thanks{FL was partially supported by IFUM; PL thanks CSIC research project 389}

\maketitle
\begin{abstract} For a probability measure \( \mu \) on \( SL_d(\mathbb R) \), we consider the Furstenberg stationary measure \( \nu \) on the space of flags. Under general non-degeneracy conditions, if \( \mu \) is discrete and if \( \sum_g \log \|g\| \, \mu (g) < + \infty \), then the measure \( \nu \) is exact-dimensional.
\end{abstract}

\section{Introduction}
\subsection{Main results}
Let $\mu $ be a probability measure on the group $ SL_d(\R)$ of $d\x d$ real matrices with determinant 1.  Let $\F$ be the space of complete flags in $\R^d$, \[ f \in \F \iff f = U_0 \subset U_1 \subset \ldots \subset U_{d-1} \subset U_d,\] where \( U_i\) is a vector space of dimension \(i\) in $\R^d, U_0 = \{0\}, U_d = \R^d.$  $SL_d(\R) $ acts naturally on $\F$. A probability measure \(\nu \) on \( \F\) is called {\it {stationary}} if it satisfies \[\int g_\ast \nu \, d\mu (g) = \nu. \]
By compactness, there always exist stationary measures. Understanding stationary  measures is central to many studies of linear groups and applications (see \cite{benoist-quint16} for a recent survey).

Let \( G\) be a group, \(\mu \) a probability on \(G\), \(X\)  a compact \(G\)-space and \( \nu \) a stationary probability measure. We define the {\it{Furstenberg entropy }}  as the nonnegative number \( h(X, \mu, \nu ) \)  given by \[ h(X, \mu, \nu ) := \int _G \int _{X} \log \frac{dg_\ast\nu}{d\nu} (x)\,  \frac{dg_\ast\nu}{d\nu } (x) \, d\nu  (x) d\mu (g) ,\]
 with the  convention that \( 0 \log 0 = 0\) and  that the entropy  is \( +\infty \) if  the measure \( g_\ast \nu \) is not absolutely continuous with respect to the measure \(\nu \) for a  set of positive \(\mu\)-measure of elements \(g \in G.\)
 
 Let \( G_0\) be the subgroup of \(SL_d(\R)\) generated by the support of \( \mu \). Assume that \(\int _{G_0}\log \|g\| \, d\mu (g) <+\infty \) and \( \nu \) is extremal among stationary measures. Then,  there are $d$  {\it{Lyapunov exponents}}  \[ \chi _1 \geq \chi _2 \geq \ldots \geq \chi _d, \; {\textrm {with }} \chi_1 + \ldots + \chi _d = 0 ,\] such that
for \( \nu \)-a.e. $f \in \F, f= \{ \{0\} \subset U_1(f) \subset \ldots \subset U_{d-1}(f) \subset \R^d \},$ any  \(j,  j = 1\ldots, d-1, \) \( \mu ^{\otimes _\N }\) and almost every sequence \( g_0, g_1, \ldots, \)  \begin{equation}\label{exponents} \sum _{i\leq j} \chi _i  \; =\;\lim\limits _{n \to \infty } \frac{1}{n}  \log |\det_{U_j (f)} (g_{n-1}\ldots g_1g_0) |  ,\end{equation}
where, for any subspace $U$ in $\R^d$, $|\det_U (g)|$ is the Jacobian of the linear mapping from $U$ to $gU$, both endowed with the Euclidean metric.

The following inequality is very general (and essentially due to Furstenberg (\cite{furstenberg1963}))
\begin{theorem}\label{ent/exp1} Let \( \mu \) be a probability  on \( SL_d(\mathbb R) \) such that \(\int _{SL_d(\mathbb R) } \log \|g\| \, d\mu (g) <+\infty. \) Then there exists a stationary measure \( \nu \) on \(\F \) such that 
\begin{equation}\label{ent/exp} h(\F, \mu, \nu )  \; \leq \; \sum _{i,j: i<j } \chi_i - \chi _j .\end{equation}
If there is equality in (\ref{ent/exp}), then the measure \( \nu  \) is exact-dimensional with  dimension \( d(d-1)/2\). \end{theorem}
The inequality (\ref{ent/exp}) is proven in Section \ref{section:ent/exp}. See the discussion after theorem \ref{exactfibers} for the equality case. 
 \begin{remark}\label{abscont} It follows from theorem \ref{ent/exp1}  and   its proof, that if \(\int _{SL_d(\R)} \log \|g\| \, d\mu (g) <+\infty \), then  there exists  a stationary measure \( \nu \) with finite entropy. In particular,  for \( \mu\)-a.e. \(g \in G_0\),  the measure \( g_\ast \nu \) is absolutely continuous with respect to the measure \(\nu \) (see Corollary \ref{finiteentropy}). \end{remark}

Let \( \MM(d)\) be the set of probability measures on \(SL_d(\R) \) with \(\int \log \|g\| \, d\mu (g) <+\infty \) such that the stationary measure on \( \F\) is unique and the Lyapunov exponents are pairwise distinct. For example, if the group \(G_0\)  is Zariski dense in \(SL(d, \R) \) and \(\int  \log \|g\| \, d\mu (g) <+\infty \), then \( \mu \in \MM(d) \) (see \cite{guivarch-raugi} and \cite{goldsheid-margulis}).

Let $(X,\rho) $ be a metric space, \( \nu \) a  measure on \(X.\) The {\it {lower dimension }} \( \un \de \) and the {\it {upper dimension}} \( \ov \de \) of \((X, \rho, \nu )\) are defined by 
\[ \un \de \, = \, \underset {\nu}{{\textrm {ess.inf}}} \, \liminf _{r\to 0} \frac{\log \nu (B(x,r))}{\log r} , \quad  \ov \de \, = \, \underset {\nu}{{\textrm {ess.sup}}} \,\limsup_{r\to 0} \frac{\log \nu (B(x,r))}{\log r}.\]
A measure $\nu $ on $X$ is called {\it {exact-dimensional with dimension }}$\delta$ if \(\un \de = \ov \de = \de  .\)

If the space \((X, \rho ) \) is bilipschitz equivalent to an Euclidean \( \R ^n, n\geq 1,\) and  $\nu $ is exact-dimensional of dimension $\delta $, then $\delta $ is the smallest Hausdorff dimension of sets of positive $\nu $-measure (see e.g. \cite{young82}, Prop 2.1).  Our main result is 
\begin{theorem}\label{main} Let $\mu \in \MM (d)$ be a discrete probability measure.  Endow the space \(\F\) of  flags with the natural Riemannian distance invariant under the action of \( SO(d).\) 
Then the unique stationary probability measure $\nu $ on the space \(\F\)   is exact-dimensional. \end{theorem}
The dimension in Theorem \ref{main} is given by a formula involving exponents and some partial entropies (see (\ref{etvoila})).  This implies an a priori bound on the dimension that we describe now.
Denote \( \{ 0 < \lambda _1 \leq \lambda _2 \leq \ldots \leq \lambda _{d(d-1)/2} \} \) the differences of exponents \( \chi _i - \chi _j \) for all \((i,j), i<j.\)  Assume \( \mu \) is discrete and belongs to \( \MM(d) \). Let \( \nu \) be the unique stationary measure. We   define the continuous, piecewise affine  function \( D_{ \mu } \) on the interval \( [ 0,d(d-1)/2 ]\)
 as:
\[ D_{ \mu } (0) := h(\F, \mu ,\nu ) \quad {\textrm {and}}\quad  D'_{ \mu } (s) = - \lambda _k\; {\textrm {for}} \; s \in (k-1,k), \; k = 1, \ldots, d(d-1)/2. \]
Observe that by theorem \ref{ent/exp1}, \( D_\mu (d(d-1)/2) \leq 0 .\)
Following Kaplan-Yorke (\cite{kaplanyorke}) and Douady-Oesterl\'e (\cite{douadyoesterle}), the {\it {Lyapunov dimension}} \( \dim_{{\textrm{LY}}}(\F, \mu  )\) 
 is the number  such that \( D_\mu ( \dim_{{\textrm{LY}}}(\F, \mu )) = 0.\) 
\begin{theorem}\label{LyaDim1} Let \( \mu \in \MM (d)\) be discrete. Then, the exact dimension \(\de \) of \( \nu  \) satisfies
\begin{equation}\label{LyaDim} \de  \; \leq \; \dim_{{\textrm{LY}}}(\F, \mu).\end{equation}
\end{theorem}
We discuss the proof of theorem \ref{LyaDim1} in Section \ref{Lyapunov}.

We can prove equality in relation (\ref{LyaDim}) in some examples: in dimension 2, for \( \mu \in \MM(2), \) we always have \( \de (\nu ) = \frac{h(\F, \mu, \nu)}{\chi _1 -\chi _2} = \dim_{{\textrm{LY}}}(\F, \mu)\) (\cite{ledrappier}, where the formula is proven with a less precise notion of dimension);
in section \ref{example}, we discuss the terms and the proof of the following 
\begin{theorem}\label{LyaDim3}  Let \(\G \) be a cocompact group of isometries of \( \H^2 \), \( \rho \) a Hitchin representation of \( \G\) in \(PSL_d(\R)\) and  \( \mu\) be an  adapted probability measure on \( \G\) such that 
\( \sum _g   | g| \, \mu (g)  < + \infty ,\) where \( |\cdot | \) is some word metric on \(\G\). Consider the random walk on \(PSL_d(\R)\) directed by the probability \( \rho_\ast (\mu) \) and \( \nu \) the stationary measure on the space \( \F\) of flags. Then, \(\nu \) is exact-dimensional  and
 \( \de (\nu ) \; = \; \dim_{{\textrm{LY}}}(\F, \mu) .\)  \end{theorem}

If \( \mu  \in \MM(d)\), we can also consider  a  partition \(Q= \{ 0< q_1 < \dots  < q_{\ell -1} < q_\ell =  d\} \)  of $\{0,1, \ldots , d\}$ into  intervals, the group \( SL_d(\mathbb R) \) acts naturally on the space  $\F_Q$  of increasing sequences of vector subspaces of $\R^d$, \[   \{0\} = U_0 \subset U_1 \subset \ldots \subset U_{\ell-1} \subset U_\ell = \R^d, \] 
with  $ \dim U_i = q_i $ for $i =  1, \ldots, \ell $. The group $G_0$ acts naturally on  \( \F_Q\) and there is a unique stationary probability measure \( \nu _Q\) for this action. Theorems \ref{ent/exp1}, \ref{main} and \ref{LyaDim1} will be the particular case (\( Q = Q_1 := \{ 0<1<2<\ldots < d \} \)) of the corresponding results stated and proven for the action of \( \F_Q\), for any partition \(Q\) (see respectively theorem \ref{ent/expQ1}, corollary \ref{mainQ} and proposition \ref{LyaDimQ}). 
 
 \

\subsection{Related results}\label{history}
 There are many results (and still open questions) related to the  regularity of the Furstenberg measure in dimension 2: the space $\F$ is the space of lines in $\R^2$ and  theorem \ref{main} holds for a general measure \( \mu \in \MM\) (Hochman and Solomyak \cite{hochman-solomyak2017}). Theorem \ref{LyaDim1} follows with equality in the formula (\ref{LyaDim}). The main result of \cite{hochman-solomyak2017} is stronger: it concerns a formula analogous to (\ref{LyaDim}) but with a larger  definition of the Lyapunov dimension and is related to the well-known problem of finding conditions under which the Furstenberg measure is absolutely continuous. 
 For any lattice \( \G \in SL_2(\R),\) by discretizing the Brownian motion on the symmetric space, Furstenberg (\cite{furstenberg1971}) constructed a probability measure on \( \G\) which belongs to \( \MM (2) \) and such that the stationary measure is Lebesgue. On the contrary, the stationary \(\nu \) is singular if \( \mu \)  has  finite support on \( SL_2(\Z) \) (Guivarc'h-Le Jan \cite{guivarchlejan}).
 It is not known whether the same is true for measures with finite support on uniform lattices. There are examples of measures with finite support and an absolutely continuous stationary measure (B\'ar\'any-Pollicott-Simon \cite{BPS} and Bourgain \cite{bourgain12}), but the group generated by the support is dense.

In all dimensions, under exponential moment and proximality conditions, Guivarc'h (\cite{guivarch1990}) showed that the unique stationary measure \( \nu \) has  uniformly  positive dimension: there is \(C, \de >0 \) such that for all \( f \in \F \), all \( \e >0,\)  \( \nu (B (f, \e ) )\; \leq C \e^\de .\)\\
The stationary measure might be absolutely continuous: Furstenberg's discretization works in all dimensions and Benoist-Quint (\cite{benoist-quint18}) found  finitely supported measures with absolutely continuous stationary measures. In the other direction, there exist measures supported on a discrete Zariski dense subgroup \( \G \) with  arbitrarily small \( \dim _{LY} \) (Kaimanovich-Le Prince \cite{kaimanovichLePrince}).

Theorem \ref{main} has been recently proven by A. Rapaport (\cite{rapaport}) for the stationary measure on the space \( \F_{0<1<d} = \R\P^{d-1} \) of directions in \( \R^d\) and on the space \( \F_{0< d-1<d} = \GG_{d-1} \) of \( (d-1)\)-dimensional hyperplanes in the case when the measure $\mu $ is finitely supported.  Theorem \ref{LyaDim1} follows with a  notion of \( \dim _{LY} \) adapted to the projective space (see Corollary 1.7 in \cite{rapaport}). We explain in section \ref{rapaport} how to recover \cite{rapaport} from our results when both apply.  D. Feng  proved a similar result for affine IFSs (\cite{feng}, see also \cite{barany-kaenmaki17}). All these papers (and this one) can be seen as higher dimension extensions of  \cite{feng-hu}. For the flag space (and for the projective space as well), the extension of the equality \( \de (\nu ) = \dim_{{\textrm{LY}}}(\F, \mu)\) to  \( \mu \) discrete in \( \MM (d), d>2 \) seems to be delicate. The proof of theorem \ref{main} may suggest a strategy. This is in accord with the results of \cite{barany-hochman-rapaport}, \cite{hochman-rapaport}.

\subsection{Strategy in dimension 3}
The main new feature of our paper is already present for finitely supported  measures in the case \( d=3.\) 
The space $\F_{0<1<3}$ is the space $\lines$ of lines in $\R^3$, the space $\F_{0<2<3}$ is the space $\planes$ of planes in $\R^3$. The stationary measure \(\nu \) on \(\F\) projects on the stationary measures \(\nu _\lines\) and \(\nu _\planes\) and the fibers of the projections are one-dimensional. We know by \cite{rapaport} that \(\nu _\lines\) and \(\nu _\planes\) are exact-dimensional. Moreover, we know by \cite{lessa} that the conditional measures on the fibers are exact-dimensional and \cite{lessa} has formulae for the almost everywhere constant dimensions. This is not enough information to be able to conclude that the measure $\nu $ is exact-dimensional and to compute its dimension. Indeed, in  the setting of Theorem \ref{LyaDim3} in dimension 3, as soon as the Hitchin representation is not Fuchsian, there exists a probability measure \( \mu _0 \in \MM (\G) \) for which the dimensions do not add up for  the projection from \(( \F, \nu _0 ) \) to \( \planes \) (see  proposition \ref{noconsdim} and  theorem \ref{Manhattan}). 

 We next recall this phenomenon of {\it{dimension conservation}} and explain what is the third codimension one projection of $\F$ that we consider and for which we will prove dimension conservation. In the following subsection, we introduce the corresponding formalism in higher dimensions.

Let \((X, \nu ), (X', \nu ') \) be  standard probability spaces, \( \pi :(X,\nu) \to (X', \nu ') \) a measure preserving mapping. Recall that a {\it{disintegration }} of the measure \(\nu \) with respect to \(\pi \) is  an a.e. defined measurable  family of probability measures \( x' \mapsto \nu^{x'} \) (or \(x' \mapsto \nu^{x'}_{X'}, x' \mapsto  \nu^{x'}_\pi \)) of probability measures on $X$ such that\[ \nu^{x'} \pi^{-1} (x') = 1 \quad {\textrm {and }} \quad \nu = \int _{X'} \nu^{x'} \, d\nu' (x').\] Two families of disintegrations of the measure \(\nu \) with respect to \(\pi \) coincide \( \nu '\)-a.e..

Assume now that \((X, d ), (X', d' ) \) are separable metric spaces and that \(\pi : (X,d)\to (X',d')\) is  a Lipschitz mapping. Let  \(\nu \) be  a probability measure on $X$. We say that the projection $\pi $ is {\it {dimension conserving for \(\nu \)}} if\begin{enumerate} 
\item the measure \(\nu \) is exact-dimensional with dimension \(\delta \),
\item  the measure \(\pi _\ast \nu \) is exact-dimensional with dimension \(\delta' \), 
\item for \(\pi _\ast \nu \)-a.e. \( x' \in X',\) the disintegration \(\nu ^{x'} \) is exact-dimensional on \( \pi ^{-1} (x') \) with dimension \(\delta - \delta'\).
\end{enumerate}
The definition is adapted from Furstenberg (\cite{furstenberg1970}, \cite{furstenberg2008}). Dimension conservation occurs often in the presence of iterations or randomness. Classical examples are the results \`a la  Marstrand and Mattila for projections of measures along almost every direction in $\R^d$ (\cite{jarvenpaa-mattila}), see \cite{jarvenpaa2-llorante} and the more recent  \cite{falconer-fraser-jin2015} and \cite{shmerkin2015} for  surveys. On the other hand, it is easy to construct examples (for instance the graphs of the Brownian trajectories or the graph of the Weierstra{\ss}    function (see \cite{shen18})) where 1 and 2 hold, but the conditional measures are Dirac measures with dimension 0 whereas \( \delta > \delta'.\) See also \cite{rapaport17} for an example in a context close to ours.

 We will now describe, in dimension 3, a third  projection defined on \(\F\) with one-dimensional fibers for which we will be able to prove dimension conservation (see corollary \ref{consdim}).

We say that two flags \(f= \{0\} \subset U_1 \subset U_2 \subset \R^3 \) and \(f'= \{0\} \subset U'_1 \subset U'_2 \subset \R^3 \) are in general position if \( U_1 \oplus U'_2 = U_2 \oplus U'_1 = \R^3.\) Set \(\F^{(2)} \) for the set of pairs of flags in general position. The mapping \( F: \F^{(2)} \to \planes \x \lines \x \F \) defined by 
\[ F(f,f') \; := \; ( U_1 \oplus U'_1, U_2 \cap U'_2, f' ) \] has one-dimensional fibers. Indeed,  fixing \(U_1\) in \(U_1 \oplus U'_1\) determines \( U_2 = U_1 \oplus (U_2 \cap U'_2) ;\) alternatively, choosing \(U_2 \supset U_2 \cap U'_2 \) determines \( U_1 = (U_1 \oplus U'_1) \cap U_2.\)

Let $\mu'$ be the image of $\mu $ under the mapping $g \mapsto g^{-1}$ and let \( \nu '\) be the \(\mu'\)-stationary probability measure  for the action of $G$ on $\F$ associated by remark \ref{stable}. From our results will follow that, for $\nu '$-a.e. $f'$, all  projections in the following sequence are dimension conserving for \( \nu \x \delta _{f'} \)
\begin{eqnarray*}
(f,f') \mapsto ( U_1 \oplus U'_1, U_2 \cap U'_2, f' ) \mapsto ( U_1 \oplus U'_1, f' ) \mapsto f' &{\textrm{ if }}& \chi _2 \geq 0,\\
(f,f') \mapsto ( U_1 \oplus U'_1, U_2 \cap U'_2, f' ) \mapsto ( U_2 \cap U'_2, f' ) \mapsto f' &{\textrm{ if }}& \chi _2 \leq 0.
\end{eqnarray*} 
In order to prove theorem \ref{main} on \( \F\) when \( d =3 \), we are reduced to three projections with one dimensional fibers, for which we want to prove exact dimensionality of the conditional measures on the fibers and dimension conservation. Moreover, we have arranged so that the exponents of the dynamics on the fibers are nondecreasing:
\begin{eqnarray*} 
\chi_3 -\chi _1 \; < \; \chi_3 - \chi _2 \; \leq \; \chi_2 - \chi _1 \; <0  &{\textrm{ if }}& \chi _2 \geq 0,\\
\chi_3 -\chi _1 \; < \; \chi_2 - \chi _1 \; \leq \; \chi_3 - \chi _2  \; <0 &{\textrm{ if }}& \chi _2 \leq 0.
\end{eqnarray*}
Therefore, we can apply the strategy of \cite{feng} and \cite{rapaport} (following \cite{ledrappier-young1985}, \cite{feng-hu}, \cite{barany-kaenmaki17}), and work one exponent at a time. We first generalize the above picture to higher dimensions.

\subsection{Topologies, configuration spaces and entropy\label{admissibletopologysection}}

A topology \(T\) on the set \(\lbrace 1,\ldots, d\rbrace\) will be called admissible if \(\lbrace i,i+1,\ldots, d\rbrace \in T\) for all \(i\).   
Given an admissible topology \(T\) we denote by \(T(i)\) the atom of \(i\) i.e. the smallest set in \(T\) containing \(i\).  Notice that any topology \(T\) is determined by listing its atoms \(T(1),T(2),\ldots, T(d)\).  
And \(T\) is admissible if, and only if, \(T(i) \subset \lbrace i,i+1,\ldots, d\rbrace\) for all \(i\).

Recall that a topology \(T\) is finer than another \(T'\) (equivalently, \(T'\) is coarser than \(T\)), denoted \(T \prec T'\), if \(T\supset T'\).
The coarsest admissible topology \(T_0\) is (defined by the list of atoms) 
\[\lbrace 1,\ldots, d\rbrace \lbrace 2,\ldots, d\rbrace \ldots \lbrace d\rbrace,\]
the finest admissible topology \(T_1\) is \(\lbrace 1\rbrace\ldots \lbrace d\rbrace\).

\begin{figure}[b]
 \includegraphics[width=\textwidth]{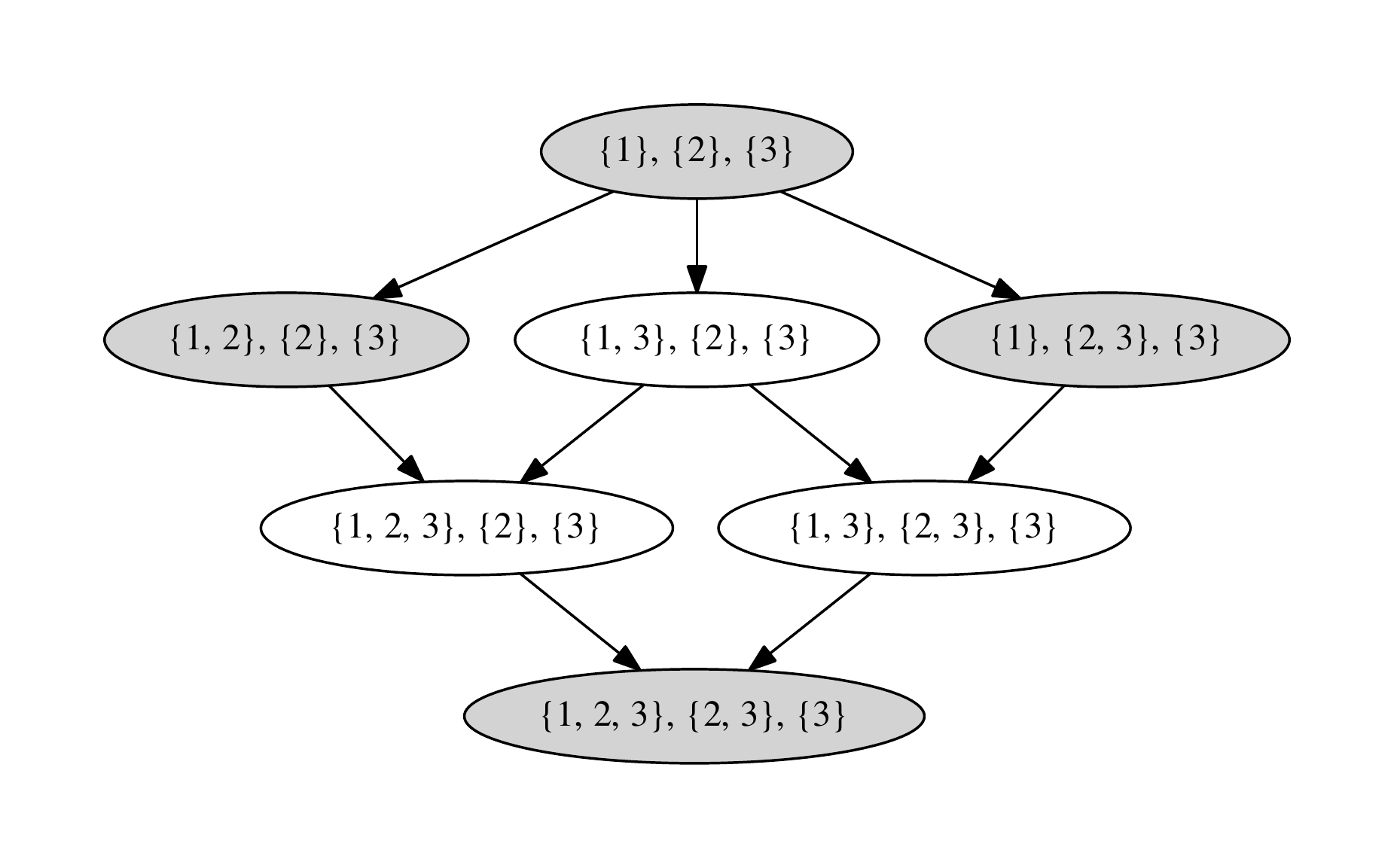}
 \caption{Admissible topologies for \(d = 3\), an arrow indicates a topology one step coarser than another, filtered topologies are indicated in gray.}
\end{figure}

We say an admissible topology \(T\) is one step finer than an admissible topology \(T'\) (equivalently \(T'\) is one step coarser than \(T\)), denoted \(T \overset{1}{\prec} T'\),  if there exists a unique \(i \in \lbrace 1,\ldots, d\rbrace\) such that \(T(i) \neq T'(i)\) and furthermore \(T'(i) \setminus T(i)\) is a singleton.
Let  \(j\) be so that \(\lbrace j \rbrace = T'(i) \setminus T(i)\). Then, \(j >i \)  and \( T(i) \setminus \{i\} \subset T(i) \subset T(i) \cup \{j\} = T'(i).\) We associate to  a pair  \(T \overset{1}{\prec} T'\) its {\it {exponent }} \(\chi _{T,T'} \) \begin{equation}\label{exponent} \chi_{T,T'} \; = \; \chi _i - \chi _j . \end{equation}

An admissible topology is called filtered if it is generated by the coarsest admissible topology and some subset of \(\left\lbrace \lbrace 1\rbrace, \lbrace 1,2\rbrace, \ldots, \lbrace 1,2,\ldots, d\rbrace \right\rbrace\). Let $Q$ be a partition of $\{0,1, \ldots , d\}$ into intervals, $Q = \lbrace q_0 = 0 < q_1 < \ldots < q_k = d\rbrace.$ We associate to it the filtered topology $T_Q$ generated by $T_0$ and the sets \( \lbrace 1,2,\ldots, q_j\rbrace, j = 1, \ldots ,k.\) There are exactly \(2^{d-1}\) filtered topologies and they are all obtained that way. The topologies $T_1$ and $T_0$ are filtered and correspond to respectively the space of complete flags and the one-point flag space of the trivial partition \(Q_0 =  \lbrace 1,2,\ldots, d\rbrace \).

\begin{figure}[b]
 \includegraphics[width=\textwidth]{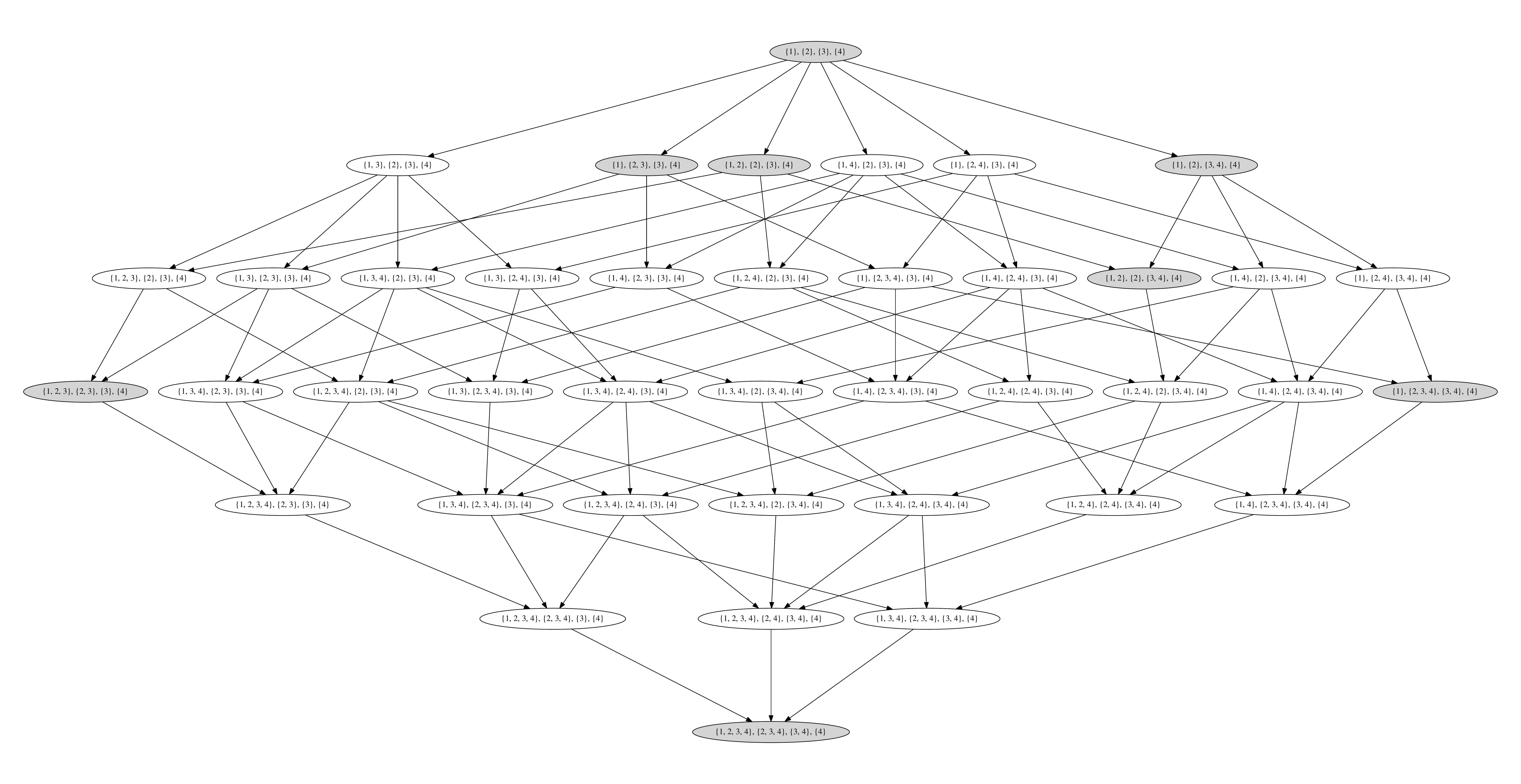}
 \caption{Admissible topologies for \(d = 4\), an arrow indicates a topology one step coarser than another, filtered topologies are indicated in gray.}
\end{figure}

Figures 1 and 2 represent the graphs of the one-step relations between admissible topologies in dimensions 3 and 4 respectively. In dimension 3, the new fibration correspond to the nonfiltered topology that is one-step finer than \(T_1\). One sees  the two ways of further descending on the graph according to the sign of \( \chi _2\). For a general \(d\), admissible topologies are in one-to-one correspondence with partial orders on \( \{ 1, \ldots, d\} \) that are suborders of the natural order. Their number as a function of \(d\) is the list A006455 of the Online Encyclopedia of Integer Sequences.\footnote{We thank Yves Coud\`ene for this observation.} In dimension 4, there are indeed 40 admissible topologies  and 92 one-step arrows. The non-trivial filtered topologies correspond to the spaces of partial  flags with only one level missing or to the Grassmannians of lines, planes or three-dimensional spaces. This correspondence is extended to all admissible topologies by constructing the {\it {configuration spaces}} as follows.

Given an admissible topology \(T\) we define the configuration space \(\X_T\) to be the space of sequences \(x = (x_I)_{I \in T}\) indexed on \(T\) where
\begin{enumerate}
 \item \(x_I\) is a \(|I|\)-dimensional subspace of \(\R^d\) for each \(I \in T\),
 \item \(x_{I \cup J} = x_I + x_J\) for all \(I,J \in T\), and 
 \item \(x_{I \cap J} = x_I \cap x_J\) for all \(I,J \in T\).
\end{enumerate}
The configuration space \(\X_{T_1} \)  is identified with the space of \(d\) independent lines in \(\R^d\);  \(\X_{T_0}\) is the space \(\F\) of complete flags. If \(T\) is finer than \(T'\) there is a natural projection mapping \(\pi_{T,T'}:\X_T \to \X_{T'}\).   In particular, all configuration spaces project onto  \(\X_{T_0}\).

We say that two flags \(f= \{0\} \subset U_1 \subset \ldots \subset  \R^d \) and \(f'= \{0\} \subset U'_1 \subset \ldots \subset \R^d \) are in general position if  for all \(j, 0 <j<d\), \( U_j \oplus U'_{d-j} = \R^d.\) Set \(\F^{(2)} \) for the set of pairs of flags in general position. Given an admissible topology \(T\), we associate to \( (f,f') \in \F^{(2)} \) the configuration \(F_T(f,f') \in \X_T \) given, for \( I \in T\),  by 
\[ [F_T(f,f')]_I \; = \oplus _{i\in I} \left( U_i \cap U'_{d-i+1} \right).\]
Observe that if \(T\) is finer than \(T'\), then \( F_{T'} = \pi _{T,T'} \circ F_T.\) Set, for an admissible topology \(T\) and a fixed \( f' \in \F,  \, \X_T^{f'} \) for the set of  \(F_T(f,f')\) for all $f$ such that \((f,f') \in \F^{(2)}.\) In particular, we identify \(\X_{T_0}^{f'} \) with \(\{f'\}\), \(\X_{T_1}^{f'} \) with the set of flags in general position with respect to \(\{f'\}\).  If  $Q$ is a partition of $\{0,1, \ldots , d\}$ into intervals, the configuration space \(\X_{T_Q}^{f'} \) is identified with the set of partial flags \( \F_Q\) in general position with respect to \(f'\).

Let $\mu'$ be the image of $\mu $ under the mapping $g \mapsto g^{-1}$ and let \( \nu '\) be the \(\mu'\)-stationary probability measure  for the action of $G$ on $\F$ associated by remark \ref{stable}. Endow 
\(\F^{(2)} \)  with the measure \( \nu \otimes \nu' \) and \(\X_T \) with  the measure \( (\F_T)_\ast (\nu \otimes \nu' ). \) Write \( f' \mapsto \nu _T^{f'} \) for the \(\nu '\)-a.e. defined family of disintegrations  of \( (\F_T)_\ast (\nu \otimes \nu' ) \) with respect to the projection on the second coordinate in \(\F^{(2)} .\) By definition, for \(\nu '\)-a.e. \( f'\), \( \nu ^{f'}_T \) is a probability measure supported by \(\X^{f'}_T \) and so that  \( (\F_T)_\ast (\nu \otimes \nu' )= \int _\F \nu _T^{f'} \, d\nu '(f').\)

We define the {\it {entropy }} \(\kappa _{T} \)  by 
\begin{equation}\label{entropy} \kappa _{T}  \; := \;\int  \log \frac{ dg_\ast\nu _T^{g^{-1} f'} }{  d\nu _T^{f'} }(y)\, dg_\ast\nu _T^{g^{-1} f'}(y)  d\nu '(f')d\mu(g) .\end{equation} 
We will see in Section \ref{section:mutualinformation} that this integral makes sense and can be seen as a conditional mutual entropy   \( H(gF_T, F_T| f') \). In particular, \(\kappa _{T_0} =0 \).  If \(T\) is filtered and associated to the partition \(Q\), then the mapping of \(\X_{T_Q}\) onto \(\F_Q\) is a bilipschitz homeomorphism  when restricted to each fiber \(  \X_{T_Q}^{f'}\), and identifies \(\nu_{T_Q}^{f'}\)  with \(\nu_Q\). Therefore, 
 \[ \kappa _{T _Q}\; = \;  \int _{G \x\F_Q} \log \frac{dg_\ast\nu_Q}{d\nu _Q} (x)\,  dg_\ast\nu_Q (x) d\mu (g) \; =\;  h (\F_Q, \mu , \nu _Q).\]

Assume that  the admissible topology \(T\) is  finer than the admissible topology \(T'\). Then, clearly,  for \(\nu '\)-a.e. \( f' \in \F \), 
\[( \pi _{T,T'} )_\ast  \nu _T^{f'}  \; = \; \nu _{T'}^{f'} \]
and  we set \( (\nu _{T,T'}^{x'} , x' \in \X_{T'}^{f'} )\) for a family of disintegrations of the measure \( \nu _T^{f'} \) with respect to \( \pi_{T,T'} .\) 
The {\it {entropy difference}} \begin{equation}\label{entropydifference}\kappa _{T,T'}   \; := \; \kappa _T -\kappa _{T'} \end{equation} 
 can be seen as a conditional mutual entropy   \( H (gF_T, F_T|F_{T'}, f') \) and can be expressed in terms of the measures \(\nu _{T,T'}^{x'} \) (see below section \ref{section:mutualinformation}).

 A key step in our proof is 
 \begin{theorem}\label{exactfibers} Fix  \( \mu \in \MM (d) .\) Assume \(T\) and \(T'\) are admissible topologies, with \(T\)  one step finer than \(T'\). With the above notations, for \(\nu'\)-a.e. \(f'\), for \( \nu _{T'}^{f'} \)-a.e. \( x' \in  \X_{T'}^{f'} \), the measure 
\( \nu _{T,T'}^{x'} \)  is exact-dimensional with dimension \( \g _{T,T'} \) given by \[ \g_{T,T'} \; = \; \frac {\kappa_{T,T'}}{\chi _{T,T'}} . \] \end{theorem}

The proof of theorem \ref{exactfibers} is given in section \ref{section:telescoping}. In dimension \( d= 2\), there is only one pair \( T_1, T_0 \), the measure \( \nu _{T_1,T_0}^{f'}  \) is the constant measure \(\nu \) and theorem \ref{exactfibers} comes from \cite{hochman-solomyak2017}. In higher dimensions, if \( T = T_1 \) and \(T'\) has one atom of the form \(\{i, i+1\} \), then theorem \ref{exactfibers} is theorem 2 from \cite{lessa}. The general scheme of the proof  is the same: we find a one-dimensional parameterization of the support of the measure 
\( \nu _{T,T'}^{x'}\) that is adapted to  the $G$-action, and then use a telescoping argument and the Maker ergodic theorem. It follows from the one-dimensional parameterization that if \(T \overset{1}{\prec} T',\) then \( \g_{T,T'} \leq 1\) (see lemma \ref{coordinate} below) and this leads to a general estimate of \(\kappa _{T,T'} \) in terms of the differences of exponents associated to \(\pi _{T,T'} \) (see proposition \ref{order}.2 below). Inequality (\ref{ent/exp}) in theorem \ref{ent/exp1} corresponds to the particular case of \( \kappa _{T_{1},T_0}.\) We provide a direct proof of  (\ref{ent/exp})  first, since we use it to ensure that the entropy \( \kappa_T\) defined  by 
equation (\ref{entropy}) is finite. Since for all \(T \overset{1}{\prec} T', \g_{T,T'} \leq  1,\)  if we have equality in (\ref{ent/exp}), all corresponding \(\g_{T,T'} \) are 1 and the equality case in theorem \ref{ent/exp1} follows from theorem \ref{mainT}.1.

In the next section, we reduce the proofs of our  results  to entropy/dimension statements related to the fine structure of the configuration spaces. A more detailed organization of the remaining proofs is given at the end of the next section.

\section{Proofs}\label{proofs}
\subsection{Proof of theorem \ref{main}}\label{section:proofs}
Observe that for  \(T\) and \(T'\)  admissible topologies, \( T\prec T' \) if, and only if, for all \( i = 1,\ldots, d-1, T(i) \subset T'(i), \) where \(T(i)\) is the atom of \(T\) containing \(\{i\}\). We denote \(D_{T,T'} \) the set of pairs \((i , j)\) such that \( j \in T'(i) \setminus T(i). \) 
The numbers \(\chi_i - \chi _j, (i,j) \in D_{T,T'}\)  are called  the exponents of the pair \( T \prec T'\).  

Observe that \(\{ i\} \subset T(i), T'(i)  \subset  \{i, i+1, \ldots d\} . \) Therefore, for \((i,j ) \in D_{T,T'}, \; i<j \) and the exponents \( \chi _i - \chi _j , (i,j) \in D_{T,T'}\) are positive. Set \( N_{T, T'} := \# D_{T,T'} = \sum _{i=1}^d \# (T'(i) \setminus T(i)).\)

\begin{proposition}\label{order}
Let  \(T \prec T'\) be a pair of admissible topologies, \( N := N_{T,T'} .\) \\
1. Then, there exists a sequence \(T^0 =T',T^1,\ldots,T^N = T\) such that \(T^{t}\) is one step finer than \(T^{t-1}\) for \(t=1,\ldots,N\), and
\[\chi_{T^1,T^0} \leq  \chi_{T^2,T^1} \leq \ldots \leq \chi_{T^{N},T^{N-1}}.\]
2. Moreover, if \(\mu \in \MM(d),\)  \( \ov \de_{T,T'}  \leq N_{T,T'} \)  and \( \kappa _{T,T'}  \leq \sum _{ (i,j) \in D_{T,T'} } (\chi _i - \chi _j),\) where 
\( \ov \de_{T,T'} \) is the essentially constant (in \(x'\)) value of the upper dimension of the measure \( \nu _{T,T'}^{x'} .\)
\end{proposition}
 Proposition \ref{order}.1 is proven in section \ref{section:topologies}.  Proposition \ref{order}.2 is proven at the end of section \ref{section:coordinates} after a more precise description of the metric spaces \( (\pi _{T,T'})^{-1} (x') \) for \( x' \in \X_{T'} .\) For any pair \(T \prec T'\) of admissible topologies, we henceforth choose and fix a sequence given by proposition \ref{order}.1. The entropy  \( \kappa _{T,T'} \) is the sum of the entropy differences \(  \kappa _{T^t, T^{t-1}}  .\)

Theorem \ref{exactfibers} yields a Ledrappier-Young  formula for the entropy \( \kappa _{T,T'} \)
 \begin{equation}\label{LY} \kappa _{T,T'} \; = \;  \sum _{t=1}^{N_{T,T'}} \kappa _{T^t, T^{t-1}} \;= \;  \sum _{t=1}^{N_{T,T'}}  \chi _{T^t,T^{t-1}} \, \g _{T^t, T^{t-1}}\, .\end{equation}
The precise form of our main result is the dimension counterpart of (\ref {LY}):
\begin{theorem}\label{mainT} Let \( \mu \in \MM(d)\) and  \(T \prec T'\)  be a  pair of admissible topologies. With the previous notations, for \(\nu '\)-a.e. \(f'\in \F\), \( \nu _{T}^{f'} \)-a.e \(x' \in \X_{T'}^{f'} \), \begin{enumerate}
\item the lower dimension  \( \un \delta _{T , T'}\) of the measure \( \nu _{T,T'}^{x'} \) is at least \(\sum \limits_{t=1}^{N_{T,T'}} \g _{T^t, T^{t-1}};\)\\
\item moreover, in the case when \( \mu \) is discrete, the measure \( \nu _{T,T'}^{x'} \) is exact-dimensional, with dimension \(\delta _{T , T'}\) given by \[ \delta _{T ,T'}\; = \; \sum _{t=1}^{N_{T,T'}} \g _{T^t, T^{t-1}}\,.\] \end{enumerate} \end{theorem}

Theorem \ref{mainT}.2 gives a condition for dimension conservation:
\begin{corollary}\label{consdim}  Assume \( \mu \) is a discrete measure in \(\MM(d)\) and let \( T \prec T' \prec T'' \)  be  admissible   topologies. Assume that all differences  \( \chi _i -\chi _j, (i,j) \in D_{T,T'} \) are at least as large as the differences \( \chi _i -\chi _j, (i,j) \in D_{T',T''} \). Then, there is dimension conservation for the  projections \( \pi _{T,T''} = \pi _{T',T''} \circ \pi _{T,T'} \) of the measure \( \nu _{T}^{f'}\):
\[ \delta _{T ,T''} \; = \; \delta _{T' ,T''} + \delta _{T,T'}.\] \end{corollary}
\begin{proof} Since differences  \( \chi _i -\chi _j, (i,j) \in D_{T,T'} \) are at least as large as the differences \( \chi _i -\chi _j, (i,j) \in D_{T',T''} \), it is possible to find the sequence associated to \( T\prec T''\) by  proposition \ref{order}.1 in such a way that \( T^{N_{T',T''}} = T'.\) The corollary follows. \end{proof}
Observe that there are examples of \( T \prec T' \prec T'' \) where  dimension conservation does not hold (see e.g. proposition \ref{noconsdim} and  theorem \ref{Manhattan}).

In the case when \(T' = T_0,\) we can identify the measures \( \nu _{T^t}^{f'} \) and
\( \nu _{T^t,T_0}^{ f'} \) on \(\X^{f'}_{T^t}\) and we obtain, setting  \( D_T := D_{T, T_0} = \{ (i,j), i<j , j \not \in T(i)\}, N_T := \# D_T\) and, for  \( (i,j) \in D_T, \g^T_{i,j} \)  for \(\g_{T^t,T^{t-1}}\) associated by proposition \ref{order}.1 to \((i,j)\).
\begin{corollary}\label{exact-dim}  Assume \( \mu \) is a discrete measure in \(\MM(d)\) and let \(T\)  be an admissible   topology. With the previous notations, we have \( \kappa _T \;= \; \sum _{(i,j) \in D_T} \g ^T_{i,j} (\chi _i -\chi _j) \) and, for \(\nu '\)-a.e. \(f'\in \F\), the measure \( \nu _{T}^{f'} \) is exact-dimensional, with dimension \(\delta _{T}\) given by \[ \delta _{T}\; = \; \delta _{T,T_0} \; =\;  \sum _{(i,j) \in D_T} \g ^T_{i,j} .\] 
\end{corollary}
 
In particular, if \(T\) is a filtered admissible topology associated to a partition $Q$, then \(\X_T\) is identified with \(\F_Q\), the measure \( \nu _{T}^{f'} \) is then identified with the measure \(\nu _Q\) for almost every \(f'\).
\begin{corollary} \label{mainQ} Let $\mu \in \MM(d)$ be a discrete probability measure on $ SL_d(\R)$.  
Let $Q$ be a partition of $\{0,1, \ldots , d\}$ into intervals. 
Then the unique stationary probability measure $\nu _Q$ on the space \(\F_Q\) of partial flags is exact-dimensional.  There are numbers \( \g ^Q_{i,j} := \g_{i,j}^{T_Q} \) such that \footnote{Comparing with theorem \ref{ent/exp1}, the content of (\ref{etvoila}) for \( Q_1\)  is that the numbers \(  \g^{Q_1}_{i,j} \) are the dimensions of certain conditional measures on specific 1-dimensional leaves in \( \F\).}
\begin{equation}\label{etvoila}  \de _Q= \sum _{i,j: \ell _Q(i) < \ell _Q(j) }  \g^Q_{i,j} , \quad \quad h(\F_Q,\mu , \nu _Q) =  \sum _{i,j: \ell _Q(i) < \ell _Q(j) }  \g^Q_{i,j} (\chi _i - \chi _j). \end{equation}\end{corollary}
Theorem \ref{main} is the particular case of corollary \ref{mainQ} when \( Q = Q_1.\)

\begin{proof}[Proof of Theorem \ref{mainT}]
We endow \( \X^{f'}_{T} \) with a smooth metric. The following limits are  constants  for  \(\nu'\)-a.e.  \(f' \in \F\), \(\nu ^{f'}_{T^t} \)-a.e. \(x' \in \X^{f'}_{T^t} \),  \(\nu ^{x'}_{T,T^t} \)-a.e. \(y' \in \X^{f'}_{T} \):
\[ \un \delta ^t := \liminf _{r\to 0} \frac {\log \nu _{T, T^t}^{x'}( B (y' ,r))}{\log r}, \;\; \ov \delta^t := \limsup_{r\to 0} \frac {\log \nu _{T, T^t}^{x'} ( B (y' ,r))}{\log r}. \]
With this notation,  lower and upper dimensions of the measure \( \nu _{T,T'}^{x'} \) are respectively \(\un \de^0 \) and \(\ov \de ^0.\)

For the first part  of theorem \ref{mainT}, we  have that for almost every \((f', x', y' )\), the following relation  (\ref{induction1})   holds for all \( t = N_{T,T'}, N_{T,T'} -1, \ldots , 1 \)
\begin{equation}\label{induction1}  \un \delta^{t-1} \;\geq \; \un \delta^{t} +  \g _{T^{t}, T^{t-1}}. \end{equation}

Indeed, since  \(T \overset {1}{\prec} T^{N_{T,T'}-1}\), (\ref{induction1})  holds for \( t = N_{T,T'} \) (with  \(\un  \de ^{N_{T,T'} } = 0 \)) by theorem \ref{exactfibers}. By \cite{ledrappier-young1985} lemma 11.3.1, (\ref{induction1}) for \( t < N_{T,T'} \)  follows from theorem \ref{exactfibers} as well. Theorem \ref{mainT}.1 follows by summing the relations (\ref{induction1}) for  \( t = N_{T,T'}, N_{T,T'} -1, \ldots , 1 .\)

For the second part of theorem \ref{mainT},  it remains to prove 
\begin{theorem}\label{counting} Assume that the measure \(\mu \in \MM(d)\) is discrete. With the above notations, for almost every \((f', x', y' )\), for all \( t = N_{T,T'}, N_{T,T'} -1, \ldots , 1, \)  \begin{equation}\label{induction2} \ov \delta^{t-1} \leq \ov \delta^{t} +  \g _{T^{t}, T^{t-1}}. \end{equation} \end{theorem}
Summing the relations (\ref{induction2})  for  \( t = N_{T,T'}, N_{T,T'} -1, \ldots , 1 \) (with  \(\ov  \de ^{N_{T,T'} } = 0 \)) gives \( \ov  \delta _{T ,T'} = \ov \de ^0   \leq    \sum _{t=1}^{N_{T,T'}} \g _{T^t, T^{t-1}}.\) Comparing with theorem \ref{mainT}.1 gives the result. \end{proof}

We prove theorem \ref{counting} in Section \ref{addingdimensions}. The analog of theorem \ref{counting} in \cite{ledrappier-young1985} is a counting argument (see section (10.2)) that uses partitions with finite entropy in the underlying space. This is not possible here.
 By working on the trajectory space of the underlying process, \cite{feng} and \cite{rapaport}  perform this counting procedure in the case when the measure $\mu $ has finite support. We follow the same scheme under the hypothesis that \( \mu \) is discrete.  

\subsection{Properties of the partial dimensions}\label{section:gamma}
  The spaces \( (\pi _{T^t, T^{t-1}})^{-1} (x)\) and the measures \( \nu^{x}_{T^t,T^{t-1}} \) for a.e. \(x\) depend only on the arrow  \( T^t  \overset {1}{\prec} T^{t-1} \)and not on its environment \( T \prec T'\) such that 
 \(T \prec  T^t  \overset {1}{\prec} T^{t-1} \prec T'\). Indeed, the dimension  \( \g_{T^t,T^{t-1}}\) in formula (\ref{LY}) does not depend on the environment \( T \prec T'\).
But still,  there might be  several arrows corresponding to the  same pair \( (i,j), 0< i <j \leq d.\) We can write
\begin{lemma}\label{factors} Assume the diagram of projections \(\begin{matrix} T & \longrightarrow & S  \\
\downarrow 1 & & \downarrow 1 &   \\
T' & \longrightarrow & S'  
\end{matrix}\) commutes and \(i,j\) are such that \(T(i) = T'(i) \setminus \lbrace j\rbrace\) and \(S(i) = S'(i) \setminus \lbrace j \rbrace\). Then, \( \g_{T,T'} \leq \g_{S,S'} .\) \end{lemma}
Lemma \ref{factors} is proven in section \ref{section:telescoping}. The example in section \ref{example} shows that, in general, \( \g_{T,T'}< \g_{S,S'} .\) 

Let \(i,j\)  satisfy \( 0< i <j \leq d\) and let $T_{i,j} $ be the topology defined by \[ T_{i,j} (k) = \{k\} \;\; {\textrm {if}} \;\; k\neq i, \; T_{i,j}(i) = \{i, j\}.\]
The topology \(T_{i,j} \) is admissible and one step coarser than $T_1$. 
\begin{corollary} Let \( T \prec T''\)  be admissible topologies.
 For  \( T^t  \overset {1}{\prec} T^{t-1} \) in the decomposition of \( T \prec T''\)  and \((i,j)\) such that  \( T^{t-1}(i) = T^{t}(i )\cup \{j\} \), we have \(  \g_{T_1, T_{i,j} }  \leq \g_{T^t, T^{t-1}}.\)  \end{corollary}
 \begin{proof}  Apply lemma \ref{factors}.to the commutative part of the diagram
 \[\begin{matrix}& & & T &\\
  &  &  & \downarrow &\\
  & T_1 &\longrightarrow & T^t &\\
 &  \downarrow 1 &  & \downarrow 1& \\
 & T_{i,j} &\longrightarrow & T^{t-1} &\longrightarrow T''.
\end{matrix}\]  \end{proof}

 In the case when  there are several pairs \(\{(i_1,j_1), \ldots , (i_k, j_k) \} \) with the same difference \( \chi _i - \chi _j \), there are different decompositions given by proposition \ref{order}.1. If the measure \( \mu \) is discrete, we can apply theorem \ref{mainT}.2 to each decomposition and get a common formula by grouping together the terms corresponding to the  same difference \( \chi _i - \chi _j \). Namely, set  \( \de_{T^{t+k}, T^{t}} := \sum \limits_{\ell =1}\limits^k \g _{T^{t +\ell}, T^{t +\ell -1}} .\)  By theorem \ref{mainT}.2 applied to \( T^{t+k} \prec T^{t} ,\) this number  \( \de_{T^{t+k}, T^{t}}\)  is the exact dimension of the measures \( \nu _{T^{t+k},T^t} \) and thus is independent of the order of the decomposition.  We also obtain that \( \kappa _{T^{t+k}, T^{t}} =(\chi _i - \chi _j ) \de _{T^{t+k}, T^{t}}.\)

\subsection{Lyapunov Dimension}\label{Lyapunov}
Theorem \ref{LyaDim1} follows  from the previous results. We state and prove  the corresponding result for a general  partition $Q$ of $\{0,1, \ldots , d\}$ into intervals. Assume \( \mu \in \MM(d)\)  is discrete. 

Let $\chi _1> \ldots > \chi _d$ be the Lyapunov exponents of $(G, \mu )$ and for $ i = 1, \ldots d,$ set $\ell _Q(i) $ for the  index such that $ q_{\ell _Q(i) -1} < i \leq  q_{\ell _Q(i)}.$ Denote \( \{ 0 < \lambda _1 \leq \lambda _2 \leq \ldots \leq \lambda _N \} \) the differences of exponents \( \chi _i - \chi _j \) for all \((i,j)\) such that \( \ell _Q(i) < \ell _Q(j).\)  Define the continuous, piecewise affine  function \( D_{\F_Q, \mu } \) on the interval \( [ 0,N= \dim \F_Q]\)
 as:
\[ D_{\F_Q, \mu } (0) := h(\F_Q, \mu, \nu _Q) \quad {\textrm {and}}\quad  D'_{\F_Q, \mu } (s) = - \lambda _k\; {\textrm {for}} \; s \in (k-1,k), \; k = 1, \ldots, N. \]
As before, the {\it {Lyapunov dimension}} \( \dim_{{\textrm{LY}}}(\F_Q, \mu  )\) 
 is such that \( D_{\F_Q,\mu} ( \dim_{{\textrm{LY}}}(\F_Q, \mu  ) ) = 0.\) Observe that by theorem \ref {ent/expQ1}, \( \dim_{{\textrm{LY}}}(\F_Q, \mu  ) \leq \dim \F _Q.\) 

\begin{proposition}\label{LyaDimQ} Let \( \mu \in \MM(d)\) be discrete, $Q$ a partition of $\{0,1, \ldots , d\}$ into intervals. Then, the exact dimension \(\de _Q\) of \( \nu  _Q\) satisfies
\[ \de _Q \; \leq \; \dim_{{\textrm{LY}}}(\F_Q, \mu).\]
\end{proposition}
\begin{proof}The measure \( \nu _Q\) is exact-dimensional with dimension \(\de_Q\).  By equation (\ref{etvoila}), there are numbers \( \g ^Q_{i,j} \) such that \(  \de _Q= \sum _{i,j: \ell _Q(i) < \ell _Q(j) }  \g^Q_{i,j} .\)  Moreover,  \(h(\F_Q,\mu , \nu _Q) =  \sum _{i,j: \ell _Q(i) < \ell _Q(j) }  \g^Q_{i,j} (\chi _i - \chi _j).\)  We ordered \( \{ 0 < \lambda _1 \leq \lambda _2 \leq \ldots \leq \lambda _N \} \) the  \( \lambda _k := \chi _{i_k} - \chi _{j_k} \) for all \( (i_k,j_k)\) such that \(  \ell _Q(i_k) < \ell _Q(j_k) \). We can define another  continuous, piecewise affine Lyapunov  function \( \un D_{\F_Q, \mu, \nu _Q } \) on the interval \( [ 0, \de _Q ]\)  such that  \( \un D_{\F_Q, \mu, \nu _Q } (0) = h(\F_Q, \mu, \nu_Q)\) and the slope of  \( \un D_{\F_Q, \mu, \nu _Q } (s) \) on the successive intervals of length \( \g^Q_{i_k,j_k}\)  is \( - \lambda _k.\) By (\ref{etvoila}),  \( \un D_{\F_Q, \mu, \nu _Q } (\de _Q) = 0.\) On the other hand, since  for all \((i,j), \;  \g^Q_{i,j} \leq 1 \) (proposition \ref{order}.2),
 for all \( s \in [0, \de _Q],\)  \[  \un D_{\F_Q, \mu, \nu _Q } (s) \; \leq \; D_{\F_Q, \mu } (s).\]
In particular,  \( D_{\F_Q, \mu } (\de_Q) \geq  \un D_{\F_Q, \mu, \nu _Q } (\de _Q) = 0.\) This shows proposition (\ref{LyaDimQ}). \end{proof}
\subsection{Content of the paper}
Given the previous discussion, we still have to prove theorems \ref{ent/exp1}, \ref{LyaDim3}, \ref{exactfibers},  \ref{counting} and \ref{Manhattan}, proposition  \ref{order} and lemma  \ref{factors}. In section \ref{section2}, we show proposition \ref {order} and  discuss the geometry of the spaces \(  \X_{T}^{f'} \). In particular, lemma \ref{coordinate} describes  the one-dimensional structure of $\pi _{T,T'}^{-1}(x')$ for 
\( x' \in \X_{T'}^{f'} \) and  \(T \overset {1}{\prec} T'\). We also discuss the multidimensional structure of the configuration spaces (proposition 
\ref{lipschitzproduct}). We  recall in section \ref{section:oseledets} the underlying trajectory space of the associated random walk, and the applications of Oseledets theorem to random walks of matrices. In section \ref{section:ent/exp},  we prove theorem \ref {ent/expQ1}, the generalisation of theorem \ref {ent/exp1}.   In section \ref{section:mutualinformation},  we recall the notion  of mutual information of random variables and its properties.  We prove  theorem \ref{exactfibers},  lemma \ref{factors} in section \ref{section:telescoping} and theorem  \ref{counting} in section \ref{addingdimensions}. The central  arguments in sections \ref{section:ent/exp}, \ref{section:telescoping} and \ref{addingdimensions} have a long history in ergodic theory. A short comment about background heads each of the corresponding sections. In section \ref{example}, we discuss the case of Hitchin representations of cocompact surface groups, exhibit examples of non-conservation of dimension and   prove theorems \ref{LyaDim3} and \ref{Manhattan}.

\section{Preliminaries}\label{section2}
\subsection{Topologies and exponents}\label{section:topologies}

\begin{proof}[Proof of Proposition \ref{order}.1]
Let us order the indices in \( D_{T, T'} \) in such a way that 
 \[\chi_{i_1} - \chi_{j_1} \le \chi_{i_2} - \chi_{j_2} \le \cdots \le \chi_{i_{N_{T, T'} }} - \chi_{j_{N_{T, T'} }}.\]
 
Beginning with \(T^0 = T'\) define the sequence inductively so that for each \(t = 1,2,\ldots, N_{T, T'} \) the topology \(T^{t}\) is generated by \(T^{t-1}\) together with the set \(T^{t-1}(i_{t}) \setminus \lbrace j_{t}\rbrace\). Since \(T^t \prec T^{t-1}\), it follows by induction that \(T^t\) is admissible.

We have that \(T^t \prec T^{t-1}\); we claim that \(T^t\) is one step finer than \(T^{t-1}\).
Indeed, if this is not the case, then  there exists \(j \in T^{t-1}(i)\) such that \(i_t < j < j_t\) and \(T^{t}(j) \setminus T^{t-1}(j) = \lbrace j_t\rbrace\).
Since \(\chi_j - \chi_{j_t} < \chi_{i_t} - \chi_{j_t}\) it must be the case that \(j_t \in T(j)\), otherwise \(j_t\) would have been removed from \(T^{t-1}(j)\) at some previous step.
Similarly, since \(\chi_{i_t} - \chi_j < \chi_{i_t} - \chi_{j_t}\) we obtain that \(j \in T(i_t)\).
However, these two facts imply that \(j_t \in T(i_t)\) which is a contradiction.  It follows that \(T^t\) is one step finer than \(T^{t-1}\) as claimed.

We finally claim that \(T \prec T^t\). To see this, observe that \(T^t\) is generated by \(\lbrace A(i) \rbrace,\) for \( i = 1, \ldots ,d ,\) where \(A(i) = T^{t-1}(i)\) if \(i \neq i_t\) and \(A(i_t) = T^{t-1}(i_t) \setminus \lbrace j_t\rbrace\).
Inductively, if \(i \neq i_t\), one has \(A(i) = T^{t-1}(i) \supset T(i)\).   By construction \(j_t \notin T(i_t)\) and therefore \(A(i_t) \supset T(i_t)\) as well.   Hence, \(T \prec T^t\) as claimed.
\end{proof}

We also record here the following fact about admissible topologies which will be used several times.
\begin{proposition}\label{onestepproposition}
 Let \(T \overset{1}{\prec} T'\) be admissible topologies and \(i,j\) be such that \(T(i) = T'(i) \setminus \lbrace j\rbrace\).   Then both \(T'(i) \setminus \lbrace i\rbrace\) and \(T'(i) \setminus \lbrace i,j\rbrace\) are in \(T'\).
\end{proposition}
\begin{proof}
Since \(T'\) is admissible we have \(T'(i) \subset \lbrace i,i+1,\ldots,d\rbrace,\) and \(\lbrace i+1,i+2,\ldots,d\rbrace \in T'\).  Therefore, \(T'(i) \setminus \lbrace i\rbrace = T'(i) \cap \lbrace i+1,\ldots,d\rbrace\) is in \(T'\) as claimed.

Repeating the argument for \(T\) we obtain that \(T'(i) \setminus \lbrace i,j\rbrace = T(i) \setminus \lbrace i\rbrace\) is in \(T\).

Combining this with the fact that \(T \overset{1}{\prec} T'\) we obtain
\[T'(i) \setminus \lbrace i,j\rbrace = \bigcup\limits_{k \in T'(i) \setminus \lbrace i,j\rbrace}T(k) = \bigcup\limits_{k \in T'(i) \setminus \lbrace i,j\rbrace}T'(k),\]
so that \(T'(i) \setminus \lbrace i,j\rbrace\) is in \(T'\) as claimed.
\end{proof}

\subsection{Distances on configuration spaces}\label{section:coordinates}

Let \(\GG_i\) be the Grassmannian manifold of \(i\)-dimensional subspaces of \(\R^d\).

We fix on each \(\GG_i\) a Riemannian metric which is invariant under the action of orthogonal transformations with the additional property that if \(S,S' \in \GG_i\) are such that \(S + S'\) has dimension \(i+1\) then the Riemannian distance satisfies
\[\dist(S,S') = \angle(\pi(S),\pi(S')),\]
where \(\pi:S+S' \to (S+S')/(S\cap S')\) is the projection onto the quotient space \((S+S')/(S \cap S')\), which is endowed with the inner product inherited from \(\R^d\).

From the definition it follows that when \(\dim(S+S') = i+1\) one has 
\[\dist(S+W,S'+W) \le \dist(S,S'),\]
for all subspaces \(W\) such that \(\dim(S+W) = \dim(S'+W)\).

Given an admissible topology \(T\) we define the distance on the configuration space \(\X_{T}\) so that
\[\dist((x_I)_{I \in T}, (x_I')_{I \in T}) = \sum\limits_{I \in T}\dist(x_I,x_I').\]
\begin{proposition}\label{lipschitzproduct} Let \( T\prec T' \) be admissible topologies and \( x' \in \X _{T'} \). Then, \((\pi _{T,T'} )^{-1} (x') \), endowed with the metric \(\dist \),  is locally bilipschitz homeomorphic to the Euclidean space \(\R ^{N_{T,T'}}.\) \end{proposition}
\begin{proof} 
By proposition \ref{order}.1, we take   a sequence \(T^0 =T',T^1,\ldots,T^N = T\) such that, \(T^{t}\) is one step finer than \(T^{t-1}\) for \(t=1,\ldots,N\). We prove by increasing induction on \(t\) that \((\pi _{T^t,T'} )^{-1} (x') \), endowed with the metric \(\dist \),  is locally bilipschitz homeomorphic to the Euclidean space \(\R ^t.\) This is trivially true for \( t= 0\). So, we assume for \(t>0\), that for any \( y' \in (\pi _{T^{t-1},T'} )^{-1} (x') ,\)  there is a neighborhood of \(y'\)  in \((\pi _{T^{t-1},T'} )^{-1} (x') \) which   is bilipschitz homeomorphic to the Euclidean space \(\R ^{t-1}.\) The fibers of the projections \( \pi _{T^t, T^{t-1}}\)  form a \(C^\infty \) foliation of \( (\pi _{T^{t},T'} )^{-1} (x') . \) We have to verify that the induced metric by \(\dist\) on the fibers is bilipschitz equivalent to the Euclidean one dimensional metric, uniformly in the neighborhood of \( y'\).

Let \((i, j)\)   such that \(T^t(i) = T^{t-1} (i) \setminus \lbrace j\rbrace\), we define a distance on each fiber of the projection \(\pi_{T^t,T^{t-1}}\) by setting
\begin{equation}\label{fibermetric}\dist_{T^t,T^{t-1}}^{x'}(x_1,x_2) = \dist((x_1)_{T(i)},(x_2)_{T(i)}),\end{equation}
for each \(x' \in \X_{T^{t-1}}\) and \(x_1,x_2 \in \pi^{-1}_{T^t,T^{t-1}}(x')\).
These  distances are all Lipschitz equivalent on the fibers:
\begin{lemma}\label{fiberdistancelemma}
For all admissible topologies \(T \overset{1}{\prec} T'\), all \(x' \in \X_{T'}\) and all \(x_1,x_2 \in \pi_{T,T'}^{-1}(x')\), one has
\[\dist_{T,T'}^{x'}(x_1,x_2) \le \dist(x_1,x_2) \le 2^d\dist_{T,T'}^{x'}(x_1,x_2).\]
\end{lemma}
\begin{proof}
The inequality \(\dist_{T,T'}^{x'}(x_1,x_2) \le \dist(x_1,x_2)\) is immediate from the definitions.  For the second inequality we assume \(x_1 \neq x_2\).

Let \(i,j\) be such that \(T(i) = T'(i) \setminus \lbrace j\rbrace\).

Since \((x_1)_{T(i)}\) and \((x_2)_{T(i)}\) are distinct codimension one subspaces of \(x'_{T'(i)}\) their sum is \(x'_{T'(i)}\), and therefore
\[\dist_{T,T'}^{x'}(x_1,x_2) = \dist((x_1)_{T(i)},(x_2)_{T(i)}) \ge \dist((x_1)_{T(i)} + W,(x_2)_{T(i)}+W),\]
for all subspaces \(W\) containing neither \((x_1)_{T(i)}\) nor \((x_2)_{T(i)}\).

For each \(I \in T \setminus T'\) one has \(i \in I\) and therefore \(T(i) \subset I\).   Noticing that \(J := \bigcup\limits_{k \in I \setminus \lbrace i\rbrace}T(k) = \bigcup\limits_{k \in I \setminus \lbrace i\rbrace}T'(k)\) belongs to \(T'\) we obtain
\begin{align*}
\dist(x_1,x_2) &= \sum\limits_{I \in T \setminus T'}\dist((x_1)_I,(x_2)_I) 
\\ &= \sum\limits_{I \in T \setminus T'}\dist((x_1)_{T(i)} + x'_J, (x_2)_{T(i)} + x'_J) \le 2^d \dist_{T,T'}^{x'}(x_1,x_2). 
\end{align*}
\end{proof}

 Consider \( T, T'\) admissible topologies with \(T \overset {1}{\prec} T'\), \( (i<j) \) such that \( T'(i )= T(i) \cup \{j\} .\)  Given \(x \in \X_T\), we use the metric (\ref{fibermetric}) to  define a bilipschitz homeomorphism \(\varphi_x:(-\pi/2,\pi/2) \to \pi_{T,T'}^{-1}(x')\) where \(x' = \pi_{T,T'}(x)\).
 
 For this purpose let \(V = x'_{T'(i)}/x'_{T'(i) \setminus \lbrace i,j\rbrace}\) endowed with the inner product inherited from the ambient space \(\R^d\) (i.e. the inner product between two classes is calculated by taking representatives perpendicular to \(x'_{T'(i) \setminus \lbrace i,j\rbrace}\)).

One has that \((\pi_{T,T'})^{-1}(x')\) consists in configurations \(z\) with \(z_{T'(k)} = x'_{T'(k)}\) for all \(k \neq i\) while \(z_{T(i)}\) is a codimension one subspace of \(x'_{T'(i)}\) which contains \(x'_{T'(i) \setminus \lbrace i,j\rbrace}\) and is distinct from \(x'_{T'(i) \setminus \lbrace i\rbrace}\).
It follows that \((\pi_{T,T'})^{-1}(x')\) endowed with \(dist^{x'}_{T,T'}\) is isometric to the space of one dimensional subspaces of \(V\) minus the projection of \(x'_{T'(i) \setminus \lbrace i\rbrace}\) with the angle distance.

Let \(X,Y\) be a pair of unit vectors in \(V\) such that \(X\) has a representative in  \(x_{T(i)}\),  \(Y\)  has a representative in \(x'_{T'(i)\setminus \lbrace i\rbrace}\) and such that \(\cos \angle (X,Y) >0 \). Define \(\vf_x :(-\pi /2, + \pi /2)  \to \pi_{T,T'}^{-1} (x') \) by associating to \(u \in (-\pi/2, +\pi /2) \) 
the configuration where the corresponding one dimensional subspace of \(V\) contains a vector of the form \(\cos u X + \sin u Y.\)  Set \( \theta : = \angle (X,Y). \)
 
\begin{lemma}\label{coordinate}
 In the above context \(\varphi_x\) is a bilipschitz homeomorphism. Moreover,  \begin{equation}\label{Lipschitzconstant} |\tan \frac{\theta }{2} | \, < \, {\textrm{ Lip }} (\vf _x) \, < \, \frac{1}{ |\tan \frac{\theta }{2} |}. \end{equation}
\end{lemma} 
\begin{proof}

Let \(e_1 = (1,0), e_2 = (0,1)\) be the standard basis of \(\R^2\) and \(\GG_1(\R^2)\) be the space of one dimensional subspaces with the angle distance.
Let \(L_2\) be the subspace generated by \(e_2\).

The mapping \(f:(-\pi/2,\pi/2) \to \GG_1(\R^2) \setminus \lbrace L_2\rbrace\) where \(f(u)\) is the subspace generated by \(\cos(u)e_1 + \sin(u)e_2\) is an isometry.

By \cite[Lemma 2]{lessa}, if \(\theta \in (0,\pi/2]\) then the mapping \(g_{\theta}:\GG_1(\R^2) \setminus \lbrace L_2\rbrace\) induced by the matrix \(A = \begin{pmatrix}\sin(\theta) & 0\\ \cos(\theta) &1\end{pmatrix}\) has derivative
\(|dg_{\theta}(L)| = \frac{|\det(A)|}{\|A_{\vert L}\|^2}\).

This implies for the composition one has
\[|dg_{\theta}\circ f(u)| = \frac{|\sin(\theta)|}{\cos(u)^2\sin(\theta)^2 + (\cos(u)\cos(\theta) + \sin(u))^2} = \frac{|\sin(\theta)|}{1 + \sin(2u)\cos(\theta)}.\]

Letting \(\theta = \dist(x_{T(i)},x'_{T'(i) \setminus \lbrace i\rbrace})\) the fiber \(\pi_{T,T'}^{-1}(x')\) endowed with \(\dist_{T,T'}^{x'}\) is isometric to the projective space \(\GG_1(\R^2) \setminus \lbrace L_2 \rbrace\) where \(x\) is identified with the subspace generated by \(\sin(\theta)e_1 + \cos(\theta)e_2\).

Under this identification \(\varphi_x\) corresponds to the composition \(g_{\theta}\circ f\), so we have obtained
\[|d\varphi_x(u)| = \frac{|\sin(\theta)|}{1 + \sin(2u)\cos(\theta)}.\]
\end{proof}

To finish the proof of proposition \ref{lipschitzproduct}, we still have to show that the homeomorphism \( \vf _z , z \in  (\pi _{T^t,T'} )^{-1} (x') \) depends continuously on \( z\) as a bilipschitz homeomorphism. Given the above construction, this  is true since we can locally  choose in a continuous way parameterizations of the spaces \[V_z =  [\pi_{T^t,T^{t-1}} (z)]_{T^{t-1} (i)}/ [\pi_{T^t,T^{t-1}} (z)]_{T^{t-1} (i)\setminus \{i, j\}} \] and in these spaces 
vectors \( X_z, Y_z \) such that \(X_z \) has a representative in \( z_{T^t (i)} ,\) \(Y_z \) has a representative in \( [\pi_{T^t,T^{t-1}} (z)]_{T^{t-1} (i)\setminus \{i\}} \)  and \( \cos \angle (X_z, Y_z) >0 .\)
\end{proof}

\begin{proof}[Proof of proposition \ref{order}.2]
For \(T \prec T'\)  admissible topologies, it follows from proposition \ref{lipschitzproduct} that \( \overline \de _{T,T'} \leq N_{T,T'} .\) Also since  \(T^t \overset{1}{\prec} T^{t-1} \), we have \( \ov \g _{T^t,T^{t-1}} \leq 1.\)  By formula (\ref {LY}), we have indeed \( \kappa _{T,T'}  \leq \sum _{ (i,j) \in D_{T,T'} } (\chi _i - \chi _j).\) \footnote{Observe that relation (\ref{LY}) depends on theorem \ref{exactfibers}, which will be proven in  section \ref{section:ent/exp}.}
\end{proof}

\subsection{One-dimensional coordinates}\label{section:coordinates2}
Let \(i<j \) and consider all arrows \( T, T'\)  of admissible topologies with \(T \overset {1}{\prec} T'\), such that \( T'(i )= T(i) \cup \{j\} .\)  Recall that $T_{i,j} $ is  the topology defined by \[ T_{i,j} (k) = \{k\} \;\; {\textrm {if}} \;\; k\neq i, \; T_{i,j}(i) = \{i, j\}.\]
The topology \(T_{i,j} \) is admissible and one step coarser than $T_1$. 
\begin{lemma}\label{lipschitzlemma}
Let \(T \overset{1}{\prec} T'\) be admissible topologies and \(i,j\) be such that \(T(i) = T'(i) \setminus \lbrace j\rbrace\).
For all \( y',x'\) such that \( \pi _{ T_{i,j},T'} y' = x',\)  \( \pi _{T_1, T} \) defines  a bilipschitz homeomorphism  between   \((\pi _{T,T'})^{-1} (x') \) and \( (  \pi _{T_1, T_{i,j}})^{-1} (y').\) The Lipschitz constants depend only on \(y'\).
\end{lemma}
\begin{proof}
Consider \(y_1,y_2 \in \X_{T_1}\), distinct and such that \(\pi_{T,T_{i,j}}(y_1) = \pi_{T,T_{i,j}}(y_2) = y'\).
Set \(x_1 = \pi_{T_1,T}(y_1)\) and \(x_2 = \pi_{T_1,T}(y_2)\) and notice that \(\pi_{T,T'}(x_1) = \pi_{T,T'}(x_2) = \pi_{T_{i,j},T'}(y') = x'\).
Since \(\pi_{T_1,T}\) consists of forgetting some subspaces of each sequence \((x_I)_{I \in T_1}\) it is \(1\)-Lipschitz, and in particular \(1\)-Lipschitz as a mapping between \(\pi_{T_1,T_{i,j}}^{-1}(y')\) and \(\pi_{T,T'}^{-1}(x')\).

To prove that the inverse is also Lipschitz let \(S_1 = (y_1)_{\lbrace i \rbrace}, S_2 = (y_2)_{\lbrace i \rbrace}\).
Notice that \(S_1,S_2\) are distinct one dimensional subspaces of the two dimensional subspace \(y'_{\lbrace i,j\rbrace}\).  Therefore \(S_1 + S_2 = y'_{\lbrace i,j\rbrace}\) and \(\dist(S_1,S_2) = \angle(S_1,S_2)\).

For each \(I \in T_{1} \setminus T_{i,j}\) notice that \(i \in I\) and \(j \notin I\).   Setting \(W_I = y'_{I \setminus \lbrace i\rbrace}\) we have
\((y_1)_I = W_I + S_1\) and \((y_2)_I = W_I + S_2\) from which it follows that 
\[\dist((y_1)_{I}, (y_2)_{I}) = \angle(\pi_{W_I^\perp}(S_1),\pi_{W_I^\perp}(S_2)) \le \angle(S_1,S_2),\]
where \(\pi_{W^\perp}:\R^d \to W^{\perp}\) is the orthogonal projection onto  \(W^{\perp}\).
By Lemma \ref{fiberdistancelemma},
\[\dist(y_1,y_2) = \sum\limits_{I \in T_1 \setminus T_{i,j}}\dist((y_1)_{I}, (y_2)_{I}) \le 2^d \angle(S_1,S_2).\]

Because \(y'\) is a configuration, the minimum over \(I \in T_1 \setminus T_{i,j}\) of the angle between \(W_I\) and \(S_1+S_2 = y'_{\lbrace i,j\rbrace}\) is positive.
It follows that there exists \(c > 0\) which depends only on \(y'\) such that 
\[\dist((y_1)_{I}, (y_2)_{I}) = \angle(\pi_{W_I^\perp}(S_1),\pi_{W_I^\perp}(S_2)) \ge c\angle(S_1,S_2),\]
for all \(I \in  T_{1} \setminus T_{i,j}\).

Notice that, since \( T_1 \setminus T_{i,j} = \{ I : I \ni i {\textrm { and }} I \not \ni j\}, \)  \(T \setminus T' \subset T_1 \setminus T_{i,j}\), and therefore
\[\dist(x_1,x_2) = \sum\limits_{I \in T \setminus T'}\angle(\pi_{W_I^\perp}(S_1),\pi_{W_I^\perp}(S_2)) \ge 2^{-d}c\dist(y_1,y_2).\]
It follows that \(\pi_{T_1,T}\) is a bilipschitz homeomorphism between \(\pi_{T_1,T_{i,j}}^{-1}(y')\) and \(\pi_{T,T'}^{-1}(x')\), as claimed.
\end{proof}

For \(x \in \X_T ,\)  \(T \overset {1}{\prec} T'\) admissible topologies, lemma \ref{coordinate} yields a Lipschitz homeomorphism between \(\pi_{T,T'}^{-1}(x')\) where \(x' = \pi_{T,T'}(x)\)  and \((-\frac{\pi}{2},\frac{\pi}{2})\) and the Lipschitz constant depends on \(x\). Combining with lemma \ref{lipschitzlemma} yields
\begin{corollary}\label{factors2} Assume the diagram of projections \(\begin{matrix} T & \longrightarrow & S  \\
\downarrow 1 & & \downarrow 1 &   \\
T' & \longrightarrow & S'  
\end{matrix}\) commutes and \(i,j\) are such that \(T(i) = T'(i) \setminus \lbrace j\rbrace\) and \(S(i) = S'(i) \setminus \lbrace j \rbrace\). Then, for \( x \in \X_T,\) there is  a bilipschitz homeomorphism between \( (\pi _{T,T'})^{-1} (\pi _{T,T'}) (x) \) and 
 \( (\pi _{S,S'})^{-1} (\pi _{T,S'}) (x) \). The Lipschitz constant depends only on \( x\). \end{corollary}

\section{Applications of Oseledets theorem}\label{section:oseledets}

\subsection{Oseledets multiplicative ergodic theorem}

We review in this section some applications of Oseledets multiplicative ergodic theorem to random walks on matrices. 
Let $\mu $ be a probability measure on the group $G= SL_d(\R)$ of $d\x d$  matrices with determinant 1. Let $(\Om, m ) = (G^\Z, \mu ^\Z)$ be the probability space of independent trials of elements of $G$ with distribution $\mu$, $\s $ the shift transformation on $\Om$.
 Let $\om \in \Om $ be the sequence \( (g_n) _{n \in \Z} \). We denote  \(g_n\) the mapping $\om \mapsto g_n(\om ) $ that associates to \( \om \in \Om \) its coordinate \( g_n \in G\). In particular \( g_0(\om )\) defines a cocycle with values in $SL(d,\R)$ over the ergodic system $(\Om, m; \s).$ Oseledets  multiplicative ergodic theorem gives
 
 \begin{theorem}[\cite{oseledets}]\label{oseledets} With the above notations, assume that \( \mu \in \MM(d)\) and let 
 $\chi_1 >\chi _2 > \ldots >\chi _d  $  with $\sum _i  \chi_i = 0 $ be the Lyapunov exponents (cf. equation (\ref{exponents})).
 Then,  for $m$-a.e. $\om \in \Om,$ there exists a direct decomposition of $\R^d$ into $d$ one-dimensional spaces 
 \[ \R^d \; = \; E_1 (\om) \oplus E_2 (\om) \oplus  \ldots \oplus E_d (\om )  \;\; \; \; {\textrm {such that }} \]
 \begin{enumerate}
 \item a vector $v \not = 0 $ belongs to $E_i(\om) $ if, and only if,
 \[  \lim\limits _{n \to +\infty } \frac{1}{n} \log |g_{n-1}(\om) \circ \ldots \circ g_0(\om ) \,v| =  \lim\limits _{n \to -\infty } \frac{1}{n} \log |(g_{n}(\om ))^{-1} \circ \ldots \circ (g_{-1}(\om ))^{-1}\, v| \; = \; \chi _i\]
 \item for $ m$-a.e. $\om $, all $i$, $ \lim\limits _{n \to \pm \infty } \frac{1}{|n|} \log| \sin \angle (E_i(\s ^n \om ), \sum_{j \neq i} E_j(\s ^n \om ) ) | \; = \; 0.$
 \end{enumerate}
 \end{theorem}
 
 
  By our assumption \( \mu \in \MM(d),\) the exponents are indeed pairwise distinct and  the spaces $E_i$ are one-dimensional. The directions $E_i (\om) $ are defined $m$-a.e. and are called {\it {Oseledets directions}}.  For two distinct dimensional subspaces $E,E' \subset \R^d$, 
 $|\sin\angle(E,E')| $ is defined by $|\sin\angle(E,E')| = \inf\limits_{v \not = 0 \in E, v' \not = 0 \in E'} \frac{ |v\wedge v'|}{|v| |v'|}. $ 
  
 The set $\Om_{reg}$ of points in $\Om$ such that properties 1 and 2 of Oseledets theorem hold is called the set of regular points. The set $\Om_{reg}$ is  $\s$-invariant, measurable  and  has full measure in $\Om$. Observe that by characterization 1, for all $i$, the mapping $\om \mapsto E_i(\om )$ is measurable  on $\Om _{reg}$ and we have \[ E_i (\s \om )\; = \; g_0 E_i (\om ).\]
  Let $i = 1, \ldots, d.$ We write $U_i(\om) := \oplus_{j=1}^i E_i(\om) $ for the unstable spaces, $U'_i(\om) := \oplus_{j=d-i+1}^d E_j(\om) $ for the stable spaces.\footnote{cf. notations of the Introduction.} 
   Since the exponents are distinct, for \( \om \in \Om _{reg}\),  the flags \(E_-(\om)\) given by  \(\{0\} = U_0  \subset U_{1}(\om) \subset \ldots  \subset U_d = \R^d \) and \( E_+ (\om )  \) given by \( \{0\} = U'_0  \subset U'_{1}(\om) \subset \ldots \subset U'_{d-1} (\om) \subset U'_d = \R^d \) are in general position.
  
An important classical observation is the following:
\begin{proposition}\label{independence} For all $i,$ the mappings $\om \mapsto U_i (\om)$ are measurable with respect to the $\s$-algebra generated by $(g_n)_{n \leq -1}$; for all $i',$ the mappings $\om \mapsto U'_{i'} (\om)$ are measurable with respect to the $\s$-algebra generated by $(g_n)_{n \geq 0}.$  In particular, $E_-$ and $E_+$ are independent. \end{proposition}
\begin{proof} It suffices to prove that for any $i$, $U_i$ depends only on $\{g_n\}_{n \leq -1}$. We claim that, for $\om \in \Om_{reg}, $  $U_i  = \{ v : \limsup _{n \to +\infty} \frac{1}{n}  \log |g_{-n}^{-1} \circ \ldots \circ g_{-1}^{-1}\, v| \leq - \chi _i \}.$ This shows that $U_i$ is completely determined when one knows $\{g_n\}_{n \leq -1}$. To prove the claim, observe that $\{ v : \limsup _{n \to +\infty} \frac{1}{n}  \log |g_{-n}^{-1} \circ \ldots \circ g_{-1}^{-1}\, v| \leq - \chi _i \}$ is a vector space that contains $E_j(\om ) , j\leq i$ by definition. It is exactly $U_i (\om)$ since any vector that 
has a nonzero component in one of the $E_\ell (\om ), \ell >i$ satisfies $\limsup _{n \to +\infty} \frac{1}{n}  \log |g_{-n}^{-1} \circ \ldots \circ g_{-1}^{-1}\, v| \geq - \chi _\ell > -\chi _i.$
One verifies in the same way that $ U'_{i'}(\om ) = \{ v : \limsup _{n \to +\infty} \frac{1}{n}  \log |g_{n} \circ \ldots \circ g_0\, v| \leq  \chi _{i'}\}$. \end{proof} 

 Let $Q$ be a partition of $\{0,1, \ldots , d\}$ into intervals, $Q = \{ q_0 = 0  <q_1 < \ldots < q_k = d\}.$ Write  $U_Q(\om ) \in \F_Q$ for the $Q$-flag \[  U_Q (\om ) := \{0\} = U_0  \subset U_{q_1}(\om) \subset \ldots \subset U_{q_{k-1}} (\om )  \subset U_d = \R^d. \]  The set $\GG _i$ of $i$-dimensional subspaces is identified with $\F_{\{0<i<d\}}.$

Since $g_0$ is independent of $U_i(\om ), U_Q (\om) $ and the distribution of $g_0$ is $\mu$, the distribution of $U_i$ (resp. $U_Q$) is a measure on $\GG _i $ (resp. $\F_Q$) which is stationary under the action of $(G, \mu ) $. 
By uniqueness, the distribution of $U_i(\om) $ is $\nu_{\{0<i<d\}},$
 the distribution of $U_Q(\om ) $ is $\nu _Q.$
   Similarly, the distribution of $E_- $ is a stationary measure $\nu' $ for the action of $\mu'$ on $\F .$  For all partition $Q$  of $\{0,1, \ldots , d\}$ into intervals, the distribution of $U'_Q(\om)$ is the stationary measure $\nu'_Q := (\pi _Q)_\ast \nu ' $ for the action of \(\mu '\) on $\F_Q.$
   
   \begin{remark}\label{stable} We do not know a priori that the measure \( \mu '\) has a unique stationary measure; in all the paper, we use the distribution \(\nu '\)  of stable flags 
as stationary measure for the action of  \( \mu '\) on \(\F\).
\end{remark}

  Let \(\Om_+ := (g_n) _{n \geq 0} \) (respectively \( \Om_- :=  (g_n) _{n \leq -1} \)) be the space of one-sided sequences of elements of $G$, $m_+$ (respectively $m_-$) the product measure with $g_k$ of distribution $\mu$ for all $k \geq 0$ (respectively for all $k<0$), $\s $ the shift transformation. Then, by proposition \ref{independence}, \( E_- \) (respectively \(E_+\)) can be seen as a mapping from \( \Om _-\) (respectively \( \Om _+\)) into \(\F\). By  Oseledets theorem \ref{oseledets},  for almost every $\om $, the pair \( E(\om ) := (E_- (\om_-), E_+(\om_+))\) belongs to \(\F^{(2)}\).

We recall in our notations the key Furstenberg result
 \begin{theorem}[\cite{furstenberg1963}]\label{FK} Assume \( \mu \in \MM(d).\) Let $f \in \F, f = \{ \{0\} \subset U_1(f) \subset \ldots \subset U_{d-1}(f) \subset \R^d \}.$ For $m$-a.e. $\om \in \Om_+$, all $j, j=1, \ldots, d,$
 \[ \lim\limits _{n \to +\infty } \frac{1}{n} \log  |\det _{U_j(f)} (g_{n-1} \circ \ldots \circ g_0) | \; = \; \sum _{i\leq j} \chi _i .\]
\end{theorem}
\begin{proof} Under our hypotheses, the distributions of all exterior products $\wedge _{i=1}^j g$ satisfy the conditions of Theorem 8.5 in \cite{furstenberg1963}.
\end{proof}
We used these relations  in the introduction to define  the exponents \(\chi _j, j = 1,\ldots, d,\) in general   for \(\nu \) extremal. 
Since \( \mu \in \MM(d),\) the stationary \( \nu \) measure is extremal and thus the skew product \( (\Om _+ \x \F, m_+\otimes \nu )\) is ergodic for the transformation \( \widehat \s: \widehat \s (\om_+, f ) = (\s \om _+ , g_0(\om ) f) \). We can write 
\begin{eqnarray*}\sum _{i\leq j} \chi _i &=&  \lim\limits _{n \to +\infty } \frac{1}{n}  \log  |\det _{U_j(f)} (g_{n-1} \circ \ldots \circ g_0) | \\
&=& \lim\limits _{n \to +\infty }  \frac{1}{n} \sum _{k=1}^{n-1} \log |\det_{U_j (\widehat \s^k (\om _+, f))}(g_0(\s^k \om ))| .\end{eqnarray*}
The last line converges to \( \int \log |\det_{U_j (f)} (g) | \, d\mu (g) d\nu (f)\) by the ergodic theorem, so that 
\begin{equation}\label{exponents;nonext} \sum _{i\leq j} \chi _i  \; = \;  \int \log |\det_{U_j (f)} (g) | \, d\mu (g) d\nu (f).\end{equation}
 We {\it {define}} the  Lyapunov exponents by equation (\ref{exponents;nonext}) when the measure \( \nu \) is not extremal.
 When the distribution of \( E_-(\om ) \) is \( \nu \) (in particular, for \( \mu \in \MM(d)\)) and since \( g_0 \) is independent of  \( E_-(\om ) ,\)  then  (\ref{exponents;nonext})  can also be written, for all $j, j=1, \ldots, d,$
\[ \sum _{i\leq j} \chi _i  \; = \; \int\limits_{\Omega}\log\left(|\det_{U_j(\omega)}(g_0(\omega))|\right) dm(\omega).\]
Using property 2 in Oseledets multiplicative ergodic theorem \ref{oseledets}, we get, in that case, 
for any subset \(I \subset \lbrace 1,\ldots,d\rbrace\), setting \(V_I(\omega) = \bigoplus\limits_{k \in I}E_k(\omega)\), 
 \begin{equation}\label{exponents2} \int\limits_{\Omega}\log\left(|\det_{V_I(\omega)}(g_0(\omega))|\right) dm(\omega) = \sum\limits_{k \in I}\chi_k.\end{equation}

\subsection{Approximation of partial Oseledets configurations}

Given an admissible topology \(T\) we define \(E_T(\omega) = F_T(E_-(\omega),E_+(\omega))\) where \(F_T\) is defined in section \ref{admissibletopologysection}.

Let \(T \prec T'\) be admissible topologies.  We will extend the configuration \(E_{T'}(\omega)\) to a configuration \(\widehat{E}(\omega)\) defined on \(T\) by applying a deterministic function to \(E_{T'}(\omega)\).

For this purpose, for each \(i = 1,\ldots,d\), we let \(\widehat{E}_i(\omega)\) be the one dimensional subspace of \(E_{T'}(\omega)_{T'(i)}\) 
which is perpendicular to \(E_{T'}(\omega)_{T'(i) \setminus \lbrace i\rbrace}\).   The configuration \(\widehat{E}(\omega)\) is defined by letting
\[\widehat{E}(\omega)_I = \sum\limits_{i \in I}\widehat{E}_i(\omega)\]
for all \(I \in T\).

\begin{proposition}
For  \(m\)-a.e. \(\omega\), and all \(I \in T'\) one has \(\widehat{E}(\omega)_I = E_{T'}(\omega)_I\).
\end{proposition}
\begin{proof}
 It suffices to verify that \(\widehat{E}(\omega)_{T'(i)} = E_{T'}(\omega)_{T'(i)}\) for \(i = 1,\ldots,d\).
 
 When \(i = d,\) this is trivial since \(E_{T'}(\omega)_{T'(d)} = E_d(\omega)\) and therefore \(\widehat{E}(\omega)_{T'(d)} = E_d(\omega)\) as well.
 
 Suppose that the claim is true for \(i+1,\ldots,d\), then 
 \[\widehat{E}(\omega)_{T'(i)\setminus \lbrace i\rbrace} = \sum\limits_{j \in T'(i)\setminus \lbrace i\rbrace}\widehat{E}(\omega)_{T'(j)} = \sum\limits_{j \in T'(i)\setminus \lbrace i\rbrace}E(\omega)_{T'(j)} = E(\omega)_{T'(i) \setminus\lbrace i\rbrace}.\]
 
 It follows that \(\widehat{E}_i(\omega)\) is complementary to the codimension one subspace \(\widehat{E}(\omega)_{T'(i) \setminus \lbrace i\rbrace}\) within \(E(\omega)_{T'(i)}\).  Taking the sum one obtains \(\widehat{E}(\omega)_{T'(i)} = E_{T'}(\omega)_{T'(i)}\), as claimed.
\end{proof}

We now show that using the extension above one may approximate \(E_T(\omega)\) using only \(E_{T'}(\omega)\) and \(g_{-1}(\omega),\ldots,g_{-n}(\omega)\) up to an error which is exponentially small as \(n \to +\infty\).

\begin{proposition}\label{configurationapproximationlemma}
 For \(m\)-a.e. \(\omega\) when \(n \to +\infty\) one has
 \[\dist(g_{-1}(\omega)\cdots g_{-n}(\omega)\widehat{E}(\sigma^{-n}(\omega))_T,E_{T}(\omega)) \le \exp(-\chi n + o(n)),\]
 where \(\chi = \min\limits_{(i,j) \in D_{T,T'}}\chi_i - \chi_j\).
\end{proposition}
\begin{proof}
To simplify calculations for each \(i = 1,\ldots,d\) let \(w_i(\omega)\) be a unit vector generating \(E_i(\omega)\), and \(v_i(\omega)\) be a unit vector generating \(\widehat{E}_i(\omega)\).   Write
\[v_i(\omega) = \sum\limits_{j \in T'(i)}a_{i,j}(\omega)w_j(\omega).\]
Since \(v_i(\sigma^{-n}\omega)\) is a unit vector  \(|a_{i,j}(\sigma^{-n}\omega)| \le e^{o(n)}\) for all \(j\), all \(n\), and \(m\)-a.e. \(\omega\). Furthermore, since the angle between \(E_i(\sigma^{-n}\omega)\) and \(\sum\limits_{j \in T'(i) \setminus \lbrace i \rbrace}E_j(\sigma^{-n}\omega)\) is at least \(e^{-o(n)}\), we obtain \(|a_{i,i}(\sigma^{-n}\omega)| \ge e^{-o(n)}\) when \(n \to +\infty\) for \(m\)-a.e. \(\omega\).

It suffices to show that
 \[\dist(g_{-1}(\omega)\cdots g_{-n}(\omega)\widehat{E}(\sigma^{-n}(\omega))_{T(i)},E_{T}(\omega)_{T(i)}) \le \exp(-\chi n + o(n)),\]
when \(n \to +\infty\) for all \(i = 1,\ldots,d\).

The claim is trivially true when \(i = d\).

Suppose that the claim is true for \(i+1,\ldots,d\), so in particular one has that
 \[\dist(g_{-1}(\omega)\cdots g_{-n}(\omega)\widehat{E}(\sigma^{-n}(\omega))_{J},E_{T}(\omega)_{J}) \le \exp(-\chi n + o(n)),\]
as \(n \to +\infty\) where \(J = T(i) \setminus \lbrace i\rbrace\).

Let \(z_n(\omega)\)  be a unit vector in the intersection of \(\widehat{E}(\sigma^{-n}(\omega))_{T(i)}\)  with the subspace \(W_i(\sigma^{-n}(\omega)),\)
 where \( W_i (\om ) \) is the space generated by \(\{ w_j (\om) , j = i \) and \( j \in T'(i) \setminus T(i) \}. \) We have \( e^{-o(n)} \leq  |<z_n , w_i (\s ^{-n} (\om ) )> | \) and \( \| z_n \| \leq 1.\) If we write 
\[z_n(\omega) = b_{n,i}(\omega)w_i(\sigma^{-n}\omega) + \sum\limits_{j \in T'(i) \setminus T(i)}b_{n,j}(\omega)w_j(\sigma^{-n}(\omega),\]we have \(|b_{n,i}(\omega)| \ge e^{-o(n)}\) while \(|b_{n,j}(\omega)| \le e^{o(n)}\) for all \(j\).

To conclude notice that
\[g_{-1}(\omega)\cdots g_{-n}(\omega)z_n(\omega) = e^{\chi_i n + o(n)}b_{n,i}(\omega)w_i(\omega) + \sum\limits_{j \in T'(i) \setminus T(i)}e^{\chi_j n+o(n)}b_{n,j}(\omega)w_j(\omega).\]
 
It follows that the angle between \(E_i(\omega)\) and the subspace \(L_n(\omega)\) generated by \(g_{-1}(\omega)\cdots g_{-n}(\omega)z_n(\omega)\)  is at most \(e^{-\chi' n + o(n)}\) where \(\chi' = \min\limits_{j \in T'(i) \setminus T(i)} \chi_i - \chi_j \ge \chi\).
Since \(g_{-1}(\omega)\cdots g_{-n}(\omega)\widehat{E}(\omega)_{T(i)} = L_n(\omega) + g_{-1}(\omega)\cdots g_{-n}(\omega)\widehat{E}(\omega)_{J},\)
the claim follows.
\end{proof}

Assume in the above discussion that  \(T \overset {1}{\prec} T'\)  and that \( i<j \) is such that \(T'(i) = T(i) \cup \{j \} \). Set \( x' = E_{T'} (\om ) \). Then the space \( W_i (\om ) \), generated by \( E_i(\om ), E_j(\om ) \) is a representative of the vector space  \(V = x'_{T'(i)}/x'_{T'(i) \setminus \lbrace i,j\rbrace}\) discussed in Lemma \ref{coordinate}.
Let \(\pi\:= \pi_{T,T'} \) be the projection from \(\X_{T}\) to \(\X_{T'}\) and consider the  coordinates \(\vf _{x(\om  )}\) given by lemma \ref{coordinate} on \(\pi^{-1}(E_{T'}(\omega)),\) setting \(x (\om)= E_{T}(\omega)\) and \(x' (\om )= E_{T'}(\omega)\). 
 The distance  \(\vf _{x(\om  )}\) on \(W_i(\om)\) is equivalent to the metric \( \dist_{T,T'}^{x'(\om)}\) (see (\ref{fibermetric})) on \( \pi^{-1}(x'(\om))\).
\begin{lemma}\label{coordinateslemma}  Let  \(T \overset {1}{\prec} T'\) be admissible topologies  and  \( i<j \)  such that \(T'(i) = T(i) \cup \{j \} \). Fix \( \beta >0 \) and let \( x_n \in \pi^{-1}(x'(\s ^{-n}\om))\) satisfy \(\vf_{x(\s ^{-n} \om)} (x_n)) \leq \beta\). Then for  \(m\)-a.e. \( \om \), as \( n \to \infty ,\) one has 
 \begin{equation}\label{northsouth}  \dist_{T,T'}^{E_{T'}(\om )} (g_{-1}(\om) \circ \ldots \circ g_{-n}(\om ) (x_n) , {E_{T}( \om )} ) \; \leq \; \exp(- \chi _{T,T'} n +o(n) ) . \end{equation} \end{lemma}
 \begin{proof} By theorem \ref{oseledets}.2, the distance from \( E_T (\s ^{-n} \om) \) to \(  {E_{T'}(\s ^{-n} \om )}\) is at least \( \exp (-o(n)) .\) Therefore, \(  \dist_{T,T'}^{E_{T'}(\s^{-n}(\om ))} ( (x_n) , {E_{T'}(\s ^{-n} \om )} )  \geq   \exp (-o(n)) \) as well.
 Following the proof of proposition \ref{configurationapproximationlemma}, we have that the point \( x_n \) 
satisfies \[\dist_{T,T'}^{E_{T'}(\om )} (g_{-1}(\om) \circ \ldots \circ g_{-n}(\om ) (x_n) , {E_{T}( \om )} ) \; \leq \; \exp(- \chi _{T,T'} n +o(n) ) .\]
The lemma follows. \end{proof}
 
 \section{Proof of theorem  \ref{ent/exp1} }\label{section:ent/exp}
 Theorem \ref{ent/exp1} states that some entropy is estimated from above by exponents. For random walks on matrices, this is a fundamental observation of Furstenberg (\cite{furstenberg1963}). Theorem \ref{ent/exp1} and its proof are one more variant of the original proof: one shows equality for a big family of random walks on the same group and one approximates using this family.  The exponents are continuous and the entropy has some upper semi-continuous properties. This should be sufficient for the astute reader, but we will give a detailed proof anyway. In particular, it gives some a priori quasi-invariance of stationary measures (see Corollary \ref{finiteentropy}). We state and prove the generalisation of theorem \ref{ent/exp1} to the action on \(\F_Q\), for any \(Q\) partition de \( \{0, 1, \ldots ,d\} \).
 \begin{theorem}\label{ent/expQ1} With the above notations, for any partition $Q$ of $\{0,1, \ldots , d\}$ into intervals, there exists a stationary measure \( \nu _Q\) on \(\F_Q\) such that 
\begin{equation}\label{ent/expQ} h(\F_Q, \mu, \nu_Q )  \; \leq \; \sum _{i,j: \ell _Q(i) < \ell _Q(j) } \chi_i - \chi _j .\end{equation}
If there is equality in (\ref{ent/expQ}), then the measure \( \nu _Q \) is exact dimensional with  dimension \( \dim \F_Q\). \end{theorem}
For \( \mu \in \MM(d)\), the stationary measure \( \nu _Q\) is unique and thus satisfies (\ref{ent/expQ}).

\subsection{Mollification of \(\mu\)}

For each \(n=1,2,3,\ldots\) fix a probability \(\lambda_n\) with a smooth positive density with respect to Haar measure on the orthogonal group of \(\R^d\), in such a way that \(\lim\limits_{n \to +\infty}\lambda_n = \delta_{Id}\) where \(\delta_{Id}\) is the point mass at the identity.

Let \(\mu_n = \lambda_n * \mu\) so one has, for all continuous \(h:G \to \R\)
\begin{equation}\label{approximation} \int\limits_{G} h(g) d\mu_n(g) = \int\limits_{G} \int h(rg) d\lambda_n(r) d\mu(g).\end{equation}

Let \(\eta\) be the orthogonally invariant probability on \(\F\).
\begin{lemma}
 For each \(n\) there is a unique \(\mu_n\)-stationary probability \(\nu_n\) on \(\F\).  Furthermore, \(\nu_n\) has a continuous positive density with respect to \(\eta\).
\end{lemma}
\begin{proof}
 For any \(\mu_n\)-stationary probability we have
 \[\nu_n = \mu_n * \nu_n = (\lambda_n * \mu) * \nu_n = \lambda_n * (\mu * \nu_n).\]
  Since, for any probability \(m\) on \(\F\), the convolution \(\lambda_n*m\) has a continuous positive density with respect to \(\eta\) it follows that any \(\mu_n\)-stationary probability has this property.
 However, any two distinct extremal \(\mu_n\)-stationary probabilities must be mutually singular.  This implies that \(\nu_n\) is unique as claimed.
\end{proof}

Endow \( \mathcal M (G \times \F) \)  with the topology of convergence over continuous function \( \vf \) on $(G\times \F)$ with \(|\vf (g,f) |\leq C \log \|g\| \) for some constant $C$. 
\begin{lemma}
 In \( \mathcal M(G\times \F),\) any limit  \(\lim\limits_{n \to +\infty}(\mu_n \times \nu_n) \) is of the form \( \mu \times \nu \) for some stationary \( \nu \) on \(  \F\).
\end{lemma}
\begin{proof} 
We have \( \lim\limits_{n \to \infty } \mu_n = \mu \) and, since any weak*-limit of \(\nu _n \) is \( \mu \)-stationary,  \( \lim\limits_{n \to \infty }( \mu_n \times \nu_n ) \) is of the form \( \mu \times \nu \) for some \( \mu \)-stationary \( \nu \) on \(  \F\).
A priori, we have convergence against only continuous functions with compact support on \( G\times \F\). We want to ensure that there is also convergence in \( \mathcal M(G\times \F).\) We may assume that \( \{ \mu_n \times \nu_n \}_{n \to \infty} \) is the converging sequence.

By  Skorohod's representation theorem,  there exists a probability space \((P,\mathcal{E},\P)\) and random elements on this space such that
\[\lim\limits_{n \to +\infty}(A_n,F_n) = (A,F)\]
almost surely, where \((A,F)\) has distribution \(\mu \times \nu\) and \((A_n,F_n)\) has distribution \(\mu_n \times \nu_n\) for each \(n\).
Let \( \vf \) be a continuous function \( \vf \) on $(G\times \F)$ with \(|\vf (g,f) |\leq C \log \|g\| \) for some constant $C$. 

With these auxiliary random elements we calculate (denoting integration with respect to \(\P\) by \(\E{}\) as usual)
\[\lim\limits_{n \to +\infty}\int \vf (g,f) d\mu _n (g) d\nu _n (f)  =  \lim\limits_{n \to +\infty}\E{\left(\vf (A_n,F_n)\right)}.\]

Observe that, since left composition with an orthogonal transformation does not change the norm a linear mapping, 
 the random variables \(\left(|\vf (A_n,F_n)|\right)_{n= 1,2,\ldots }\) are uniformly integrable.  It follows that 
 \[ \lim\limits_{n \to +\infty}\E{\vf (A_n,F_n)} \; = \; \E{\vf (A,F)}. \]
\end{proof}

\begin{lemma}\label{limitschi}
 For each \(n\) and each \(i = 1,\ldots,d\) let \( \chi _{i,n}\) be such that \[ \sum _{j\leq i } \chi_{j,n} := \int\limits_{G} \int\limits_{\F} \log\left(|\det_{U_i(f)}(g)|\right) d\nu_n(f)d\mu_n(g)\].
 Then one has \(\lim\limits_{n \to +\infty}\chi_{i,n} = \chi_{i}\) for each \(i = 1,\ldots,d\).
\end{lemma}
\begin{proof} Recall the definition of the Lyapunov exponents by equation  (\ref{exponents;nonext}). We may apply the preceding lemma since \( |\det_{U_i(f)}(g)| \leq d! |\log g|. \) 
\end{proof}

\subsection{Quasi-independence and mutual information}\label{mutualinformation1}
 Let \( (\Om , \P)\) be a probability space. Given \(X:\Omega \to \X\) and \(Y:\Omega \to \mathcal{Y}\) taking values in Polish spaces, we say \( X \) and \( Y\) are quasi-independent if the distribution \( \P_{X,Y} \) of \(X,Y\) is absolutely continuous with respect to the product \( \P_X\otimes \P_Y \) of the distributions of \(X\) and \(Y\). We write \( f(x,y) \) or \( f_\P (x,y) \) for the Radon-Nikodym derivative \( \frac{d\P_{X,Y} }{d(\P_X\otimes \P_Y)}.\) In that case, the disintegration of the measure \( \P_{X,Y} \)  with respect to the projections over \( X \) and \( Y\) is respectively \( \P^x =  f(x,y) \P_Y \) and \( \P^y =  f(x,y) \P_X \).
 
 For  \( X \) and \( Y\) as above, define the mutual information between \(X\) and \(Y\)  as
\[I(X,Y) = I_{\P}(X,Y) = \sup \sum\limits_{A \in \mathcal{A}}\log\left(\frac{\P_{(X,Y)}(A)}{(\P_{X}\times \P_{Y})(A)}\right)\P_{(X,Y)}(A),\]
where the supremum is over finite Borel partitions \(\mathcal{A}\) of \(\X\times \mathcal{Y}\)\footnote{By convention, \( \log\left(\frac{\P_{(X,Y)}(A)}{(\P_{X}\times \P_{Y})(A)}\right)\P_{(X,Y)}(A) =0 \) if \( \P_{X,Y} (A) = 0,\) the sum is \(  +\infty \)  if there is one \( A \in \mathcal A\) such that \( \P_{X,Y} (A) \not = 0\) and \( (\P_{X}\times \P_{Y})(A) =0.\)}.

Directly from the definition one sees that \(I(X,Y) = I(Y,X)\).   By Jensen inequality \(0 \le I(X,Y) \le +\infty\) with equality to \(0\) if and only if \(X\) and \(Y\) are independent.   If \(X\) takes countably many values and has finite entropy \(H(X)\) in the sense of \cite{shannon}  one has \(I(X,Y) \le H(X)\). It was shown in \cite{dobrushin} that \(I(X,Y)\) is the supremum over any sequence of partitions which generate the Borel \(\sigma\)-algebra in \(\mathcal{X}\times \mathcal{Y}\). 

It was shown in \cite{gelfand-yaglom} and \cite{perez} that if \(I(X,Y) < +\infty\) then \(X\) and \(Y\) are quasi-independent  and 
\[ I(X,Y) \; = \; \int \log f(x,y)\, d\P_{X,Y} (x,y) \; =\; \int f(x,y)\log f(x,y)  \,d\P_X(x)d\P_Y(y) .\]

Let \(Q\) be a partition of \(\lbrace 0,1,\ldots d\rbrace\) into intervals and \(\pi _Q:\F \to \F_Q\) the projection from \(\F\) into the corresponding space of partial flags.
Consider the stationary measure  \(\nu_Q = (\pi_Q)_*(\nu)\) that is  a limit  of the measures \(\nu_{Q,n} = (\pi_Q)_*\nu_n\) as \( n \to \infty.\)

We define the probability \(\P\) on \(G \times \F_Q\) so that \[ \int h(g,f) d\P(g,f) = \int\limits_{G} \int\limits_{\F_Q} h(g,gf) d\nu_Q(f) d\mu(g)\] and the mutual information \(I \)  between the coordinate projections on \(G \times \F_{Q}\) with respect to \(\P\). Inequality  (\ref {ent/expQ}) for \(\nu _Q\) and thus theorem \ref {ent/expQ1} will follow directly from 
\begin{proposition}\label{Ifinite} With the above notations, \[ I \le \sum\limits_{\ell_Q(i) < \ell_Q(j)}\chi_i - \chi_j.\] \end{proposition}

Indeed, by proposition \ref{Ifinite}, the variables \(G\) and \( \F_Q\) are quasi-independent under $\P$, with density \( \frac{d g_\ast\nu _Q}{d \nu _Q} (f) \). In particular,
\[ h(\F_Q,\mu,\nu_Q) = \int \log\left(\frac{dg_*\nu_{Q}}{d\nu_{Q}}(f)\right) dg_*\nu_Q(f)d\mu(g) = I \le \sum\limits_{\ell_Q(i) < \ell_Q(j)}\chi_i - \chi_j,\]
which is the statement of (\ref {ent/expQ}).

\begin{proof}[Proof of proposition \ref{Ifinite}]
We analogously define the probability \(\P_n\) on \(G \times \F_Q\) so that \[ \int h(g,f) d\P_n (g,f) = \int\limits_{G} \int\limits_{\F_Q} h(g,gf) d\nu_{Q,n}(f) d\mu_n(g)\] and the mutual information \(I _n\)  between the coordinate projections on \(G \times \F_{Q}\) with respect to \(\P_n\). 
\begin{lemma}\label{uppersemicont.}
 In the above context, \(I \le \liminf\limits_{n \to +\infty}I_n\).
\end{lemma}
\begin{proof}
By Dobrushin theorem the supremum may be taken over partitions whose sets belong to any generating set for the Borel \(\sigma\)-algebra.   Therefore we may consider only partitions into sets satisfying \(\lim\limits_{n \to +\infty}\P_n(A) = \P(A)\).   The inequality follows immediately.
\end{proof}

\begin{lemma}\label{abs.cont}
 For each \(n\) one has \(I_n = \sum\limits_{\ell_Q(i) < \ell_Q(j)}\chi_{i,n} - \chi_{j,n}\).
\end{lemma}
\begin{proof}
Let \(\varphi_n\) be the density of \(\nu_{Q,n}\) with respect to the rotationally invariant probability \(\eta_{Q}\) on \(\F_Q\).

By the Gelfand-Yaglom-Perez theorem one has
\begin{align*}
I_n &= \int\limits_{G \times \F_Q} \log\left(\frac{dg_*\nu_{Q,n}}{d\nu_{Q,n}}(f)\right)d\P_n(g,f)
\\ &= \int\limits_{G}\int\limits_{\F_Q}\log\left(\frac{dg_*\nu_{Q,n}}{d\nu_{Q,n}}(gf)\right)d\nu_{Q,n}(f)d\mu_n(g)
\\ &= \int\limits_{G}\int\limits_{\F_Q}\log\left(\frac{\varphi_n(f)}{\varphi_n(gf)}\frac{dg_*\eta_{Q}}{d\eta_{Q}}(gf)\right)d\nu_{Q,n}(f)d\mu_n(g)
\end{align*}

By \(\mu_n\)-stationarity of \(\nu_{Q,n}\) the integrals involving \(\varphi_n\) cancel, and one obtains
\begin{align*}
I_n &= \int\limits_{G} \int\limits_{\F_Q} \log\left(\frac{dg_*\eta_{Q}}{d\eta_{Q}}(gf)\right)d\nu_{Q,n}(f)d\mu_n(g).
\end{align*}
We have the following explicit formula for \(\frac{dg_*\eta_{Q}}{d\eta_Q}(gf)\)
\begin{proposition}\label{explicit}
For \(Q = \lbrace q_0 = 0 < q_1 < \cdots < q_k = d\rbrace\) and \(\eta = \eta_Q\) the rotation invariant measure then
\[\frac{dg\eta}{d\eta}(gx) = \frac{|\det_{U_{q_1}(x)} (g)|^{q_2} |\det_{U_{q_2}(x)} (g)|^{q_3-q_1} \cdots |\det_{U_{q_{k-1}}(x)} (g)|^{d-q_{k-2}}}{|\det (g)|^{q_{k-1}}}.\]

In particular on the space of full flags one has
\[\frac{dg\eta}{d\eta}(gx) = \frac{|\det_{U_{1}(x)} (g)|^{2} |\det_{U_{2}(x)} (g)|^{2} \cdots |\det_{U_{k-1}(x)} (g)|^{2}}{|{\det (g)|^{d-1}}}.\]
\end{proposition}
Proposition \ref{explicit} is proven in the next subsection. Given proposition \ref{explicit}, we may write
\begin{align*}
I_n &= \int\limits_{G} \int\limits_{\F_Q} \log\left(\frac{|\det_{U_{q_1}(f)}(g)|^{q_2}|\det_{U_{q_2}(f)}(g)|^{q_3-q_1} \ldots |\det_{U_{q_{k-1}}(f)}(g)|^{d-q_{k-2}}}{|\det(g)|^{q_{k-1}}}\right)d\nu_{Q,n}(f)d\mu_n(g)
\\ &= \left(\sum\limits_{j = 1}^{k-1} (q_{j+1}-q_{j-1})\sum\limits_{i = 1}^{q_j} \chi_{i,n}\right)    - q_{k-1}\sum\limits_{i = 1}^{d} \chi_{i,n}
\\ &= \sum\limits_{\ell_{Q}(i) < \ell_{Q}(j)}\chi_{i,n} - \chi_{j,n},
\end{align*}
where the last equality follows by direct computation.\end{proof}

Using lemma \ref{limitschi}, proposition \ref{Ifinite} follows. 
\end{proof}
\begin{corollary}\label{finiteentropy} Let \( \mu \in \MM(d), \nu\)  the stationary measure on \( \F\). Then,
\[ h( \F, \mu ,\nu ) \; \leq \; \sum _{0 < i 
<j \leq d} \chi_i - \chi _j  \; < \; + \infty .\]
In particular, for \(\mu \)-a. e. \(g  , g_\ast \nu \) is absolutely continuous with respect to \( \nu \) and the function \( \log \frac {dg_\ast \nu }{d\nu } \) is integrable with respect to \( g_\ast \nu .\) \end{corollary}
Indeed, we have \( \int\limits_{G}\int\limits_{\F}\log\left(\frac{dg_*\nu}{d\nu}(f)\right)\,dg_*\nu(f)\,d\mu(g) < +\infty .\) 

\subsection{Proof of proposition \ref{explicit}}

\begin{lemma}
Let \(Q = \lbrace 0 < i < d\rbrace\) and \(\eta = \eta_Q\) be the unique rotationally invariant probability on \(\F_Q\) the Grasmannian manifold of \(i\)-dimensional subspaces of \(\R^d\).  Then
\[\frac{dg\eta}{d\eta}(gx) = \frac{|\det_x (g)|^d}{|\det (g) |^i},\]
for all \(g \in \GL_d(\R)\).
\end{lemma}
\begin{proof}
Let \(\pi_x:\R^d \to x\) be the orthogonal projection onto \(x\) and \(\pi_{\R^d/x}:\R^d \to \R^d/x\) the canonical projection.  
The quotient space \(\R^d/x\) is endowed with the inner product such that the projection is an isometry when restricted to the orthogonal complement of \(x\).

Since \(\eta\) is invariant under orthogonal transformations we may compose \(g\) with such transformations on both sides.   In particular we may assume that \(g(x) = x\) and therefore \(g\) induces a transformation on \(\R^d/x\) which we denote by \(g\) as well. 

We may further assume that:
\begin{enumerate}
 \item There is an orthogonal basis \(v_1,\ldots,v_i\) of \(x\) and positive eigenvalues \(\sigma_1,\ldots,\sigma_i > 0\) such that \(gv_k = \sigma_k v_k\) for \(k = 1,\ldots,i\).
 \item There is an orthogonal basis \(w_1,\ldots,w_{d-i}\) of \(\R^d/x\) and positive eigenvalues \(\mu_1,\ldots,\mu_{d-i}> 0\) such that \(gw_k = \mu_k w_k\) for \(k = 1,\ldots,d-i\).
\end{enumerate}

Notice that \(\sigma_1\cdots \sigma_i = |\det_x (g)|\) while \(\mu_1\cdots \mu_{d-i} = |\det (g) |/|\det_x (g)|\).

A local parametrization of \(\F_Q\) around \(x\) is given by identifying each linear mapping \(\varphi:x \to \R^d/x\) with the subspace \(x_\varphi\) such that \(v \in x_\varphi\) if and only if \(\varphi(\pi_x(v)) = \pi_{\R^d/x}(v)\).

Identify each \(\varphi\) with its matrix \(\left(a_{lk}\right)_{k,l}\) where \(\varphi(v_k) = \sum\limits_{l}a_{lk} w_l\).   
With this identification, we define a volume form \(\omega\) on \(\F_Q\) such that in any coordinates constructed as above one has \(\omega(0) = \pm\bigwedge_{k,l} da_{lk}\).  
Since \(\omega\) is orthogonally invariant it must define a volume on \(\F_Q\) which is a constant multiple of \(\eta\).

In this chart the action of \(g\) on \(\F_Q\) maps \(\varphi\) to \(g_{\vert \R^d/x}\circ \varphi\circ g^{-1}_{\vert x}\) so that 
\[g_{\vert \R^d/x}\circ \varphi\circ g^{-1}_{\vert x}(v_k) = g\varphi(\sigma_k^{-1}v_k) =  \sum\limits_{l}a_{lk}\frac{\mu_{l}}{\sigma_k}w_l.\]

It follows that the pull-back under \(g\) of \(\omega\) satisfies \(g^*\omega(0) = \prod\limits_{k,l}\frac{\mu_l}{\sigma_k} \omega(0) = \frac{|\det g|^i}{|\det_x g|^d}\omega(0)\), which implies the desired claim.
\end{proof}

\begin{corollary}
 If \(Q\) is obtained by a splitting an interval \(k+1 < k+m\) of \(Q'\) into \(k+1 <  k+i < k+m\) then
\[\frac{dg\eta_{Q,Q'}^x}{d\eta_{Q,Q'}^{gx}}(gx) = \frac{|\det_{U_{k+i}(x)} (g)|^{m}}{|\det_{U_{k}(x)} (g)|^{m-i}|\det_{U_{k+m}(x)} (g)|^i},\]
for all \(g \in \GL_d(\R)\) and all \(x \in \F_Q\).
\end{corollary}
\begin{proof}
The fiber of the projection from \(\F_Q\) to \(\F_{Q'}\) which contains \(x\) is naturally identified with the the \(i\)-dimensional Grasmannian of \(U_{k+m}(x)/U_k(x)\).

The Jacobian of \(g\) as a mapping from \(U_{k+m}(x)/U_k(x)\) to its image is \(|\det_{U_{k+m}(x)} (g)|/|\det_{U_k(x)} (g)|\).  
The Jacobian of the restriction of \(g\) to the subspace of \(U_{k+m}(x)/U_k(x)\) represented by \(U_{k+i}\) is \(|\det_{U_{k+i}} (g)|/|\det_{U_{k}} (g)|\).

The result follows replacing these values in the previous lemma.
\end{proof}
Proposition \ref{explicit}  follows from the previous corollary by splitting \(\lbrace 0 < d\rbrace\) successively into \(\lbrace 0 < q_{k-1} < d\rbrace\), \(\lbrace0 < q_{k-2} < q_{k-1} < d\rbrace\), etc.

 \section{Mutual information}\label{section:mutualinformation}
\subsection{Conditional mutual information}\label{section:cond.mut.info}

 Given \(X:\Omega \to \mathcal{X} , Y:\Omega \to \mathcal{Y}\) and \(Z:\Omega \to \mathcal{Z} \) taking values in polish spaces,
 one may define for \(\P_Z\)-a.e. \(z \in \mathcal{Z} \) the disintegrations \( \P^z_{X\times Y} \) with respect to the projections  on \(\mathcal {Z}\)  and the  mutual information 
\( I_{\P^z}(X,Y)\)
of \(X\) and \(Y\) given \(Z\) using these conditional distributions \(\P^{z}\) of \((X,Y)\) given \(Z\). Define the  {\it {conditional mutual information}} \( I(X,Y\vert Z)\)  of \( X \) and \(Y\) given \(Z\) by  \[ I(X,Y \vert Z) \; := \;  I_{\P^z} (X,Y)  \] and the  {\it {conditional mutual entropy}} \( H(X,Y\vert Z)\)  of \( X \) and \(Y\) given \(Z\) by \[ H(X,Y \vert Z) \; := \; \int  I(X,Y\vert Z) \, d\P_Z(z) .\]

Observe that \(X\) and \(Y\) are conditionally independent given \(Z\) if, and only if, \( I(X,Y \vert Z) = 0 {\textrm{ almost everywhere}}, \) if, and only if, \( H(X,Y \vert Z) = 0  .\)  Observe also that, by definition, \( I(X,Y \vert Z) = I(Y,X \vert Z) \).
If \(W,X,Y,Z\) are measurable functions from \(\Omega\) into Polish spaces then one has the following chain rule
\begin{eqnarray*}  H(W,(X,Y)\vert Z) &=&H(W,X\vert Y,Z) +  H(W,Y\vert Z).\end{eqnarray*}

Let \( \P\) be the measure on \( G \times \F \) introduced in section \ref{mutualinformation1}: for any positive measurable function \(h \) on \( G \times \F \),
\[ \int h(g,f) \,d\P (g, f ) \; = \; \int h(g, gf) \, d\mu (g) d\nu (f).\]
We showed that  \( h(\F, \mu ,\nu ) = I_\P (g, f) < +\infty  .\) 
\begin{lemma}\label{quasiindependence} The measure \(\P\) is the distribution of the  variables \( (g_{-1}, E_- ) \) on \( (\Om, m).\) In particular, \( g_{-1}(\om ) \) and \( E_-(\om ) \) are quasi-independent.\end{lemma} 
\begin{proof} We indeed have, by invariance of $m,$ for any positive  \(h \) on \( G \times \F \) 
\begin{eqnarray*} \int h(g_{-1}(\om ) , E_- (\om )) \, dm(\om ) &=&  \int h(g_{-1}(\s \om ), E_- (\s\om )) \, dm(\om )\\ &= & \int h(g_{0}(\om ), g_0(\om )E_- (\om )) \, dm(\om ) = \int h \, d\P.\end{eqnarray*} \end{proof}
Using proposition \ref{independence}, we may write, for \(m_+\)-a.e. \( \om _+\):
\[ h(\F, \mu ,\nu ) \;=\; I (g_{-1}, E_-) \;= \;  I(g_{-1}, E_- \vert E_+) \;= \; I (g_{-1}, E_-\vert g_0, g_1, \ldots ).\]

Let \(T\) be an admissible topology. Recall that we defined the entropy  \(\kappa _{T} \) by equation (\ref{entropy}) \[ \kappa _{T}  \; := \; \int  \log \frac{ dg_\ast\nu _T^{g^{-1} f'} }{  d\nu _T^{f'} }(g,y)\, dg_\ast\nu _T^{g^{-1} f'}(y)  d\nu '(f') d\mu(g) .\] For \(m\) -a.e. \( \om \in \Om \) write \( E_T (\om ) \in \X_T^{E_+(\om )} \) for \( E_T(\om ) := F_T (E_-(\om ), E_+ (\om ) ) .\) The next three propositions give other useful expressions for \(\kappa _T \).
\begin{proposition} With the above notations, we have
\begin{equation}\label{entropyT} \kappa _T  \; = \; H(g_{-1} , E_T\vert E_+ ) \; <\; +\infty . \end{equation} \end{proposition}
\begin{proof} Observe first that  the RHS of (\ref{entropyT}) is finite because
\[  H (g_{-1} , E_T \vert E_+ ) \leq    H (g_{-1} , E_-\vert E_+) < +\infty. \]
 The distribution of \( g_{-1} (\om ) \)  given \(E_+ (\om )\)  is \(\mu \) and the distribution of \(E_T (\om )\) given \(E_+ (\om )\)  is \(  \nu _T^{ E_+(\om ) } \). Remains to compute the joint distribution of \( (g_{-1}(\om ),E_T(\om ) )\) given \(E_+ (\om )\). 
 We claim that it projects to \(\mu\) with conditional measures given by \(g_{-1}(\omega)_*\nu_T^{g_{-1}(\omega)^{-1}E_+(\omega)}\).
 
  It follows that, for \(m_+\)-a.e. \( E_+ (\om _+),\)
\[ I (g_{-1}(\om ),E_T (\om )  \vert E_+(\om ) ) = \int \log \frac{ dg_\ast\nu _T^{g^{-1} E_+(\om )} }{  d\nu _T^{E_+(\om ) } }(g,y)\,  dg_\ast\nu _T^{g^{-1} E_+(\om )} (y) d\mu(g).\]
By integrating in \(E_+ (\om) \), i.e. in \(f'\) with respect to the measure \( \nu ',\) we find that \(  H (g_{-1} , E_T\vert E_+ ) \) is given by 
\[ H (g_{-1} , E_T\vert E_+ ) \; = \; \int  \log \frac{ dg_\ast\nu _T^{g^{-1} f'} }{  d\nu _T^{f'} }(g,y)\, d\mu(g) dg_\ast\nu _T^{g^{-1} f'}(y)  d\nu '(f'), \] which is the formula defining \(\kappa _T \) in relation (\ref{entropy}).
 
We prove the claim:
firstly, \( g_{-1}(\om) \) is independent of \( E_+(\om )\), so that what remains to compute is the distribution of \( E_T(\om ) \) given \(((g_{-1}(\om ),E_+ (\om ))\). The distribution of   \( E_T(\s ^{-1} \om ) \) given \(((g_{-1}(\om ),E_+ (\om ))\) is by definition \( \nu _T^{E_+(\s ^{-1}\om )} .\) Since \(  E_+(\s^{-1} \om ) = g_{-1}(\om ^{-1})E_+(\om ), \) the distribution of \( E_T(\om ) \) given \(((g_{-1}(\om ),E_+ (\om ))\) is indeed given by \(g_{-1}(\omega)_*\nu_T^{g_{-1}(\omega)^{-1}E_+(\omega)}\).
\end{proof}

\begin{proposition}\label{quasiMarkov} We also have  \[ \kappa _T \;=\;  H (g_{-1} , E_T (\om )  \vert ( E_+(\om ) , g_0, g_1, \ldots) )\;= \;  H (g_{-1} , E_T (\om )  \vert \om _+) .\] \end{proposition} 
\begin{proof} By proposition \ref{independence}, the conditional measures on \(E_{T_1}(\om ) \) with respect to \(E_+ (\om) \) and \(E_+ (\om), g_0, g_1, \ldots \) coincide with \( \nu _{T_1}^{E_+(\om )}.\)
We have, for all admissible \(T\), \(E_T(\om ) = \pi _{T_1,T} E_{T_1} (\om), \) so that the conditional measures on \(E_{T}(\om ) \) with respect to \(E_+ (\om) \) and \(E_+ (\om), g_0, g_1, \ldots \) coincide with \((\pi _{T_1,T} )_\ast  \nu _{T_1}^{E_+(\om )}= \nu _{T}^{E_+(\om )}.\)
\end{proof}

\

The projection \( \om \in \Om  \mapsto E_T(\om ) \in \X_T^{E_+(\om _+)} \) admits disintegrations that we denote \( m_T^{x}.\) We still denote by \( m_T^{x}\) the projection  of \( m_T^{x}\)  to \(\Om_-\).
 For instance, \(m_{T_0}^{f'}  = m_- \) for \(\nu '\)-a.e. \(f'\),   \(  m_{T_1}^{f} \) is the distribution of \( \om _- \) given the unstable flag \(f\).

By Proposition \ref{independence} and (\ref{entropyT}), the variables \( g_{-1} \) and \( E_T \) are quasi-independent (Indeed, since \( E_+ \) is \( E_T\) measurable, \( I(g_{-1}, E_T) = I(g_{-1}, E_+) + H (g_{-1} , E_T \vert E_+ ) \)  equals \( \kappa _T\) and  thus is finite). It follows that the density \( f \) of the measure \( m_T^{E_T(\om )}\) restricted to \( g_{-1}(\om ) \) with  respect to \( \mu \) is given by \[f(\om ) :=\,\frac{dm_T^{E_T(\om)}|_{g_{-1}}}{d\mu} (g_{-1}(\om ))\,=\, \frac{ d(g_{-1} (\om ))_\ast\nu _T^{(g_{-1}(\om ))^{-1} E_+(\om )} }{  d\nu _T^{E_+(\om ) } }(E_T(\om )) .\]
This yields another formula for \( \kappa _T :\)
\begin{equation}\label{kappa3} \kappa _T \; = \;   \int \left(\log\frac{dm_T^{E_T(\om)}|_{g_{-1}}}{d\mu} (g_{-1}(\om ))\right) \, dm(\om )\; =\; \int \log f(\om )\, dm(\om )\end{equation} 
\begin{proposition}\label{kappa4} We have, for \(m\)-a.e. \(\om \),  \[\kappa _T = \lim\limits _{n \to \infty } \frac{1}{n} \log \frac{dm_T^{E_T(\om )}|_{\{g_{-1}, \ldots, g_{-n}\}}}{d\otimes _1^n \mu} (g_{-1}(\om ), \ldots, g_{-n}(\om )).\] \end{proposition}
\begin{proof} We claim that the ratio \(f^{(n)}(\om ) :=  \frac{dm_T^{E_T(\om )}|_{\{g_{-1}, \ldots, g_{-n}\}}}{d\otimes _1^n \mu} (g_{-1}(\om ), \ldots, g_{-n}(\om ) )\) is given by \(  f ^{(n)} (\om )= \prod _{j=0}^{n-1} f(\s ^{-j} \om ).\) Then \(\frac{1}{n} \log f^{(n)} (\om )\) is an ergodic average of the function \( \log f(\om ).\) By the Birkhoff ergodic theorem this average converges to \( \int \log f (\om ) \, dm(\om ) = \kappa_T\) (the function \(\log f \) is integrable by Corollary \ref{finiteentropy}).

We prove the claim: by the law of composition of conditional probabilities, the ratio \( f^{(n)}(\om )/ f^{(n-1)}(\om ) \) is the density with respect to \( \mu \) of the conditional measures of \(m\) relative to \( (g_{-1}(\om ), \ldots, g_{-n+1}(\om ), E_T(\om ) ) \) restricted to \( g_{-n}(\om ).\) But the $\s$-algebra generated by \( (g_{-1}(\om ), \ldots, g_{-n+1}(\om ), E_T(\om ) ) \) is the same as the $\s$-algebra generated by \( (E_T(\s^{-n+1}\om ), g_0(\s^{-n+1}\om ),  \ldots, g_{n-2}(\s^{-n+1}\om )).\) By proposition \ref{quasiMarkov} and stationarity, this is the density of the measure  \( m_T^{E_T(\s^{-n+1}\om )}\) restricted to \( g_{-n }(\om ) = g_{-1}(\s^{-n+1} \om )\) , that is \(f(\s^{-n+1}\om ).\)

\end{proof}

\subsection{The case of the projective space \(\R\P^{d-1}\)}\label{rapaport}
Consider the case of the topology \( T_{Q} \) associated to the partition \( Q := 0 <1<d\). The space \(\F _{Q} \) is the projective space \(\R\P^{d-1}\) and the measure \( \nu _{Q} \) is the stationary measure considered by \cite{rapaport}. In this section, we compare the formulations of our results in the case when both framework coincide. Namely,  from corollary \ref{mainQ} applied to  the partition \( Q\), we get:
\begin{corollary} \label{mainR} Let $\mu \in \MM(d)$ be a discrete probability measure on $ SL_d(\R)$.  
Let $Q$ be the  partition \( Q := 0 <1<d\).
Then the unique stationary probability measure $\nu _Q$ on the projective space \(\F_Q =\R\P^{d-1}\) is exact-dimensional.  There are numbers \( \g _j \) such that
\begin{equation}\label{etvoilaR}  \de _Q= \sum _{j=1}^{d-1}  \g_{j} , \quad \quad h(\R\P^{d-1},\mu , \nu _Q) =  \sum _{j=1}^{d-1}   \g_{j} (\chi _1 - \chi _{j+1}). \end{equation} Let   \( T_{Q} = T^{d-1} \overset {1}{\prec} T^{d-2} \overset {1}{\prec}  \ldots \overset {1}{\prec} T^j \overset {1}{\prec} \ldots \overset {1}{\prec} T^1 \overset {1}{\prec} T^0 = T_0 \) be the intermediate topologies from  proposition \ref{order}.1 applied to \( T_Q.\) The numbers \( \g _j \) are the  exact dimensions of the conditional measures \( \nu _{T^{j}, T^{j-1}} \) and  we have 
\begin{equation}\label{ent/expR}  \g_{j} (\chi _1 - \chi _{j+1}) \; = \; \kappa _{T^{j}, T^{j-1}} \; = \; H(g_{-1} , E_{T^j}\vert E_+ ) - H(g_{-1} , E_{T^{j-1}}\vert E_+ ) .\end{equation}
 \end{corollary}

 \begin{proof} The topology \(T_Q\) associated to \( Q\) is (defined by its atoms)\[ T_{Q} \; = \; \{1\}, \{2, \ldots , d\} , \ldots, \{ d-1, d\}, \{d\} .\] Then, \( D_{T_{Q} , T_0 } = \{ (1,j), 1< j \leq d \}\) and the set of differences of exponents is \( \chi _1 - \chi _{j+1} , 1\leq j<d \).
Therefore,  the intermediate topologies from  proposition \ref{order}.1 \[ T_{Q} = T^{d-1} \overset {1}{\prec} T^{d-2} \overset {1}{\prec}  \ldots \overset {1}{\prec} T^j \overset {1}{\prec} \ldots \overset {1}{\prec} T^1 \overset {1}{\prec} T^0 = T_0 \] are given by \( T^j = \{1, j+2, \ldots , d\}, \{2, \ldots , d\} , \ldots, \{ d-1, d\}, \{d\} .\) Since   \(\chi _{T^{j},T^{j-1}} =\chi _1 - \chi _{j+1}\) for \(1\leq  j <d\), we indeed have that   \( \chi_{T^j, T^{j-1}}\) is increasing in \(j.\) 
Relation (\ref{etvoilaR}) follows from  theorem \ref{mainT} and corollary \ref{mainQ}, with \( \g_j = \de _{T^{j}, T^{j-1}}\)  for  \(1\leq  j <d\), relation (\ref{ent/expR}) from theorem \ref{exactfibers} and equation (\ref{entropyT}). \end{proof}

 If the measure \( \mu \) has finite support, a similar formula was shown in \cite{rapaport} under the (weaker) hypothesis that the support of \(\mu\) generates a strongly irreducible and proximal semi-group. Our purpose in this section is to indicate why both formulas are the same.\footnote{We will not discuss how, when restricted to \( T_Q\), our arguments would still hold under the hypothesis that the stationary measure is unique on \(\R\P^{d-1}\)  and that \( \chi _1 > \chi_2\). This would contain (and be very close  to) \cite{rapaport}.} In our case, all exponents are distinct; so, the \( \widetilde \la _j \) in  \cite{rapaport} are in our notations,
 \[ \widetilde \la _j \; = \; \chi _{j+1} - \chi _1 \; \textrm{{ for }} \; j = 1, \ldots, s= d-1.\] 
 Comparing (\ref{ent/expR}) with  \cite{rapaport} theorem 1.3, we see that 
 \[  \kappa _{T^j}- \kappa _{T^{j-1}} \; =\;  \kappa _{T^{j}, T^{j-1}} \; = \;  \g_{j} (\chi _1 - \chi _{j+1}) \; = \; - \g_j  \widetilde \la _j \; = \; H_{j-1} - H_j ,\]
 where \(H_j\) are the partial entropies from \cite{rapaport} theorem 1.3. Indeed, we 
  now verify that \( \kappa _{T^j} = H_0 - H_j.\)

Fix \(f' \in E_+(\Om_{reg} ) \)  and write  it as the stable flag of the Oseledets decomposition
\[ f' \;=\;  \{0\} \subset  U'_1 \subset \ldots \subset U'_j \subset \ldots \subset   U'_{d-1}  \subset  \R^{d} . \] 
We identify \( \X_{Q}^{f'} \) with the set of directions \( V_1\) in \(\R^{d} \) that do not belong to  \( U'_{d-1}\) and   for \( 1\leq j \leq d-1, \,  \X_{T^j}^{f'} \) with  the set of \((d-j)\)-planes \( V_{d-j}\)  such that  \( V_{d-j} \cap U'_{d-1} = U'_{d-j-1} \). The projection \(\pi _{T_Q, T^{j}} \) associates to \(V_1\) the space generated by \( V_1\) and \( U'_{d-j+1}.\) For \( 1\leq j <d, \) we define the partition \(\zeta _j\) of \(\Om _{reg} \) by the mapping \( \om \mapsto E_{T^j} (\om) \), i.e.
\[ \zeta_j (\om) := \{ \om ' \in \Om_{reg}, U_1 (\om') + U'_{d-j-1}(\om ') = U_1 (\om) + U'_{d-j-1} (\om) \},\]
where, for \( \om \in \Om _{reg}, \;  U_1 (\om )\) is the first unstable direction of the Oseledets decomposition. 
Consider the conditional measures \( m_j^{E_{T^j}(\om)} \) of \( m\) with respect to the partition \( \zeta _j\).

The entropy \( \kappa _{T^j} \)  is given by 
 (\ref{kappa3}): \[   \kappa _{T^j}  \; = \;   \int \left(\log\frac{dm_{j}^{E_{T^j}(\om)}|_{g_{-1}}}{d\mu} (g_{-1}(\om ))\right) \, dm(\om )\] and 
by proposition \ref{kappa4},  we have, for \(m\)-a.e. \(\om \),  \[\kappa _{T^{j}} = \lim\limits _{n \to \infty } \frac{1}{n} \log \frac{dm_{j}^{E_{T^j}(\om)}|_{\{g_{-1}, \ldots, g_{-n}\}}}{d\otimes _1^n \mu} (g_{-1}(\om ), \ldots, g_{-n}(\om )),\]
which coincide with \( H_0 - H_j \) in \cite{rapaport}, Theorem 1.3.  Observe that we have not used the fact that \( \mu \) is discrete for this last expression. If \(\mu \) is not discrete, relation (\ref{etvoilaR}) is not proven, but this last equation holds as soon as  \( \int \log \|g\| \, d\mu (g) <+\infty .\)

\subsection{Entropy difference}

For any pair of admissible topologies \(T \prec T'\) we defined  \(\kappa_{T,T'} =  \kappa_T - \kappa_{T'}\).
By relation (\ref{entropyT}), \(\kappa_{T,T'} <+\infty \) and if \(T \prec T' \prec T''\) one has
\[\kappa_{T,T''} = \kappa_{T,T'} + \kappa_{T',T''}.\]

By the chain rule for conditional mutual entropy, relation(\ref{entropyT})  and proposition \ref{quasiMarkov}, we have
\[\kappa_{T,T'} =  H(g_{-1},E_T\vert E_{T'}) = H(g_{-1},E_T\vert (E_{T'},g_0,g_1,\ldots))=  H(g_{-1},E_T\vert (E_{T'},\om_+)).\]

Recall that in the introduction we defined, for \(\nu '\)-a.e. \(f' \in \F,\) \( \nu _{T'}^{f'} \)-a.e. \( x' \in \X_{T'}^{f'} , \) \( \nu _{T,T'}^{x'} \) as a family of disintegrations of the measure \( \nu _T^{f'} \) with respect to \( \pi_{T,T'} .\)  
Then, \begin{equation}\label{transitions} \frac{ dg_\ast\nu _T^{f'} }{  d\nu _T^{gf'} }(y) = \frac{ dg_\ast\nu _{TT'}^{ x'} }{  d\nu _{T,T'}^{gx'} }(y) \frac{ dg_\ast\nu _{T'}^{ f'} }{  d\nu _{T'}^{gf'} }(x') .\end{equation}

\begin{proposition}\label{kappa1}
If \(T \prec T'\) are admissible topologies then
\[\kappa_{T,T'} = \int\limits_{\Omega}\log\left(\frac{dg_0(\omega)\nu_{T,T'}^{E_{T'}(\omega)}}{d\nu_{T,T'}^{E_{T'}(\sigma(\omega))}}(E_T(\sigma(\omega))\right) dm(\omega).\]
\end{proposition}
\begin{proof}
We write \(\kappa_T \) as \( \kappa _{T} \; = \; \mathbb E  \left[I (g_{-1},E_T   \vert E_+ ) (\om) \right].\) Reporting the explicit expression for \( I (g_{-1}(\om ),E_T (\om )  \vert E_+(\om ) ),\)
we have \begin{eqnarray*} \kappa _T &=& \mathbb E\left[ \log \frac{d (g_{-1} (\om )) _\ast \nu _T^{(g_{-1}(\om ) )^{-1}E_+(\om )} } {d\nu _T^{E_+(\om )}} (E_T(\om ))\right]\\ &=& \mathbb E\left[\log \frac{ d(g_{0} (\om )) _\ast \nu _T^{E_+(\om )} } {d\nu _T^{E_+(\s \om )}} (E_T(\s \om ))\right],\end{eqnarray*}
where the second line follows by \(\s\)-invariance. The proposition follows by making the difference \( \kappa _{T,T'} = \kappa _T -\kappa _{T'} \) and applying (\ref{transitions}). \end{proof}
Fix \( T \prec T' .\) For a.e. \( x \in (\pi_{T,T'})^{-1} (E_{T'} (\om ) ), \) set 
 \(f_\omega(x) := \frac{dg_0(\omega)\nu_{T,T'}^{E_{T'}(\omega)}}{d\nu_{T,T'}^{E_{T'}(\sigma(\omega))}}(x).\)
From proposition \ref{kappa1}, follows 
\begin{equation}\label{kappa2} \kappa _{T,T'} \; =\; \mathbb E \left[ \log f_\om (E_T(\om )) \right].\end{equation}

Recall formula (\ref{kappa3}) for \(\kappa _T \) and \(\kappa _{T'} .\) With the same notations, we have
\[ \kappa _{T,T'}\;   = \;   \int \left(\log \frac{dm_T^{E_T(\om)}|_{g_{-1}}}{dm_{T'}^{E_{T'}(\om)}|_{g_{-1}}}(\om )) \right) \, dm(\om )\]
and the following corollary of Proposition \ref{kappa4}
\begin{corollary}\label{kappa5} Assume \(T \prec T'\) are admissible topologies. Then, we have, for \(m\)-a.e. \(\om \),  \[\kappa _{T,T'} = \lim\limits _{n \to \infty } \frac{1}{n} \log \frac{dm_T^{E_T(\om )}|_{\{g_{-1}, \ldots, g_{-n}\}}}{dm_{T'}^{E_{T'}(\om )}|_{\{g_{-1}, \ldots, g_{-n}\}}} (g_{-1}(\om ), \ldots, g_{-n}(\om )).\] \end{corollary}

\subsection{Zero entropy difference}
\begin{proposition}\label{kappa0}
If \(\kappa_{T,T'} = 0\) then \(\nu_{T,T'}^{E_{T'}(\omega)}\) is the point mass at \(E_{T}(\omega)\) for \(m\)-a.e. \(\omega \in \Omega\).  In particular, it is exact dimensional with dimension \(0\).
\end{proposition}
\begin{proof}
By proposition \ref{order}, we may assume that  \( T \overset {1}{\prec} T'\) and that \( i<j\) are such that \( T'(i) = T(i) \cup \{ j\}.\) By proposition \ref{kappa1}, we have 
\[\kappa_{T,T'} = \int\limits_{\Omega}\log\left(\frac{dg_0(\omega)\nu_{T,T'}^{E_{T'}(\omega)}}{d\nu_{T,T'}^{E_{T'}(\sigma(\omega))}}(E_T(\sigma(\omega))\right) dm(\omega).\]

Therefore, by Jensen inequality, if \(\kappa_{T,T'} = 0\) then for \(m\)-almost every \(\omega\)  we have 
 \(g_{0}(\omega)\nu_{T,T'}^{E_{T'}}(\omega) = \nu_{T,T'}^{E_{T'}(\sigma(\omega))}\).
Since \(m\) is \(\sigma\)-invariant, we obtain that \(m\)-almost everywhere one has \(\nu_{T,T'}^{E_{T'}(\omega)} = g_{-1}(\omega)\ldots g_{-n}(\omega)\nu_{T,T'}^{E_{T'}(\sigma^{-n}(\omega))}\) for all \(n \ge 1\).

Observe that, since \(E_+\) and \(E_-\) are in general position, one has \(E_{T'}(\omega)_{T'(i) \setminus \lbrace i\rbrace} \neq E_{T}(\omega)_{T(i)}\) 
for \(m\)-almost every \(\omega \in \Omega\), and therefore \(\nu_{T,T'}^{E_{T'}(\omega)}\left(\lbrace E_{T'}(\omega)_{T'(i) \setminus \lbrace i\rbrace}\rbrace \right) = 0\). Using the  coordinate \( \vf _{x(\om )} \) (see  lemma \ref{coordinateslemma}),    we have
\( \nu_{T,T'}^{E_{T'}(\omega)} [\vf _{x(\om )}((-\pi/2, \pi /2))] = 1.\)

By Poincar\'e recurrence, for almost every  \( \om \in \Om_- \),  there exists an infinite sequence \(n_k, k\in \N\) and \( \alpha >0\) such that 
\( \nu_{T,T'}^{E_{T'}(\sigma^{-n_k}(\omega))} \)converges to \(\nu_{T,T'}^{E_{T'}(\omega)}\) as \( k \to \infty \) and \(\nu_{T,T'}^{E_{T'}(\omega)}\left( (\vf _{x(\s ^{-n_k} (\om))} ) ((-\alpha, \alpha))\right) > \alpha.\)

Hence, for those \( \om \) for which  lemma \ref{coordinateslemma} hold, \( g_{-1}(\omega)\ldots g_{-n_k}(\omega)\) sends any neighborhood of \( E_{T}(\sigma^{-n_k}(\omega)) \) to a small neighborhood of \( E_{T}(\omega) \).
This is possible only if \( \nu_{T,T'}^{E_{T'}(\omega)} \) is concentrated at \( E_{T}(\omega) \). \end{proof}

 \section{Proof of theorem \ref{exactfibers}} \label{section:telescoping}
 
Theorem \ref{exactfibers} is an almost everywhere statement. It uses a telescoping argument mixed with weak type (1,1) techniques  as in  the proof of Shannon-McMillan-Breiman theorem for finite entropy partitions. 
The proof  follows  \cite{hochman-solomyak2017} and  \cite{lessa}.

 In this section, we  assume that  \( T \overset {1}{\prec} T'\) and that \( i<j\) are such that \( T'(i) = T(i) \cup \{ j\}.\) Let \( \chi := \chi _{T,T'}, \kappa := \kappa _{T,T'}. \) we want to show that for \( \nu ' \)-a.e. \( f' \in \F ,\) \( \nu _{T'}^{f'} \)-a.e. \( x' \in \X_{T'}^{f'}, \) the conditional measure \(\nu_{T,T'}^{x'} \)  is exact-dimensional with dimension \( \kappa / \chi.\)

\subsection{Length of stationary neighborhoods}

Let \(\pi\:= \pi_{T,T'} \) be the projection from \(\X_{T}\) to \(\X_{T'}\) and consider the  coordinates given by lemma \ref{coordinate} on \(\pi^{-1}(E_{T'}(\omega)),\) setting \(x = E_{T}(\omega)\) and \(x' = E_{T'}(\omega)\). 

For each \(\alpha,  0 < \alpha < \pi /2\) let \(N^{\alpha}(\omega)\) (respectively \(N^{\alpha,+}(\omega), N^{\alpha,-}(\omega)\)) the set \( \vf ((-\alpha, \alpha )) \) (respectively \(\vf( [0, \alpha)), \vf((-\alpha, 0] )\). Recall that we denoted by \( \eta\) the rotation invariant probability measure on \(\pi^{-1}(E_{T'}(\omega)).\) The measure \( \vf _\ast du \) has a bounded continuous density with respect to \( \eta\) (the bound depends on \(\om\)). From lemma \ref{coordinateslemma} follows

\begin{lemma}\label{intervallengthlemma}
For all \(\alpha,  0 < \alpha <\pi /2\), for \(m\)-a.e. \(\omega \in \Omega\) one has 
 \[\eta\left(g_{-1}(\omega)\ldots g_{-n}(\omega)N^{\alpha}(\sigma^{-n}(\omega))\right) = e^{-\chi n + o(n)}\text{ as }n \to +\infty,\]
 and the same holds for \(N^{\alpha,+}\) and \(N^{\alpha,-}\).
\end{lemma}
In the sequel, we set, for each \(\alpha,  0 < \alpha < \pi /2\) and  for each \(n \ge 1\),  \(N_n^\alpha (\omega) \)  for the interval  in \( \pi ^{-1}(E_{T'}( \om )) \) given by \[N_n^{\alpha }(\om )= (g_{-1}(\omega)\ldots g_{-n}(\omega))(N^{\alpha} (\sigma^{-n}(\omega)).\]   
\( N_n^{\alpha,+} (\omega),  N_n^{\alpha,-} (\omega) \) are defined similarly, \( N_n^{\alpha,\pm} (\omega)\) is a choice of one of the  three intervals.

\subsection{Probability of stationary neighborhoods}
Theorem \ref{exactfibers} will follow by comparing lemma \ref{intervallengthlemma} and the following
\begin{proposition}\label{intervalprobabilitylemma}
 For all \(\alpha,  0 < \alpha <\pi /2\), for \(m\)-a.e. \(\omega \in \Omega\),  one has 
 \[\nu_{T,T'}^{E_{T'}(\omega)}(N_n^{\alpha,\pm }(\omega)) = e^{-\kappa n + o(n)}\nu_{T,T'}^{E_{T'}(\sigma^{-n}(\omega))}(N^{\alpha,\pm' }(\sigma^{-n}(\omega)))\text{ as }n \to +\infty.\]
\end{proposition}
 In this formula, the sign \(\pm '\) is not necessarily the same as \( \pm .\)

\begin{proof}
Recall that we set  \(f_\omega(x) = \frac{dg_0(\omega)\nu_{T,T'}^{E_{T'}(\omega)}}{d\nu_{T,T'}^{E_{T'}(\sigma(\omega))}}(x)\),  and that by (\ref{kappa2}), \[ \kappa \; =\; \int  \log f_\om (E_T(\om )) dm (\om ) \; =\; \mathbb E \left[ \int _{\pi^{-1}(E_{T'}(\omega))}\log f_\om (x)\, d\nu_{T,T'}^{E_{T'}(\omega)}(x) \right] .\]
Since \( \kappa  < +\infty ,\)  for \(m\)-a.e. \( \om \), \( \int _{\pi^{-1}(E_{T'}(\omega))} \log f_\om (x)\, d\nu_{T,T'}^{E_{T'}(\omega)}(x) < +\infty .\)

Set  for each \(\alpha,  0 < \alpha < \pi /2\), for each \(n \ge 1\),
 \[f_{n}^{\alpha, \pm }(\omega) = \frac{\nu_{T,T'}^{E_{T'}(\omega)}(N_n^{\alpha, \pm }(\omega))}{\nu_{T,T'}^{E_{T'}(\sigma(\omega))}(g_0(\omega) N_n^{\alpha, \pm }(\omega))} = \frac{\int\limits_{g_0(\omega)N_n^{\alpha, \pm }(\omega)}f_{\omega}(x)d\nu_{T,T'}^{E_{T'}(\sigma(\omega))}(x)}{\nu_{T,T'}^{E_{T'}(\sigma(\omega))}(g_0(\omega) N_n^{\alpha, \pm }(\omega))}.\]

\begin{lemma} \label{makerhypothesislemma}
  For all \(\alpha,  0 < \alpha <\pi /2\), for \(m\)-a.e. \(\omega \in \Omega\) one has \(\lim\limits_{n \to +\infty}f_n^{\alpha ,\pm }(\omega) = f_{\omega}(E_{T}(\sigma(\omega)))\).   Furthermore \(\int \sup\limits_{n \ge 1}|\log(f_n^{\alpha ,\pm }(\omega))| \,dm(\omega) < +\infty\).
\end{lemma}
\begin{proof}
 By lemma \ref{intervallengthlemma} the intervals \(N_n^{\alpha ,\pm }(\omega)\) intersect to \(E_{T}(\omega)\) when \(n \to +\infty\) for \(m\)-a.e. \(\omega \in \Omega\).   By the Lebesgue differentiation theorem this implies
 \[\lim\limits_{n \to +\infty}f_n^{\alpha ,\pm }(\omega) = f_{\omega}(g_0(\omega)E_{T}(\omega)) = f_{\omega}(E_{T}(\sigma(\omega))),\]
 for \(m\)-a.e. \(\omega \in \Omega\) as claimed.

For the second claim, let \(M f_{\omega}\) be the maximal function defined by 
\[M f_{\omega}(x) = \sup\limits_{N \supset \lbrace x\rbrace} \frac{1}{\nu_{T,T'}^{E_{T'}(\sigma(\omega))}(N)}\int\limits_{N} f_{\omega}(x)\,d\nu_{T,T'}^{E_{T'}(\sigma(\omega))}(x),\]
where the supremum is over intervals \(N\) starting or finishing at \(x\) in \(\pi^{-1}(E_{T'}(\sigma(\omega))).\)
Since  \( \pi^{-1}(x')\) is one-dimensional,  for all \(x' \in \X_{T'}\), the maximal operator is of weak type (1,1) and we have (see \cite[Lemma 8]{lessa}, \cite[Lemma 9]{lessa} for details)

 \begin{equation}\label{maximalintegrabilitylemma} \int \log\left(M f_{\omega}(E_{T}(\sigma(\omega)))\right)\, dm(\omega) < +\infty.\end{equation}
 It remains to show that \(\inf\limits_{n}\log(f_n^{\alpha ,\pm }(\omega))\) is \(m\)-integrable. See \cite[Lemma 10]{lessa} for the argument in a very similar setting. \end{proof}
 
 From Birkhoff ergodic theorem we have
 \begin{align*}
\kappa &= \lim\limits_{n \to +\infty} \frac{1}{n}\sum\limits_{k = 1}^{n}\log\left(\frac{dg_{-k}\nu_{T,T'}^{E_{T'}(\sigma^{-k}(\omega))}}{d\nu_{T,T'}^{E_{T'}(\sigma^{-k+1}(\omega))}}(E_{T}(\sigma^{-k+1}(\omega)))\right)
\\ &= \lim\limits_{n \to +\infty} \frac{1}{n}\sum\limits_{k = 1}^{n}\log\left(f_{\sigma^{-k}(\omega)}(E_{T}(\sigma(\sigma^{-k}(\omega))))\right)
\end{align*}

To conclude, using Lemma \ref{makerhypothesislemma} and Maker theorem (see e.g. \cite[Theorem 5.7]{rapaport}) we obtain
 \begin{align*}
\kappa&= \lim\limits_{n \to +\infty} \frac{1}{n}\sum\limits_{k = 1}^{n}\log\left(f_{n-k}^{\alpha, \pm }(\sigma^{-k}(\omega))\right)
\\ &= \lim\limits_{n \to +\infty} \frac{1}{n}\sum\limits_{k = 1}^{n}\log\left(\frac{\nu_{T,T'}^{E_{T'}(\sigma^{-k}(\omega))}(N_{n-k}^{\alpha, \pm }(\sigma^{-k}(\omega)))}{\nu_{T,T'}^{E_{T'}(\sigma^{-k+1}(\omega))}(g_{-k}(\omega) N_{n-k}^{\alpha, \pm }(\sigma^{-k}(\omega)))}\right)
\\ &= \lim\limits_{n \to +\infty} \frac{1}{n}\sum\limits_{k = 1}^{n}\log\left(\frac{\nu_{T,T'}^{E_{T'}(\sigma^{-k}(\omega))}(g_{-k-1}(\omega) \ldots g_{-n}(\omega) N^{\alpha, \pm }(\sigma^{-n}(\omega)) )}{\nu_{T,T'}^{E_{T'}(\sigma^{-k+1}(\omega))}(g_{-k}(\omega) \ldots g_{-n}(\omega) N^{\alpha, \pm }(\sigma^{-n}(\omega)))}\right)
\\ &= \lim\limits_{n \to +\infty} \frac{1}{n}\log\left(\frac{\nu_{T,T'}^{E_{T'}(\sigma^{-n}(\omega))}(N^{\alpha, \pm }(\sigma^{-n}(\omega)))}{\nu_{T,T'}^{E_{T'}(\omega)}(N_n^{\alpha, \pm }(\omega))}\right).
\end{align*}
\end{proof}

\subsection{Exact dimensionality}
We finish the proof of theorem \ref{exactfibers}.
 Firstly, if \( \kappa =0, \) by proposition \ref{kappa0}, the measure \( \nu_{T,T'}^{E_{T'}(\om )} \) is exact-dimensional with dimension 0. On the other hand, since \(  \chi >0 \) and \(\kappa =0 \), we have \( 0= \kappa /\chi.\)

So we may assume \( \kappa >0.\) By proposition \ref{intervalprobabilitylemma}, for \(m\)-a.e. \( \om \), all \(\alpha,  0 < \alpha < \pi /2\), \(n\) going to infinity, 
\[ \frac{1}{n} \log \nu_{T,T'}^{E_{T'}(\omega)}(N_n^{\alpha,\pm}(\omega)) \; =\; - \kappa + o(1) + \frac{1}{n} \log \nu_{T,T'}^{E_{T'}(\s ^{-n} \omega)}(N^{\alpha,\pm ' }(\s^{-n} \omega)) .\]
On the other hand, we have, by lemma \ref{intervallengthlemma},  for \(m\)-a.e. \( \om \), all \(\alpha,  0 < \alpha < \pi /2\), \(n\) going to infinity, 
\[ B_{\pi^{-1} (E_{T'}(\omega))}^\pm  (E_T(\om ), e^{-n\chi + o(n)})  \subset N_n^{\alpha, \pm' } (\om )  \subset B_{\pi^{-1} (E_{T'}(\omega))}^\pm  (E_T(\om ), e^{-n\chi + o(n)}) ,\]
where \( B_{\pi^{-1} (E_{T'}(\omega))}^\pm  (x,r) \) is either of the two intervals of size \(r\) based on \(x\).

It follows that for \(m\)-a.e. \( \om \),  \(\nu_{T,T'}^{E_{T'}(\omega)}\)-a.e. \(x\),
\[\liminf\limits_{r \to 0}\frac{\log(\nu_{T,T'}^{E_{T'}(\omega))}(B(x,r))}{\log(r)} \ge \frac{\kappa}{\chi}.\]

We cannot estimate directly  \( \limsup\limits_{r \to 0}\frac{\log(\nu_{T,T'}^{E_{T'}(\omega))}(B(x,r))}{\log(r)} \) in the same way, because we do not know a priori that \( \liminf\limits_{n\to \infty }\frac{1}{n} \log (\nu_{T,T'}^{E_{T'}(\s ^{-n}\omega))}(N^{\alpha, \pm } (\s^{-n}\om )) =0.\) The observation is that to compute \(\limsup\limits_{r \to 0}\frac{\log(\nu_{T,T'}^{E_{T'}(\omega))}(B(x,r))}{\log(r)} , \) it suffices to consider values of \(-\log r\) in \( \N\) with a positive density. 

We claim that we can take \( \alpha \) large enough that one of  \( \nu_{T,T'}^{E_{T'}(\omega)}(N^{\alpha, + } (\om ))\) and  \( \nu_{T,T'}^{E_{T'}(\omega)}(N^{\alpha, -  } (\om ))\) is at  least some \(c >0\)  on a set \( \Om ' \subset \Om \) of positive measure. This finishes the proof, because, by Birkhoff ergodic theorem,  for \(m\)-a.e.  \(\om \) the sequence \( n_k \) such that \( \s ^{-n_k} \om  \in \Om ' \) has positive density. 

Remains to prove the claim. We prove it by contradiction: if it is not true, \( \nu_{T,T'}^{E_{T'}(\omega)}(N^{\alpha } (\om )) = 0 \)
\( m\)-a.e. for all \(\alpha \). But then, the measure \( \nu_{T,T'}^{E_{T'}(\omega)}\) is concentrated on \( E_{T'}(\omega)_{T'(i) \setminus \lbrace i\rbrace} \), a contradiction with the fact that \(E_+\) and \(E_-\) are in general position.

\subsection{Proof of lemma \ref{factors}}
Assume the diagram of projections \(\begin{matrix} T & \longrightarrow & S  \\
\downarrow 1 & & \downarrow 1 &   \\
T' & \longrightarrow & S'  
\end{matrix}\) commutes and \(i,j\) are such that \(T(i) = T'(i) \setminus \lbrace j\rbrace\) and \(S(i) = S'(i) \setminus \lbrace j \rbrace\).

Then, by Corollary \ref{factors2}, for \(x' \in \X_{S'} \) and \( y'  \in (\pi _{T',S'})^{-1} (x') \)  there is a bilipschitz homeomorphism between 
 \( (\pi _{S,S'})^{-1} (x') \) and  \( (\pi _{T,T'})^{-1} (y') \). 
 The measure  \( \nu ^{x'}_{S,S'} \) is the average over \( y'\) of the measures  \( (\pi _{T,S} )_\ast \nu ^{y'}_{T,T'} \), the average being taken under \( d\nu ^{x'}_{T',S'}(y') .\)  Lemma \ref{factors} then follows from \cite{ledrappier-young1985} Lemma 11.3.2. \( \qed \)

\section{Proof of Theorem \ref{counting} }\label{addingdimensions}
In this section, we assume that the measure \( \mu \) is discrete and prove Theorem \ref{counting}.

The general idea of the proof is that  entropy conservation implies dimension conservation. In \cite{ledrappier-young1985}, dynamical balls are disjoint ellipsoids with exponentially big eccentricities, not suitable for dimension estimates. But they behave very well for entropy estimates, even when considering their slices by invariant foliations. If one takes the slices in increasing order of size, this forces dimension conservation. For IFS or stationary measures, the dynamical balls are not disjoint any more. The idea of \cite{feng} is to look at the dynamical balls and the slicing at the level of 
the invertible dynamics on \(\Om \). Since  the measure  $\mu$ is discrete, working at the level of the space \(\Om \)  is possible here as well. 

\subsection{Setup}

Recall that we consider \( T \prec T' \) and the decomposition \(T^0 ,T^1,\ldots,T^{N_{T,T'}} \), where \(T^0 = T'\) and \(T^{N_{T,T'}} = T\) of proposition \ref{order} with \( T^t \overset{1}{\prec} T^{t-1}\) for \( t = 1, \ldots {N_{T,T'}}\).   

Recall that for each admissible topology \(S\) we set \(E_S(\omega) = F_{S}(E(\om))\).

We will use \(x\) to denote a point in \(\X_{T}\) and always set \(x_t = \pi_{T,T^t}(x)\) and \(x_{t-1} = \pi_{T,T^{t-1}}(x)\).

We fix \(t\) and set \(\chi = \chi_{T^t,T^{t-1}}\), \(\kappa = \kappa_{T^t,T^{t-1}}\) and \(\delta = \overline{\delta^{t}}\). Then,  \(\gamma_{T^t,T^{t-1}} = \kappa/\chi\).

We are interested in comparing for \(\nu_{\omega} = \nu_{T,T_0}^{E_+(\omega)}\)-almost every \(x\), the upper dimension at \(x\) of \(\nu _{T, T^t}^{x_t}\) and \(\nu _{T, T^{t-1}}^{x_{t-1}}\).

For this purpose we fix \(\varepsilon > 0\) and set \(r_n = \exp(-n(\chi - \varepsilon))\) for all \(n\) in what follows.

\subsection{Approximating configurations}

\begin{lemma}\label{approximationlemma}
There exist functions \(f_n\) such that, for \(m\)-a.e. \( \om  ,\) setting \(E_{T,n}(\omega) = f_n(E_{T^{t-1}}(\omega),g_{-1}(\omega),\ldots,g_{-n}(\omega))\) there exists \(n(\omega)\) such that for all \(n \ge n(\omega)\),
\(E_{T,n}(\omega) \in B(E_{T}(\omega),r_n)\).
\end{lemma}
\begin{proof}
The construction of Proposition \ref{order} guarantees that 
\[\chi = \min\left\lbrace\chi_{T^s,T^{s-1}}: s = t,t+1,\ldots, N_{T,T'}\right\rbrace.\] 

Therefore the claim follows directly from proposition \ref{configurationapproximationlemma}.
\end{proof}

\begin{corollary}\label{approximationcorollary1}
For \(m\)-a.e. \( \om \), \(\nu_{\omega}\)-almost every \(x\) and \(m_{T^t}^{x_t}\)-almost every \((h_n)_{n \le -1}\) there exists \(N(\om , x_t,(h_n)_{n \le -1})\) such that, for all \(n \ge N(\om, x_t,(h_n)_{n \le -1})\).
\[f_n(x_{t-1},h_{-1},\ldots,h_{-n}) \in B\left(\lim\limits_{k \to +\infty}f_k(x_{t-1},h_{-1},\ldots,h_{-k}), r_n\right).\]
\end{corollary}

\subsection{Approximating conditional probabilities}
Let \(T\) be an admissible topology. By definition, \(E_T(\om ) \in \X_T^{E_+(\om _+)} \) and, if \(T \prec T', \; E_{T'}(\om ) = \pi _{T,T'} (E_T(\om )).\) We introduced in section \ref{section:cond.mut.info} the disintegrations 
of the projection \( \om \in \Om  \mapsto E_T(\om ) \in \X_T^{E_+(\om _+)} \)  and their restrictions \( m_T^{E_T(\om)}\)  to \(\Om_-\). 
 With  our notations, we have, for \( \nu' \)-a.e. \(f', \nu _{T^t}^{f'} \)-a.e. \( z \in \X_{T^{t-1}}, \)
\begin{equation}\label{m-measures}  m_{T^{t-1}}^z \; = \; \int _{\X_{T^t, T^{t-1}}^z} m_{T^{t}}^y \, d\nu_{T^t, T^{t-1}}^z(y)
 . \end{equation}

Given \(x \in \X_T\), let \(A_n(x_{t-1})\) be the set of  sequences \((h_{n})_{n \le -1}\) such that, for all \( m \geq n,\) \(f_m(x_{t-1},h_{-1},\ldots,h_{-m}) \in B\left(\lim\limits_{k \to +\infty}f_k(x_{t-1},h_{-1},\ldots,h_{-k}), r_m\right)\).
 We let \([h_{-n},\ldots,h_{-1}]\) denote the set of sequences \((h_n')_{n \le -1}\) such that \(h_{-n}=h_{-n}',\ldots,h_{-1}=h_{-1}'\).

\begin{lemma}\label{conditionalslemma}
For \( m \)-a.e. \( \om \), \(\nu_{\omega}\)-almost every \(x\) and \(m_{T^{t}}^{x_t}\)-almost every \((h_n)_{n \le -1}\), there exists \(N(\om, x_t,(h_n)_{n \le -1})\) such that, for all \(n \ge N(\om, x_t,(h_n)_{n \le -1})\),
\[m_{T^{t}}^{x_t}([h_{-n},\ldots,h_{-1}] \cap A_n(x_{t-1})) \le \exp(n(\kappa + \varepsilon))m_{T^{t-1}}^{x_{t-1}}([h_{-n},\ldots,h_{-1}] \cap A_n(x_{t-1})).\]
\end{lemma}
\begin{proof}
For \(m\)-a.e. \( \om \), consider the measure \(\P_\om\) on the space \(\X_T \times \lbrace (h_k)_{k \le -1} \rbrace\) which projects to  \(\nu_{\omega}\) on \(\X_T\) and whose disintegrations on the fibers projecting to \(x \in \X_T\) are given by \(x \mapsto m_{T^{t}}^{x_t}\). 

By corollary \ref{approximationcorollary1}, for \(m\)-a.e. \( \om \), \(\nu_{\omega}\)-almost every \(x\), \(A_n(x_{t-1})\) is increasing with \(n\) and \(\cup_n A_n(x_{t-1})= \Om _-\) up ot a set of \(m_{T^{t}}^{x_t}\)-measure 0.  By the martingale convergence theorem,  the conditional expectation with respect to \(\P_\om \) of the indicator of the event \(\lbrace (h_k)_{k \le -1} \in A_n(x_{t-1})\rbrace\) with respect to the \(\sigma\)-algebras generated by \(x_{t},h_{-1},\ldots,h_{-n}\) (respectively  \(x_{t-1},h_{-1},\ldots,h_{-n}\)) converge to 1. 

The first conditional expectation is given at  \(m\)-a.e. \( \om \), \(\nu_{\omega}\)-almost every \(x\) and \(m_{T^t}^{x_t}\)-almost every \((h_n)_{n \le -1}\) by
\[\frac{m_{T^{t}}^{x_t}([h_{-n},\ldots,h_{-1}] \cap A_n(x_{t-1}))}{m_{T^{t}}^{x_t}([h_{-n},\ldots,h_{-1}])} .\]
The second one by
\[ \frac {\int\limits_{\X_{T^t, T^{t-1}}^{x_{t-1}}}m_{T^{t}}^{y}([h_{-n},\ldots,h_{-1}] \cap A_n(x_{t-1}))\,d\nu_{T^t,T^{t-1}}^{x_{t-1}}(y)}{\int\limits_{\X_{T^t,T^{t-1}}^{x_{t-1}}} m_{T^{t}}^{y}([h_{-n},\ldots,h_{-1}] )\,d\nu_{T^t,T^{t-1}}^{x_{t-1}}(y) }.\]
Using (\ref{m-measures}), this last expression is 
\[ \frac{m_{T^{t-1}}^{x_{t-1}}([h_{-n},\ldots,h_{-1}] \cap A_n(x_{t-1}))}{m_{T^{t-1}}^{x_{t-1}}([h_{-n},\ldots,h_{-1}])} \]
 Thus,   at  \(m\)-a.e. \( \om \), \(\nu_{\omega}\)-almost every \(x\) and \(m_{T^t}^{x_t}\)-almost every \((h_n)_{n \le -1}\),
\[\lim\limits_{n \to +\infty}\frac{m_{T^{t}}^{x_t}([h_{-n},\ldots,h_{-1}] \cap A_n(x_{t-1}))}{m_{T^{t}}^{x_t}([h_{-n},\ldots,h_{-1}])} =  \lim\limits_{n \to +\infty}\frac{m_{T^{t-1}}^{x_{t-1}}([h_{-n},\ldots,h_{-1}] \cap A_n(x_{t-1}))}{m_{T^{t-1}}^{x_{t-1}}([h_{-n},\ldots,h_{-1}])} = 1.\]

We have shown (in corollary \ref{kappa5}) that for \( m\)-a.e. \(\om \), \(\nu_\omega\)-almost every \(x\) and \(m_{T^{t}}^{x_t}\)-almost every \((h_n)_{n \le -1}\) one has
\[\lim\limits_{n \to +\infty}\frac{1}{n}\log\left(\frac{m_{T^t}^{x_t}([h_{-n},\ldots,h_{-1}])}{m_{T^{t-1}}^{x_{t-1}}([h_{-n},\ldots,h_{-1}])}\right) = \kappa.\]

Combining these three statements we obtain that for \(m\)-a.e. \( \om \),  \(\nu_\omega\)-almost every \(x\) and \(m_{T^{t}}^{x_t}\)-almost every \((h_n)_{n \le -1}\) one has
\[\lim\limits_{n \to +\infty}\frac{1}{n}\log\left(\frac{m_{T^{t}}^{x_t}([h_{-n},\ldots,h_{-1}] \cap A_n(x_{t-1}))}{m_{T^{t-1}}^{x_{t-1}}([h_{-n},\ldots,h_{-1}] \cap A_n(x_{t-1}))}\right) = \kappa,\]
from which the desired result follows immediately.
\end{proof}

\subsection{Lebesgue Density}

Given \(x \in \X_T\) let \(B_n(x_{t})\) denote the set of sequences \((h_n)_{n \le -1}\) such that
\[m_{T^{t}}^{x_t}([h_{-n},\ldots,h_{-1}] \cap A_n(x_{t-1})) \le \exp(n(\kappa + \varepsilon))m_{T^{t-1}}^{x_{t-1}}([h_{-n},\ldots,h_{-1}] \cap A_n(x_{t-1})).\]

\begin{lemma}\label{lebesguelemma}
For \(m \)-a.e. \( \om \) and \(\nu_{\omega}\) almost every \(x\) there exists \(N(\om ,x)\) such that, for all \(n \ge N(\om ,x )\),
 \[m_{T^t}^{x_t}\left(\left\lbrace \lim\limits_{k \to +\infty}f_k(x_{t-1},h_{-1},\ldots,h_{-k}) \in B(x,r_n)\right\rbrace \cap A_n(x_{t-1}) \cap B_n(x_t)\right) \ge \exp(-n(\delta + \varepsilon)\chi).\]
\end{lemma}
\begin{proof}
The probability on the left-hand side in the statement is
\[\int\limits_{B(x,r_n)}m_{T}^{y}(A_n(x_{t-1}) \cap B_n(x_t))d\nu_{T,T^t}^{x_{t}}(y) = \int\limits_{B(x,r_n)}m_{T}^{y}(A_n(y_{t-1}) \cap B_n(y_t))d\nu_{T,T^t}^{x_{t}}(y).\]

Let \(C_n(x) = \bigcap\limits_{m \ge n}B_m(x_t)\) so that \(C_n(x)\) is increasing with \(n\).

We have
\[\int\limits_{B(x,r_n)}m_{T}^{y}(A_n(y_{t-1}) \cap B_n(y_{t}))d\nu_{T,T^t}^{x_t}(y) \ge \int\limits_{B(x,r_n)}m_{T}^{y}(A_n(y_{t-1}) \cap C_n(y_t))d\nu_{T,T^t}^{x_t}(y).\]

By lemma \ref{conditionalslemma}, for \( m \)-a.e. \( \om \), \(\nu_\omega\)-a.e.  \(x\), the function \(y \mapsto m_{T}^{y}(A_n(y_{t-1}) \cap C_n(y_t))\) increases to \(1\) at   \(\nu_{T,T^{t}}^{x_t}\)-a.e.  \(y\). 
Applying the Lebesgue differentiation theorem (justified since configuration spaces are bilipschitz homeomorphic to Euclidean spaces) we obtain a set \(L(\om, x_t )\) of \(\nu_{T,T^t}^{x_t}\)-full measure such that for all \(z\) in this set
\[\lim\limits_{n \to +\infty}\frac{1}{\nu_{T,T^t}^{x_t}(B(z,r_n))}\int\limits_{B(z,r_n)}m_{T}^{y}(A_n(y_{t-1}) \cap C_n(y_t))d\nu_{T,T^t}^{x_t}(y) = 1.\]

For \(m\)-a.e. \( \om \),  \(\nu_\omega\)-a.e.  \(x\) belongs to \( L(\om, x_t)\). Moreover, by hypothesis, for \(m \)-a.e. \( \om \) and \(\nu_\omega\)-a.e. \(x\),  there exists \(N(\om , x)\)  such that  for \(n \geq N(\om ,x) \), \(\nu_{T,T^t}^{x_t}(B(x,r_n)) \ge \exp(-n(\delta+\varepsilon)\chi)\).  The result follows.
\end{proof}

\subsection{Proof of the theorem}

We now complete the proof of Theorem 2.4.

We begin with Lemma \ref{lebesguelemma} and observe that if \(\lim\limits_{k \to +\infty}f_k(x_{t-1},h_{-1},\ldots,h_{-k}) \in B(x,r_n)\) and \((h_n)_{n \le -1} \in A_n(x_{t-1}),\) then in fact \(f_n(x_{t-1},h_{-1},\ldots,h_{-n}) \in B(x,2r_n)\).  

This implies that for \(m \)-a.e. \( \om \), \(\nu_{\omega}\)-a.e.  \(x\) and all \(n \ge N(\om ,x)\) given by Lemma \ref{lebesguelemma} one has
\begin{align*}
 \exp(-n(\delta+\varepsilon)\chi) &\le m_{T^t}^{x_t}\left(\left\lbrace \lim\limits_{k \to +\infty}f_k(x_{t-1},h_{-1},\ldots,h_{-k}) \in B(x,r_n)\right\rbrace \cap A_n(x_{t-1}) \cap B_n(x_t)\right)
\\ &\le m_{T^t}^{x_t}\left(\left\lbrace f_n(x_{t-1},h_{-1},\ldots,h_{-n}) \in B(x,2r_n)\right\rbrace \cap A_n(x_{t-1}) \cap B_n(x_t)\right).
\end{align*}

Notice that both \(D_n(x_{t-1}) : = \lbrace (h_n)_{n \le -1}: f_n(x_{t-1},h_{-1},\ldots,h_{-n}) \in B(x,2r_n)\rbrace\) and \(B_n(x_t)\) are a union of cylinders \([h_{-n},\ldots,h_{-1}]\).  
Therefore from the second line above and the definition of \(B_n(x_t)\) we obtain
\begin{align*}
 \exp(-n(\delta+\varepsilon)\chi) &\le \sum\limits_{[h_{-n},\ldots,h_{-1}] \subset B_n(x_t) \cap D_n(x_{t-1})}m_{T^t}^{x_t}([h_{-n},\ldots,h_{-1}] \cap A_n(x_{t-1}))
 \\ &\le \exp(n(\kappa+\varepsilon))\sum\limits_{[h_{-n},\ldots,h_{-1}] \subset D_n(x_{t-1})}m_{T^{t-1}}^{x_{t-1}}([h_{-n},\ldots,h_{-1}] \cap A_n(x_{t-1})).
\end{align*}

For \(m \)-a.e. \( \om \), \(\nu_{\omega}\)-a.e.  \(x\) and all \(n \ge N(\om ,x)\), we have shown that 
\[\exp(-n(\delta+\varepsilon)\chi)\exp(-n(\kappa+\varepsilon)) \le m_{T^{t-1}}^{x_{t-1}}\left(\left\lbrace f_n(x_{t-1},h_{-1},\ldots,h_{-n}) \in B(x,2r_n)\right\rbrace \cap A_n(x_{t-1})\right).\]

Since whenever \((h_n)_{n \le -1} \in A_n(x_{t-1})\) we have that \(f_n(x_{t-1},h_{-1},\ldots,h_{-n})\) and \(\lim\limits_{k \to +\infty}f_k(x_{t-1},h_{-1},\ldots,h_{-k})\) are at distance at most \(r_n\), this implies that for \(m \)-a.e. \( \om \), \(\nu_{\omega}\)-a.e.  \(x\) and all \(n \ge N(\om ,x)\),
\begin{align*}
 \exp(-n(\delta+\varepsilon)\chi)\exp(-n(\kappa+\varepsilon)) &\le m_{T^{t-1}}^{x_{t-1}}\left(\left\lbrace (h_k)_{k \le -1}: \lim\limits_{k \to +\infty}f_k(x_{t-1},h_{-1},\ldots,h_{-n}) \in B(x,3r_n)\right\rbrace\right)
 \\ &= \nu_{T,T^{t-1}}^{x_{t-1}}(B(x,3r_n)).
\end{align*}

Since this holds for all \(\varepsilon > 0\) it follows that \(\overline{\delta^{t-1}} \le \delta + \frac{\kappa}{\chi} = \overline{\delta^t} +  \gamma_{T^t,T^{t-1}} \) as claimed.

\section{Application to Hitchin representations of compact surface groups}\label{example}

In this section, we first illustrate by an example the  discussion of section 1.2. Take \( \mu \) discrete in  \(\MM( SL_3(\R) )\).  Consider the  nine arrows of Figure 1  describing projections from \( \X^{f'}_T \) to \( \X^{f'}_{T'} \), where \(T \overset{1}{\prec} T'\), in dimension 3. By  Theorem \ref{mainT}, six of these projections are dimension conserving.   There is a natural  family of examples of random walks on \( SL_3(\R) \) for which two of the other projections are not dimension conserving as soon as the middle exponent \( \chi _2 \) is not $0$. Namely, these are the random walks on images of a surface group by a Hitchin representation in \( SL_3(\R) .\) We present these examples and then extend the discussion to Hitchin representations in \( PSL_d(\R) \), for all \( d\geq 3.\)

\subsection{Hitchin component in dimension 3}\label{Hitchin}

Consider a closed surface \( \Sigma \) of genus at least two and the group \( \G := \pi _1 (\Sigma ).\) A representation \(  \rho : \G \to  PSL_2(\R) \) is called {\it {Fuchsian}} if it  is discrete and cocompact. A representation \(  \rho : \G \to  SL_3(\R) \) is also called {\it {Fuchsian}} if it is the composition of a Fuchsian representation and a canonical irreducible representation of \( PSL_2(\R) \) into \( SL _3(\R) \).  It is called {\it {Hitchin}} if it can be obtained by a deformation of a Fuchsian representation. 

Hitchin representations have been studied from many points of view, we only list the properties we are going to use.
 We shall describe points in \( \F\) as pairs \( (\zeta, \ov \zeta )\), where \( \zeta \) in a point in the projective plane \( \R\P^2\) and \( \ov \zeta  \) a line in \( \R\P^2 \) containing \( \zeta.\) The pairs \( (\zeta, \ov\zeta ), (\eta, \ov \eta ) \) are in general position if, and only if, \( \zeta \not \in \ov \eta, \eta \not \in \ov \zeta .\) By classical results of Koszul \cite{koszul},  Goldman \cite{goldman} and Choi-Goldman \cite{choigoldman}, if the representation \( \rho \) is Hitchin, then there exists a \( C^1\)  convex subset \(\Delta \subset  \R\P^2\) invariant under \(\rho (\G)\) and a H\"older continuous mapping \(( \xi ,\ov \xi ): \S^1 \to\F \) such that \\
-- for \( s \neq t \in \S^1, \; (\xi, \ov \xi ) (s) \) and \( (\xi, \ov\xi )(t)\) are in general position,\\
-- \( \xi (\S^1) \) is the boundary \( \partial \Delta \),\\
-- for \(t \in \S^1, \, \) \( \ov \xi (t) \) is the tangent direction to \( \partial \Delta \) at \( \xi (t)\) and\\
-- the set \( \La :=  (\xi, \ov \xi') (\S^1) \) is an invariant set for the action of \( \rho (\G) \) on \( \F\). The set \(\La \) consists in the tangent elements to \( \partial \Delta \).

So, the convex \( \Delta \) admits a cocompact group of projective mappings, i.e. it is  {\it{divisible.}} A classical result of Benz\'ecri is that the boundary is of class \(C^2\) if, and only if, \(\Delta \) is an ellipse, if, and only if, the representation is conjugated to a Fuchsian representation. In that case, \( \chi_2 = 0\) and all the dimension questions reduce to the \(PSL_2(\R) \) case. Therefore, we may assume that  representation \( \rho \) is Hitchin but not a Fuchsian representation. Such representations were studied in detail by Y. Benoist. In particular, he showed   that\\
-- the boundary \( \partial \Delta \)  is \( C^{1+ \beta} \) for some \( \beta >0 \), but not \( C^{1+ abs.cont.} \)(\cite{benoist04}),\\
-- the group \( \rho (\G) \) is Zariski dense in \(SL_3(\R) \)  (\cite{benoist}).\\

We denote \( \MM (\G)\) the set of probability measures \( \mu \) on \(\G\) such that the group generated by the support of \( \mu \) is \( \G\) and \( \sum _ \g |\g| \mu (\g) < + \infty ,\) where \( |\cdot | \) is some word metric on \(\G\). Let \(\mu \in \MM(\G) \); consider the random walk \( (\rho (\G), \rho_\ast (\mu)) \) and the stationary measures \( \nu , \nu '\) on \( \F\). The measures \( \nu \) and \( \nu '\) are supported on \( \La .\)  For \(m \)-a.e. \( \om \in \Om ,\) there are \( (\xi_+, \ov \xi_+)(\om)\) and \((\xi _-, \ov\xi_-) (\om)  \) distinct points in \( \La\) that are  the supports  of the limit measures of \( \left(g_{-1}(\om)  \ldots  g_{-n}(\om)\right)_\ast \nu\) and, respectively,   \( \left(g_{0}(\om)^{-1}  \ldots g_{n-1}(\om)^{-1}\right)_\ast \nu ' \) as \( n \to +\infty \). The distribution of  \((\xi_+, \ov \xi_+)(\om) \)  is \( \nu \), the distribution of  \( (\xi_-, \ov \xi_-)(\om) \) is  \( \nu' \).  The point \( \xi _+(\om )  \) is  the  direction of the  expanding \( E_1(\om) \), the point \( \xi _-(\om ) \) is  the  direction of the  contracting  \( E_3(\om) \). The central direction \( E_2(\om) \) is obtained as \( \ov \xi_+(\om ) \cap \ov\xi _-(\om ) .\)

\begin{proposition}\label{noconsdim} Let \( \rho \) be a Hitchin representation of \( \G\) and  \( \mu \in \MM(\G).\)  Consider the random walk on \( SL_3(\R)\) directed by the probability \( \rho_\ast (\mu) \), \( \chi _1 > \chi _2 > \chi _3 \) its Lyapunov exponents. Let \( \F, \LL , \PP \) be the spaces of flags, lines and planes in \(\R^3\),   \( \nu,  \nu _\LL , \nu _\PP \) the respective stationary measure and \( \de, \de_\LL, \de _\PP \) their dimensions. \\
Assume \( \chi _2 >0 \). Then, \( \de_\PP < \de_\LL = \de\). Moreover, the projections \( \nu \rightarrow \nu _\PP\) are not dimension conserving.\\
Assume \( \chi _2 <0 \). Then, \( \de _\LL < \de_\PP =\de \). Moreover, the projections \( \nu \rightarrow \nu _\LL\) are not dimension conserving.\\
Assume \(\chi _2 =0 .\) Then, \( \de = \de _\LL  = \de _\PP .\) All the  natural projections are dimension conserving. \end{proposition}
\begin{proof}
 By the above discussion, our results apply in this setting. We can consider\\
-- The distribution \( \nu \) of \(f = (\xi _+, \ov\xi'_+ )(\om ) \in \F\). It has entropy \( h \) and dimension \( \de .\)\\
-- The distribution \( \nu _\LL \) of \( \xi _+(\om) \).  It has entropy \( h _\LL\) and dimension \( \de _\LL.\) Observe that, once one knows \( \xi (t) \in \partial \Delta , \, \ov\xi (t) \) is the tangent direction to \( \partial \Delta \) at \( \xi (t) \), so it is uniquely determined. In other words, the projection from \( \nu \) to \( \nu _\LL \) is a.e. one-to-one. By \cite{lessa} \( h_\LL = h \), but, a priori, there is no dimension conservation  and we only get \( \de \geq \de _\LL .\)\\
-- The distribution \( \nu _\PP \) of \( \ov\xi _+(\om  )\).  It has entropy \( h _\PP\) and dimension \( \de _\PP.\) Observe that, similarly, once one knows \( \ov\xi (t)\) is a tangent direction to \( \partial \Delta \) at some point, then this point is \( \xi (t) \), so it is uniquely determined. In other words, the projection from \( \nu \) to \( \nu _\PP \) is a.e. one-to-one. By \cite{lessa} again,  \( h_\PP = h \), but, a priori, there is no dimension conservation  and we only get \( \de \geq \de _\PP .\)

We choose \( f' = (\eta,\ov \eta)  \in \F \). Write \(\LL\PP  := \X^{f'}_{\{1,3\}, \{2\}, \{3\}}, \LL'  := \X^{f'}_{\{1,3\}, \{2,3\}, \{3\}}\) and \( \PP'  := \X^{f'}_{\{1,2,3\}, \{2\}, \{3\}} .\)  For \( \nu ' \) -a.e. \( f' \in \F\), write \( \nu _{\LL\PP}^{f'}, \nu _{\LL'}^{f'} , \nu _{\PP'} ^{f'}\) for the corresponding conditional measures.\\  We  can also consider
the distribution \( \nu _{\LL\PP}^{f'} \) of the couple made of the point \(\ov \xi _+ \cap \ov \eta \) and the line \( (\eta, \xi _+) \) given \( f' = (\eta, \ov\eta).\)  It has entropy \( h _{\LL\PP}\) and dimension \( \de _{\LL\PP}.\) Again, this determines \( ( \xi, \ov\xi) \) by intersection with \( \partial \Delta \) and the dimension on fibers of the projection from \( \F \) to \( \X ^{f'}_{\LL\PP} \) is 0. But now, by theorem~\ref{mainT}, there is  dimension conservation, so \( h = h_{\LL\PP} \) and \( \de = \de _{\LL \PP}.\)

Assume \( \chi _2 \geq 0.\) We project both \( \nu _\LL \) and \( \nu _{\LL\PP}^{f'} \) to the space \( \LL'\) of lines going through \( \eta \) by associating the line going through \( \xi \) and \(\eta \) in the first case and by forgetting \( \ov \xi \cap \ov \eta \) in the second case. The projection and the image measure \( \nu _{\LL'}^{f'} \) depend on \( f'\). For \( \nu '\)-a.e. \( f'\), we have entropy \(h_{\LL'}\) and dimension \(\de _{\LL'}\) on \( \LL'.\) Both projections have almost everywhere trivial fibers: intersecting the line \( (\eta, \xi )\) with \( \partial \Delta \) determines everything. Therefore, \[ h = h_\LL = h_{\LL\PP} = h_{\LL'}.\] Moreover, by theorem \ref{mainT}, both projections have dimension conservation  (observe that \( \chi _{\LL,\LL'} = \chi _1 - \chi _3 \geq \chi _{\LL'} = \chi _1 - \chi _2\)).  So we obtain
\[ \de = \de_\LL = \de_{\LL\PP} = \de _{\LL'} = \frac{h}{\chi _1 -\chi _2} .\]

Remain to understand the projections of both \( \nu _\PP \) and \( \nu _{\LL\PP}^{f'}\) on the space \( \PP'\) of points of \( \ov \eta\). The projection and the image measure \( \nu _{\PP'}^{f'} \) depend on \( f'\). For \( \nu '\)-a.e. \( f'\), write \(h_{\PP'} \) and \(\de _{\PP'} \) for the entropy and the dimension of  \( \nu _{\PP'}^{f'} \) . Once more, knowing the  point \(E_2 \) in \( \ov \eta\) determines the rest by drawing the unique other tangent to \( \partial \Delta \) going through \(E_2\). So, \[  h_\PP  =h_{\PP'}  =  h_{\LL\PP } \] (both \( h_\PP \) and \( h_{\LL\PP } \) are \(h\) by the  above discussion) and all the entropies are the same \(h\). Moreover,  since  \( \chi _{\PP,\PP'} = \chi _1 - \chi _3 \geq \chi _{\PP'} = \chi _2 - \chi _3,\) \[ \de _\PP = \de _{\PP'}  = \frac{h}{\chi_2 - \chi_3}. \] 

If \( \chi _2 = 0 \), then \( \chi_1 - \chi _2 = \chi _2 - \chi _3 , \de _{\LL'} = \de _{\PP'} ,\)  all the dimensions coincide and there is  dimension conservation  at all the projections of Figure 1.

 If \( \chi _2 > 0, \) then 
\( \chi_2 - \chi_3 > \chi_1 - \chi_2\) and \( \de _{\PP'}  < \de _{\LL'} \). So \( \de _\PP < \de \) and the projection from \( \nu \) to \(\nu_\PP \) is not dimension conserving. In the same way, \( \de _{\PP'}  < \de _{\LL\PP} \) and, for \( \nu '\)-a.e. \( f'\), the  projection from \( \nu _{\LL\PP}^{f'}\) to \( \nu _{\PP'}^{f'} \) is not dimension conserving either. 
 
In the case when \( \chi_2 \leq 0\), the discussion is the same, exchanging the role of points and lines and of \( \chi_2 - \chi_3 \) and \( \chi_1 - \chi_2\). \end{proof}

\begin{proof}[Proof of theorem \ref{LyaDim3} in dimension \( d = 3\)] By the above proof, we have \( h = \de (\nu ) \min \{ ( \chi_1 - \chi _2), ( \chi _2 - \chi _3 )\}, \) independently of the sign of \(\chi _2 .\) 
 Theorem \ref{LyaDim3} follows when \( d = 3\). \end{proof}

\subsection{Rigidity of Hitchin representations} \label{Gibbs}
In this section, we prove that the hypotheses  of proposition \ref{noconsdim} are satisfied for some probability measure in \( \MM (\rho (\g)) \) if the representation \( \rho \) is Hitchin but not Fuchsian, namely that one can find such a measure with \( \chi _2 \neq 0\). We have the
\begin{theorem}\label{Manhattan} Let \( \rho \) be a Hitchin representation of a cocompact surface group in \( SL_3(\R) \) such that for all probability measures in \( \MM (\rho (\G)), \, \chi _2 \leq 0 .\) Then, the representation \( \rho \) is Fuchsian. \end{theorem}
Such variational characterizations of Fuchsian representations among Hitchin components have been proven by M. Crampon (\cite{crampon09}) and R. Potrie and A. Sambarino (\cite{potriesambarino}) in greater generality. Theorem \ref{Manhattan} is a variant of their results adapted to the dimension 3.
\begin{proof} By proposition \ref{noconsdim}, our hypothesis is that for all \(\mu \in \MM(\G),\) the dimensions \( \de _\LL, \de _\PP \) of the stationary measures on the spaces of lines and planes satisfy \begin{equation}\label{diminequality} \de _\LL \; \leq \; \de _\PP .\end{equation}
We are going to use thermodynamical formalism for the geodesic flow on \( \rho (\G) \setminus H \Delta,\)  where \( H\Delta \) is the homogeneous tangent bundle to \( \Delta\) and a construction of \cite{connellmuchnik} to obtain, for any Hitchin representation \( \rho \),
some \( \mu \in \MM(\rho (\G)) \) such that \( \de _\LL = 1 .\) Since \( \nu _\PP\) is also supported on a \( C^1 \)  circle,  (\ref{diminequality}) implies that \( \de _\PP = 1 \) as well. Using dynamics of the geodesic flow and \cite{benoist04}, section 6, this will imply that  the representation is Fuchsian.

Recall that  all matrices \( \rho (\g), \g \in \G, \rho(\g) \neq Id, \) have three distinct real eigenvalues with absolute values \( e^{\ell _1(\g)} > e^{\ell _2(\g)} > e^{\ell _3(\g)} \) (\cite{labourie06}). 
Let \( \vf \) be the linear functional on \( \Si := \{( \ell_1, \ell_2, \ell_3 ) \in \R^3: \ell _1 + \ell _2+ \ell _3 = 0 \} \) defined by \( \vf := \ell _1 -  \ell _2 .\)
Recall that the geodesic flow on \( \rho (\G) \setminus H\Delta\)  is an Anosov flow. There exists a H\"older continuous function \(f\) on \( H\Delta\) such that 
for any \(\g \in \G, \g \neq Id,\) \[ \ell _1(\g )  - \ell _2 (\g) \; = \; \int _{\s_\g} f, \]  where \( \s_\g \) is the periodic orbit associated to \( \g \) (see \cite{potriesambarino}, sections 2 and 7). Moreover, for any ergodic invariant measure \(m\) for the geodesic flow, \( \int f \, dm \) is the positive Lyapunov exponent of the geodesic flow for \(m\) (\cite{benoist04}, Lemma 6.5). In particular, the equilibrium measure \( m_0\)  for \(- f \) is absolutely continuous along unstable manifolds. 

 Fix a point \( o \in \Delta .\) Then, the {\it {Gibbs-Patterson-Sullivan construction}} (see e.g. \cite{ledrappier94})
  yields an equivariant  family of measures \( \nu _0 \)  at the boundary such that for all \( \g \in \G, \frac{d(\rho (\g))_\ast \nu _0}{d\nu_0} (\xi ) \) is a H\"older continuous function and with the  property that, if a set \(A\) of points in \( \partial \Delta \) is \(\nu_0\)-negligible, then the set of geodesics with end in \(A \) is \(m_0\)-negligible. By the absolute continuity of the stable foliation, this implies that \( \nu _0 \) is absolutely continuous on \( \partial \Delta.\)

 Recall that the limit set \( \La\) of \( \rho (\G)\) projects one-to-one in \( \partial \Delta.\) Denote by \( \nu \) the lift of the measure \( \nu _0 \) to \( \La\).
 
 Next step consists  in  finding a random walk in \( \MM(\rho(\G) )\) such that \( \nu \) is the stationary measure on \(\F\) or equivalently such that \(\nu _0 \) is the stationary measure for the action on \( \partial \Delta \). 
\begin{lemma} Let  \(\G\) be a co-compact group of isometries of \(\H^2\), \( \rho \) a Hitchin non-Fuchsian  representation of \(\G\) in \(SL_3(\R) \), \(\Delta\) the  open convex proper subset of \( \R\P^2 \) invariant under \( \rho (\Gamma ),  \, \nu _0\) be a finite measure  on \( \partial \Delta \) such that for all \( \g \in \G, \frac{d(\rho (\g))_\ast \nu _0}{d\nu_0} (\xi ) \) is a H\"older continuous function. Then there exists a probability measure \( \mu_0 \in \MM(\G) \) such that \( \nu _0\) is \( \rho _\ast (\mu _0)\)-stationary.
\end{lemma}
\begin{proof} We can apply \cite{connellmuchnik}, theorem 1.1, to the action of \( \G\) on the hyperbolic plane with the measure \(\nu _{\S} =  (\xi^{-1})_\ast \nu _0 .\) 
Since the mapping \( \xi  \) is H\"older continuous and \( \G\)-equivariant, the measure \(  \nu  _{\S} \) has H\"older continuous Radon-Nikodym derivatives under the action of \(\G\) as well. Let \( \mu_0 \) be the measure given by \cite{connellmuchnik} Theorem 1.1 and such that \(\nu _{\S} \)  is the stationary measure under \( \mu \). The measure \( \mu \)  has whole support on \( \G\) and   satisfies \( \sum _g \mu_0 (g) d(o, go ) < + \infty  \) (\cite{connellmuchnik}, page 488). It does indeed belong to \( \MM (\G).\) \end{proof}

To summarize, the measure \( \mu :=   \rho _\ast \mu _0\) on \(\rho (\G) \) has the property that \( \nu _\LL \) is absolutely continuous and is the measure at infinity of the SRB measure of the geodesic flow on \( H\Delta .\) Moreover, by (\ref{diminequality}), \( \de _\PP = 1.\)

Consider the dual representation \( \rho ^\ast (\g) = (\rho (\g)^t)^{-1} \)  and the measure \( \mu ^\ast := (\rho ^\ast )_\ast \mu _0 \). 
The exponents of the random walk \( (\G, \mu ^\ast )\)  are the opposite \( -\chi_3 > -\chi _2 > - \chi _1 \)  and we claim that the entropy \( h^\ast \) is the  same entropy \( h^\ast  =  h \). Therefore, the dimension of the stationary measure \(  \nu ^\ast_\LL \) is \( \de _\PP =1.\) 
By the variational principle again, \(  \nu ^\ast_\LL \) is absolutely continuous. By \cite{benoist04} Proposition 6.2, the representation is Fuchsian. 

To prove the claim, observe that, since \(\rho (\G) \) is discrete in \(SL_3(\R),\) the entropy \( h \) is given by the random walk entropy \( h=   h_{{\textrm {RW}}} (\mu ):= \lim\limits_n \frac{1}{n} H(\mu ^{ (n)} ) \) (\cite{ledrappier85}), which is  the same for \( \mu \) and \( \mu ^\ast .\)\end{proof}

Using the dual representation \( \rho ^\ast \), we also have 
\begin{corollary}\label{Manhattan*}  Let \( \rho \) be a Hitchin representation of a cocompact surface group in \( SL_3(\R) \) such that for all probability measures in \( \MM (\rho (\G)), \, \chi _2 \geq 0 .\) Then, the representation \( \rho \) is Fuchsian. \end{corollary}

\subsection{Hitchin components in higher dimensions}

Consider  the surface  group \( \G .\) As before, a representation \(  \rho : \G \to  PSL_d(\R) \) is  called {\it {Fuchsian}} if it is the composition of a Fuchsian representation and the canonical irreducible representation of \( PSL_2(\R) \) into \( PSL _d(\R) \).  It is called {\it {Hitchin}} if it can be obtained by a deformation of a Fuchsian representation. Geometric properties of Hitchin representations have been studied, notably by F. Labourie (see \cite{labourie06}, \cite{labourie08} for history, background, the properties we use below and much more). Let \( \rho :\g \to PSL_d(\R)  \) be a Hitchin representation and denote again by \( \rho(\G) \) a lift of the representation to \( SL_d(\R) \).  By \cite{guichard08}, proposition 14, the action of \( \rho (\G)\) on \( \R^d\) is strongly irreducible: there is no  finite union of proper vector subspaces of \( \R^d\) that is invariant under \(\rho (\G).\)  By \cite{labourie06}, theorem 1.5, the matrix \( \rho (\g ) \), for \( \g \) non-trivial has all eigenvalues real and distinct. In particular, for \( \g \)  non trivial, there is a unique attracting fixed point \( \g^+ \) for the action of \(\rho( \g )\) on \( \F\). By definition, the {\it {limit set }} \(\La\) is the closure of the set of all \( \g^+, \g \neq Id \in \G.\) Moreover, the projection from the  limit set \( \La \) to \( \R \P^{d-1} \) is one-to-one (\cite{labourie06} theorem 4.1).

 Let \( \mu \in \MM(\G) \). We claim that \( \mu \in \MM (d) \):  on the one hand, the action is proximal on all exterior products   and by  \cite{guivarch-raugi}, the  exponents
\( \chi _1 > \ldots >\chi _d \)  are distinct; on the other hand,  since the action  on \(\R^d\) is strongly irreducible, there is   a unique stationary measure \( \nu \) on \( \R\P^{d-1} \). That measure has a unique lift to \( \La\) and therefore, there is a unique stationary measure on \( \F\). For the same reason, there is a unique stationary measure for \(\mu '\) on \( \F\).  All our discussion and theorem \ref{main} apply, the measure \( \nu \) is exact dimensional with dimension \(\de \)  and entropy \( h := h(\F, \mu, \nu ) \). 
\begin{proposition}\label{noconsdim1} Let \( \la := \inf _{i<j} (\chi_i -\chi _j).\) Then \( \de = h/ \la .\) More generally, let \(T \neq T_0\) be an admissible topology, \( \kappa _T, \de _T\) as defined in  (\ref{entropy}) and corollary \ref{exact-dim}. Let  \( \la _T: = \inf _{i<j, j \not \in T(i)} (\chi_i -\chi _j).\) Then, \( \kappa _T= h \) and  \( \de _T= h/ \la  _T.\) \end{proposition} 
In particular, theorem \ref{LyaDim3} follows in all dimensions. Another  consequence is that, as in dimension 3,  comparing  \( \la \) and \( \la _T \) is enough to decide whether the projection from \( \F \) to \( \X_T^{f'} \) is dimension conserving for \( \nu '\)-a.e. \( f'.\)
\begin{corollary} Let \(T \) be an admissible topology such that \( T_1  \overset {1}{\prec} T \). Then \( \de_T = \de \) unless there is a unique \( i \) with \( \la = \chi _i - \chi _{i+1} \) and furthermore \( T = \{1\}, \{2\}, \ldots  , \{ i, i+1\}, \ldots , \{d\} .\) 
\end{corollary}

\begin{proof}[Proof of Proposition \ref{noconsdim1}] The measure \(\nu \) is supported by the limit set \( \La \subset \F.\) Labourie showed that \( \La \) is a {\it { hyperconvex Frenet }} curve with {\it {Property (H)}}. Namely:
\begin{enumerate}
\item There is a H\"older  continuous \(\G\)-equivariant  mapping \( \xi : \S^1 \to \La ,\) \[ \xi (t) = \{0\} \subset \xi _1 (t) \subset \ldots \subset \xi _i (t) \subset \ldots \subset \xi _d (t) = \R^d . \]
\item For any distinct points \( t_1 \ldots , t _\ell \) integers \( d_1, \ldots, d_\ell \) with \( p := \sum _{j=1}^\ell d_j ,\) the following sum is direct
\[ \xi _{d_1, \ldots ,  d_\ell }(t_1, \ldots, t_\ell ) \; : = \;  \xi _{d_1} (t_1) \oplus \ldots \oplus \xi _{d_\ell} (t_\ell) \]
and, if  the distinct \( t_1 \ldots , t _\ell \)  all converge to \( x\), then \( \xi _{d_1, \ldots ,  d_\ell }(t_1, \ldots, t_\ell ) \) converge to \( \xi _p (x) \).
\item for any triple of distinct points \( (s,t,t')\), any integer \(i, 0<  i < d\),
 \[ \xi _i (s) \oplus (\xi _i (t) \cap \xi _{d-i+1} (t') )\oplus \xi _{d-i -1} (t') \; = \; \R^d.\]
\end{enumerate}
Property (2) defines a hyperconvex Frenet curve (\cite{labourie06}, theorem 1.4) and property (3) is relation (6) in \cite{labourie06} theorem 4.1. (Property (3) is called Property (H) in \cite{labourie06} section 7.1.4.)

Let \(T^1\) be an admissible topology such that  \( T^1 \overset {1}{\prec} T_0 \). We claim that there is a unique integer \( i , 0< i< d,\) such that \[ T^1(k) = \{ k, k+1, \ldots, d \} {\textrm { for }} k \neq i, \quad T^1(i) = \{ i, i+2, \ldots , d \} .\]
Indeed, by definition, there is \( i , 0< i< d,\) such that \( T^1(k) = T_0(k) \) for \( k \neq i \) and  \( j >i \) such that \(T^1(i) = T_0(i) \setminus \{j\}\) . By Proposition \ref{onestepproposition}, \( T^1(i) \setminus \{i,j \} \in T_0, \) and this is possible only if \( j = i+1.\)

Fix \( t' \in \S^1\) and set \(f' := \xi (t') .\) By lemma \ref{coordinate}, the set \( \X_{T^1}^{f'} \) is bilipschitz homeomorphic to an open interval.
We associate to \( t \in \S^1, t \neq t',\) a configuration \( \Psi (t ) \in \X_{T^1}^{f'} \) by setting:
\begin{eqnarray*} \Psi (t) _{{T^1}(k )} &=& \xi _{d-k+1} (t') \; {\textrm { for }} \; k \neq i ,\\
\Psi (t) _{{T^1}(i)} &=& \left(\xi _i(t) \cap \xi_{d-i+1}(t')  \right) \oplus \xi _{d-i-1 }(t') .\end{eqnarray*} 
\begin{lemma}\label{convexity} The mapping \( \Psi \) is an homeomorphism between \( \S^1 \setminus \{ t'\} \) and its image in \( \X_{T^1}^{f'}.\) \end{lemma}
\begin{proof} By the  hyperconvexity property (2), dim\(\left(\xi _i(t) \cap \xi_{d-i+1}(t')  \right) =1\) and that space is in general position with respect to \( \xi_{d-i-1} (t') \); therefore   the  mapping \( \Psi \) is continuous. Since \( \Psi \)  is a mapping between two open intervals, it suffices to show that \( \Psi \) is one-to-one. Assume by contradiction that there is \( s\neq t,t' \) such that 
\[  \left(\xi _i(s) \cap \xi_{d-i+1}(t')  \right) \oplus \xi _{d-i -1}(t')  =  \left(\xi _i(t) \cap \xi_{d-i+1}(t')  \right) \oplus \xi _{d-i -1}(t') .\]
By property (3), \( \xi _i (s) \) should be in direct sum with \( \left(\xi _i(s) \cap \xi_{d-i+1}(t')  \right) \oplus \xi _{d-i -1}(t') \), which is possible only if \( \xi _i(s) \cap \xi_{d-i+1}(t') = \{0\} . \) This contradicts hyperconvexity. \end{proof}

By lemma \ref{convexity}, for any \( x' \in \X_{T^1}^{f'}\), there  at most  one point \( s \in \S^1 \) such that \( \pi _{T_1, {T^1}} ( \xi (s) ) = x', \) i.e. \( (\pi _{T_1, {T^1}})^{-1} (x') \) is at most one point. So, \( \kappa _{T_1, {T^1} } = 0 .\) By theorem \ref{exactfibers}, we have \( \de _{T^1} =h /\chi _{{T^1}, T_0} .\)

 For a general admissible topology \( T,\) we apply proposition \ref{order} and obtain a topology \( T^1 \) such that \( T\prec T^1 \overset {1}{\prec} T_0 \) and \( \chi _{T^1, T_0} = \la _T. \) Since \( T_1 \prec T \prec T^1\), \( \kappa _{T, T^1} = 0\) and  \( \kappa _T = h \). By (\ref{LY}) and corollary \ref{exact-dim}, \( \de _T \) is the same as \( \de _{T^1} = h / \la _T.\)
\end{proof}


\begin{thebibliography}{KLP11}
\bibitem[Be00]{benoist} 
Yves Benoist.
\newblock Automorphismes des c\^ones convexes 
\newblock {\em Invent. Math.}, 141(1): 149--193, 2000.

\bibitem[Be04]{benoist04} 
Yves Benoist.
\newblock Convex divisibles, I
\newblock {\em Algebraic groups and arithmetic, Tata Inst. Fund. Res. Stud. Math.} 17: 181--237 (2004).


\bibitem[BHR19]{barany-hochman-rapaport}
Bal\'azs B\'ar\'any, Michael Hochman and Ariel Rapaport.
\newblock Hausdorff dimension of planar self-affine sets and measures 
\newblock{\em Invent. Math.} 216 (3): 601--659, (2019).

\bibitem[BK17]{barany-kaenmaki17}
 Bal\'azs B\'ar\'any and Antti K\"aenm\"aki.
\newblock Ledrappier-Young formula and exact dimensionality of self-affine measures
\newblock {\em Adv. Math.}, 318: 88--129,  (2017).

\bibitem[BPS]{BPS}
 Bal\'azs B\'ar\'any, Mark Pollicott and K\'aroly Simon.
\newblock  Stationary measures for projective transformations
\newblock {\em Jour. Stat. Physics}, 148: 393--421, (2012).


\bibitem[BQ16]{benoist-quint16} 
Yves Benoist and Jean-Fran\c cois  Quint.
\newblock Random walks on reductive groups
\newblock {\em Ergebnisse der Mathematik und ihrer Grenzgebiete}, vol. 62, Springer 2016.

\bibitem[BQ18]{benoist-quint18} 
Yves Benoist and Jean-Fran\c cois  Quint.
\newblock On the regularity of stationary measures
\newblock{\em Israel J. Math.} 226: 1--14 (2018).

\bibitem[B12]{bourgain12}
Jean Bourgain.
\newblock Finitely supported measures on \( SL_2(\R) \) which are absolutely continuous at infinity,
\newblock in {\em Geometric aspects of functional analysis. Lect. Notes in Math.} 2050: 133--141, Springer, Heidelberg (2012).
 
\bibitem[CG93]{choigoldman}
Suhyoung Choi and William M. Goldman.
\newblock Convex real projective structures on surfaces are closed
\newblock {\em Proc. Amer. Mat. Soc.} 118: 657--661 (1993).

\bibitem[CM07]{connellmuchnik}
Chris Connell and Adam Muchnik.
\newblock Harmonicity of Gibbs measures
\newblock {\em Duke Math. J.} 137: 461--509 (2007).

\bibitem[Cr09]{crampon09}
Micka\"el Crampon.
\newblock Entropies of compact strictly convex projective manifolds
\newblock {\em J. of Mod. Dyn.} 3: 511--547 (2009).


\bibitem[D80]{derriennic80}
Yves Derriennic.
\newblock Quelques applications du th\'eor\`eme ergodique sous-additif
\newblock {\em Conference on random walks (Kleebach, 1979),} {\em Ast\'erisque}, 74: 283--301 (1980).

\bibitem[DO80]{douadyoesterle}
Adrien Douady and Joseph Oesterl\'e.
\newblock Dimension de Hausdorff des attracteurs
\newblock {\em Comptes Rendus Acad. Sci.}, 290: 1136--1138 (1980).


\bibitem[Dob59]{dobrushin}
  Roland L.  Dobru\v{s}in.
 \newblock A general formulation of the fundamental theorem of {S}hannon
              in the theory of information
\newblock {\em Uspehi Mat. Nauk}, 14: 3--104, (1959).

\bibitem[F88]{falconer}
Kenneth J. Falconer.
\newblock{The Hausdorff dimension of self-affine fractals}
\newblock{\em Math. Proc.Cambridge Phil. Soc.} 103(2): 339--350 (1988).

  
  
\bibitem[FFJ15]{falconer-fraser-jin2015} 
Kenneth Falconer, Jonathan  Fraser and Xiong Jin. 
\newblock Sixty years of Fractal Projections
\newblock {\em Fractal Geometry and Stochastics V; C. Bandt, K. Falconer and M. Z\"ahle eds}, Progress in Probability, 70, Birkh\"auser, Cham (2015), 3--25.

\bibitem[Fen19]{feng}
De-Jun Feng.
\newblock Dimension of invariant measures for affine iterated function systems,
 ArXiv 1901.01691.
  
\bibitem[FH09]{feng-hu}
De-Jun Feng and Huyi Hu.
\newblock Dimension theory of iterated function systems
\newblock {\em Comm. Pure Appl. Math.}, 62: 1435--1500,  (2009).


\bibitem[F63]{furstenberg1963}
Harry Furstenberg.
\newblock Noncommuting random products
\newblock {\em Trans. Amer. Math. Soc.}, 108: 377--428  (1963).

\bibitem[F70]{furstenberg1970}
Harry Furstenberg.
\newblock Intersections of Cantor sets and transversality of semi-groups
\newblock {\em Problems in Analysis (Symp. Solomon Bochner, Princeton University, Princeton N.J. 1969)} Princeton University Press, Princeton, N.J. (1970), 41--59.

\bibitem[F71]{furstenberg1971}
Harry Furstenberg.
\newblock Random walks and discrete subgroups of Lie groups.
\newblock {\em Advances in Probability and Related Topics,} 1: 1--63,  Dekker, New York (1971).

\bibitem[F08]{furstenberg2008}
Harry Furstenberg.
\newblock Ergodic fractal measures and dimension conservation
\newblock {\em Ergod. Th. \& Dynam. Sys.}, 28: 405--422  (2008).


\bibitem[GfY]{gelfand-yaglom}
Israel M. Gel\cprime fand,  and Akiva M. Yaglom.
\newblock Calculation of the amount of information about a random
              function contained in another such function,
\newblock {\em Amer. Math. Soc. Transl. (2)}, 12: 199--246  (1959).

     

\bibitem[GM89]{goldsheid-margulis} 
Ilya Ya. Gol'dshe\u{\i}d and Gregory A. Margulis.
\newblock Lyapunov exponents of a product of random matrices 
\newblock {\em Russian Math Surveys}, 44: 11--71 (1989).



\bibitem[Gol90]{goldman} 
William M. Goldman. 
\newblock Convex real projective structures on compact surfaces. 
\newblock {\em J. Diff. Geom.}, 31: 791--845, 1990

\bibitem[G08]{guichard08}
Olivier Guichard
\newblock Composantes de Hitchin et repr\'esentations hyperconvexes de surfaces
\newblock {\em J. Differential  Geom.}, 80: 391--431 (2008).

\bibitem[G90]{guivarch1990}
Yves Guivarc'h.
\newblock Produits de matrices al\'eatoires et applications
\newblock {\em Ergod. Th. Dynam. Sys.}, 10: 483--512 (1990).

\bibitem[GL93]{guivarchlejan}
Yves Guivarc'h and Yves Le~Jan.
\newblock  Asymptotic winding of the geodesic flow on modular
surfaces and continued fractions
\newblock {\em Ann. Sci. ENS}, 26: 23--50, (1993).

\bibitem[GR89]{guivarch-raugi}
Yves Guivarc'h and Albert Raugi.
\newblock Propri\'{e}t\'{e}s de contraction d'un semi-groupe de matrices
              inversibles. {C}oefficients de {L}iapunoff d'un produit de
              matrices al\'{e}atoires ind\'{e}pendantes
\newblock{\em Israel J. Math.}, 65: 165--196 (1989).             
              
\bibitem[H14]{hochman14}
Michael Hochman.
\newblock On self-similar sets with overlaps and inverse theorems for entropy 
\newblock{\em Ann. Math.} (2) 180(2): 773--822, (2014).

\bibitem[H15]{hochman15} 
Michael Hochman.
\newblock On self-similar sets with overlaps and inverse theorems for entropy in \( \R ^d\)
\newblock {\em to appear (Memoirs of the A.M.S.)} ArXiv 1503.09043.

\bibitem[HR19]{hochman-rapaport}
Michael Hochman and Ariel Rapaport.
\newblock Hausdorff dimension of planar self-affine sets and measures with overlaps
\newblock{\em preprint} ArXiv 1904. 09812.


\bibitem[HS17]{hochman-solomyak2017}
Michael Hochman and Boris Solomyak.
\newblock On the dimension of {F}urstenberg measure for {$SL_2(\Bbb R)$} random
  matrix products
\newblock {\em Invent. Math.}, 210(3): 815--875, (2017).


\bibitem[JJL04]{jarvenpaa2-llorante}
Esa J\"arvenp\"a\"a, Maarit J\"arvenp\"a\"a and Marta Llorante.
\newblock Local dimensions of sliced measures and stability of packing dimensions of sections of sets
\newblock {\em Adv. maths.}, 183: 127--154  (2004).

\bibitem[JM98]{jarvenpaa-mattila}
Maarit J\"arvenp\"a\"a and Pertti Mattila.
\newblock Hausdorff and packing dimensions and sections of measures
\newblock {\em Mathematika}, 45: 55--77   (1998).


\bibitem[KLP11]{kaimanovichLePrince}
Vadim~A. Kaimanovich and Vincent Le~Prince.
\newblock Matrix random products with singular harmonic measure.
\newblock {\em Geom. Dedicata}, 150:257--279, (2011).


\bibitem[Kos68]{koszul}
Jean-Louis Koszul.
\newblock D\'eformations de connexions localement plates
\newblock {\em Ann. Inst. Fourier (Grenoble)},18: 103--114, (1968).


\bibitem[KV83]{KV83}
Vadim~A. Kaimanovich and Anatoly~M. Vershik.
\newblock Random walks and discrete groups: boundary and entropy
\newblock {\em Ann. Probab.}, 11: 457--490, (1983).

\bibitem[KY79]{kaplanyorke}
James L. Kaplan and James A. Yorke.
\newblock Chaotic behavior of Multidimensional Difference Equation
\newblock {\em Springer Lect. Notes}, 730: 204--277 (1979).

\bibitem[Lab06]{labourie06}
Fran\c cois Labourie.
\newblock Anosov Flows, Surface Groups and Curves in Projective Space
\newblock {\em Invent. Math.} 165: 51--114 (2006).

\bibitem[Lab07]{labourie08}
Fran\c cois Labourie.
\newblock Cross ratios, surface groups, \(PSL(n,\R) \) and diffeomorphisms of the circle
\newblock {\em  Publ. Math. Inst. Hautes \'Etudes Sci.} 106: 139–213 (2007).



\bibitem[Led84]{ledrappier}
Fran\c cois Ledrappier.
\newblock Quelques propri\'{e}t\'{e}s des exposants caract\'{e}ristiques
\newblock In {\em \'{E}cole d'\'{e}t\'{e} de probabilit\'{e}s de
  {S}aint-{F}lour, {XII}---1982}, volume 1097 of {\em Lecture Notes in Math.},
  305--396. Springer, Berlin, (1984).

\bibitem[Led85]{ledrappier85}
Fran\c cois Ledrappier.  
  \newblock Poisson boundaries of discrete groups of matrices 
  \newblock {\em Israel J. Math.}, 50: 319--336 (1985).

  

\bibitem[Led94]{ledrappier94}
Fran\c cois Ledrappier.
\newblock Structure au bord des vari\'et\'es \`a courbure n\'egative
\newblock{\em S\'eminaire de th\'eorie spectrale et g\'eom\'etrie de Grenoble} 71: 97--122, (1994/1995)

\bibitem[Les21]{lessa}
   Pablo Lessa.
   \newblock Entropy and dimension of disintegrations of stationary measures
   \newblock {\em Trans. Amer. Math. Soc. Ser. B}, 8: 105--129 (2021).


\bibitem[LY85]{ledrappier-young1985}
Fran\c cois Ledrappier and Lai-Sang Young.
\newblock The metric entropy of diffeomorphisms. {II}. {R}elations between
  entropy, exponents and dimension
\newblock {\em Ann. of Math. (2)}, 122(3): 540--574, (1985).



\bibitem[O68]{oseledets}
Valery I. Oseledets.
\newblock A multiplicative ergodic theorem. {C}haracteristic {L}japunov
              exponents of dynamical systems
\newblock {\em Trudy Moskov. Mat. Obsc}   19:179--210 (1968).   


\bibitem[Per59] {perez}
Albert Perez.
\newblock Information theory with an abstract alphabet. {G}eneralized
              forms of {M}c{M}illan's limit theorem for the case of discrete
              and continuous times
\newblock{\em Theor. Probability Appl.}, 4: 99--102  (1959).


\bibitem[PS17]{potriesambarino}
Rafael Potrie and Andr\'es Sambarino.
\newblock  Eigenvalues and entropy of a Hitchin representation 
\newblock {\em Invent. Math.},  209: 885--925 (2017).



\bibitem[Rap17]{rapaport17}
Ariel Rapaport.
\newblock A self-similar measure with dense rotations, singular projections and discrete slices
\newblock {\em Adv. Math.} 321: 529--546 (2017). 
            

\bibitem[Rap21]{rapaport}
Ariel Rapaport.
\newblock Exact dimensionality and Ledrappier-Young formula for the Furstenberg measure
\newblock  {\em {T.A.M.S.}} 374: 5225--5268 (2021).




\bibitem[Sha48] {shannon}
Claude E. Shannon.
\newblock A mathematical theory of communication
\newblock {\em Bell System Tech. J.}, 27: 379--423, 623--656   (1948).
  
\bibitem[She18]{shen18}
Weixiao Shen.
\newblock Hausdorff dimension of the graphs of the classical Weierstrass functions
\newblock {\em Math.Z.}, 289: 223--266   (2018).


\bibitem[S15]{shmerkin2015}
Pablo Shmerkin.
\newblock Projections of self-similar and related fractals: a survey of recent developments
\newblock {\em Fractal Geometry and Stochastics V; C. Bandt, K. Falconer and M. Z\"ahle eds}, Progress in Probability, 70, Birkh\"auser, Cham (2015), 33--74.

\bibitem[Y82]{young82}
Lai-Sang Young.
\newblock Dimension, Entropy and Lyapunov exponents
\newblock {\em  Ergodic Theory Dynam. Systems}. 2: 109--124  (1982).
\end{thebibliography}
\end{document}